%% file: Project_2_Masterfile.tex
\documentclass[a4paper]{article}
\usepackage{amsmath,amsthm,amssymb,amsfonts}
\usepackage{a4wide} 
\allowdisplaybreaks[3]
\newcommand{\interior}[1]{\overset{\smash{\raisebox{-0.05ex}{$\scriptstyle\circ$}}}{#1}\rule{0pt}{2.ex}}
\usepackage{graphicx}
\usepackage[title]{appendix}
\newcommand\numberthis{\addtocounter{equation}{1}\tag{\theequation}}

\graphicspath{{./Plots/}} 

\usepackage{subfigure}

\newtheorem{theorem}{Theorem}

\newtheorem{lemma}[theorem]{Lemma}

\def \ln {\text{ln}}

\title{High-order compact schemes for parabolic problems with mixed derivatives in multiple space dimensions}

\author{Bertram D{\"u}ring\thanks{Email:~b.during@sussex.ac.uk, Department of Mathematics, University of Sussex, Pevensey II, Brighton, BN1 9QH, United Kingdom}
\and Christof Heuer\thanks{Email:~cheuer@uni-wuppertal.de, Lehrstuhl f{\"u}r Angewandte Mathematik und Numerische Analysis, Fachbereich C, Bergische Universit{\"a}t Wuppertal, Gau{\ss}str. 20, 42119 Wuppertal, Germany}}

\begin{document}
\maketitle

\begin{abstract}
\noindent  We present a
high-order compact finite difference approach
for a class of parabolic partial
differential equations with time and space dependent coefficients as
well as
with mixed second-order derivative terms in $n$ spatial
dimensions. Problems of this type arise frequently in computational
fluid dynamics and computational finance.
We derive general conditions
on the coefficients which allow us to obtain a high-order compact scheme
which is fourth-order accurate in space and second-order accurate in time.
Moreover, we perform a thorough von Neumann stability analysis 
of the
Cauchy problem in two and three
spatial dimensions for vanishing mixed derivative terms, and also give
partial results for the general case.  The results suggest unconditional stability of the scheme.
As an application example we consider the pricing of European Power
Put Options in the multidimensional Black-Scholes model for two and
three underlying assets. Due to the low regularity of typical initial
conditions we employ the smoothing operators of Kreiss et al.\
to ensure high-order convergence of the approximations of the smoothed problem to the true solution.
\end{abstract}


\input{intro}
\input{Project_n_dim_pde_with_Sections_for_BLS_Coefficients}
\bibliographystyle{siam} 
\bibliography{Project_2_Masterfile}
\begin{appendices}
\input{Appendix_3_dim_general_coefficients}

\end{appendices}
\end{document}

%% file: intro.tex
\section{Introduction}

In the last decades, starting from early efforts of
Gupta et al.\ \cite{GuMaSt84,GuMaSt85}, high-order compact finite difference schemes were
proposed for the numerical
approximation of solutions to elliptic \cite{SpoCar96,MR2285863} and
parabolic \cite{SpoCar01,KarZha02}
partial differential equations.
These schemes are able to exploit the smoothness of solutions to such
problems and allow to achieve high-order 
numerical convergence rates (typically strictly larger than two in the spatial discretisation
parameter) while generally having good stability properties.
Compared to finite element approaches the
high-order compact schemes are parsimonious and memory-efficient
to implement and hence prove to be a viable alternative if the complexity of the
computational domain is not an issue.
It would be possible to achieve higher-order approximations also by increasing the
computational stencil but this leads to increased bandwidth of the
discretisation matrices and complicates formulations of boundary
conditions. Moreover, such approaches sometimes suffer from restrictive
stability conditions and spurious numerical oscillations. These
problems do not arise when using a compact stencil.

Although applied successfully to many important applications,
e.g.\ in computational fluid dynamics \cite{SpoCar95,LiTaFo95,LiTa01,FouRig11} and computational
finance \cite{DuFoJu03,DuFoJu04,TaGoBh08,DuFo12,DuFoHe14},
an even wider breakthrough of the high-order compact methodology has
been hampered by the algebraic complexity that is inherent to this
approach.
The derivation of high-order compact schemes is algebraically
demanding, hence these schemes are often taylor-made for a specific
application or a rather smaller class of problems (with some notable
exceptions as, for example Lele's paper \cite{Lele92}).
The algebraic complexity is even higher in the numerical stability analysis of
these schemes. Unlike for standard second-order schemes, the
established stability notions imply formidable algebraic problems for
high-order compact schemes. As a result, there are relatively few stability results for
high-order compact schemes in the literature. This is even more
pronounced in higher spatial dimension, as most of the existing
studies with analytical stability results for high-order compact
schemes are limited to a one-dimensional setting.

Most works focus on the isotropic case where the main part of the
differential operator is given by the Laplacian.
Another layer of complexity is added when the
anisotropic case is considered and mixed second-order derivative
terms are present in the operator. Few works on high-order compact
schemes address this problem, and either study constant coefficient
problems \cite{FoKa06} or specific equations \cite{DuFo12}.

Consequently, our aim in the present paper is to establish a
high-order compact methodology for a class of parabolic partial
differential equations with time and space dependent coefficients and
mixed second-order derivative terms in arbitrary spatial
dimension. 
We derive general conditions
on the coefficients which allow to obtain a high-order compact scheme
which is fourth-order accurate in space and second-order accurate in time.
Moreover, we perform a 
von Neumann stability analysis of the
Cauchy problem in two and three
spatial dimensions for vanishing mixed derivative terms, and also give
partial results for the general case.
As an application example we consider the pricing of European Power
Put Basket options with two and three underlying assets in the multidimensional Black-Scholes model. The 
partial differential equation
features second-order mixed derivative terms and, as an
additional difficulty, is supplemented by an
initial condition with low regularity. We use the smoothing operators
of Kreiss et al.\ \cite{KrThWi70} to restore high-order convergence.

The rest of this paper is organised as follows. In the next section,
we state the general parabolic partial differential equation in $n$
spatial dimensions and give the central difference
approximation for the associated elliptic problem.  We then derive
auxiliary relations for the higher-order derivatives appearing in the
truncation error of the central difference approximation in
Section~\ref{sec:auxiliary}.
In Section~\ref{sec:conditions} we give conditions on the coefficients
of the partial differential equation under which a high-order compact scheme is obtainable.
Semi-discrete high-order compact schemes in $n=2$ and $n=3$ space
dimensions are derived in
Section~\ref{sec:hocsemi}. Section~\ref{sec:hoctime} discusses the
time discretisation. A thorough von Neumann
stability analysis of the Cauchy problem in $n=2$ and $n=3$ space
dimensions is performed in Section~\ref{sec:stability}.
In Section~\ref{Section_Application} we apply the schemes to option pricing problems for
European Basket Power Put options and report results of our 
numerical experiments in Section~\ref{sec:numexp}.
Section~\ref{sec:Conclusion_n_dim_HOC} concludes.


%% file: Project_n_dim_pde_with_Sections_for_BLS_Coefficients.tex
\section{Parabolic problem and its central difference approximation}

We consider the following parabolic partial differential equation with
mixed derivative terms in $n$ spatial dimensions for $u=u(x_1,\dots, x_n,\tau)$,
\begin{align}\label{basic_general_pde_nD_without_usage_of_f}
 u_{\tau} + \sum_{i=1}^n a_i \frac{\partial^2 u}{\partial x_i^2} + \sum\limits_{\substack{i,j=1\\ i<j}}^n b_{ij} \frac{\partial^2 u}{\partial x_i \partial x_j} + \sum\limits_{i=1}^n c_i \frac{\partial u}{\partial x_i}=&g \quad \text{ in } \Omega \times \Omega_{\tau},
\end{align} 
with initial condition $u_0=u(x_1 , \ldots x_n,0)$ and suitable
boundary conditions, with space- and time-dependent coefficients
$a_i=a_i(x_1 , \ldots x_n, \tau)<0$, $b_{ij}=b_{ij}(x_1 , \ldots x_n,
\tau)$, $c_i=c_i(x_1 , \ldots x_n, \tau)$ and $g=g(x_1 , \ldots x_n,
\tau)$. The spatial domain $\Omega \subset\mathbb{R}^n$  is of
$n$-dimensional rectangular shape with
$\Omega  = \Omega_1  \times \ldots \times \Omega_n $ and $x_i \in
\Omega_i = \bigl[x_{\min}^{(i)} , x_{\max}^{(i)} \bigr]$ with
$x_{\min}^{(i)} < x_{\max}^{(i)}$ for $i \in \{1, \ldots , n\}$. The temporal domain is given by $\Omega_{\tau} = \left]0,\tau_{\max}\right]$ with $\tau_{\max} >0$.
The functions $a(\cdot , \tau)$, $b(\cdot , \tau)$, $c(\cdot , \tau)$ and $g(\cdot , \tau)$ are assumed to be in $C^2(\Omega)$ for any $\tau \in \Omega_{\tau}$, $u(\cdot , \tau)\in C^6(\Omega)$ and $u$ is assumed to be differentiable with respect to $\tau$. Introducing $f:=-u_{\tau} + g$ we can rewrite \eqref{basic_general_pde_nD_without_usage_of_f} as
\begin{align}\label{basic_general_pde_nD}
\sum_{i=1}^n a_i \frac{\partial^2 u}{\partial x_i^2} + \sum\limits_{\substack{i,j=1\\ i<j}}^n b_{ij} \frac{\partial^2 u}{\partial x_i \partial x_j} + \sum\limits_{i=1}^n c_i \frac{\partial u}{\partial x_i}= f.
\end{align}

We start by defining a grid on $\Omega$, 
\begin{align}\label{n_dimensional_Grid_general_general_stepsizes}
G^{(n)} :=& \bigl\{\big (x^{(1)}_{i_1}, x^{(2)}_{i_2} ,\ldots , x^{(n)}_{i_n}\big) \in \Omega \text{ } | \text{ } x^{(k)}_{i_k} = x_{\min}^{(k)} +i_k \Delta x_k, 0\leq i_k \leq N_k, \, k=1, 2,\dots, n \bigr\},
\end{align}
where $\Delta x_k = \big({x_{\max}^{(k)}- x_{\min}^{(k)}}\big)/{N_k}>0$  are the step sizes in the $k$-th direction with $N_k \in \mathbb{N} $ for $k=1,2, \ldots , n$. 
We use $\interior{G}^{(n)}$ for the interior of $G^{(n)}$. On this grid we denote by $U_{i_1, \ldots , i_n}$ the discrete approximation of the continuous solution $u$ at the point $\big(x_{i_1}^{(1)}, x^{(2)}_{i_2} ,\ldots , x_{i_n}^{(n)} \big) \in G^{(n)}$ and time $\tau\in \Omega_{\tau}$.
Using the central difference operator $D_k^c$ and the standard second-order central difference operator $D_{k}^2$   in $x_k$-direction we get
\begin{align*}
\frac{\partial^2 u}{\partial x_k^2}=&D^2_k u
- \frac{(\Delta x_k)^2}{12}\frac{\partial^4 u}{\partial x_k^4} +
\mathcal{O}\left((\Delta x_k)^4\right),\\
\frac{\partial u}{\partial x_k}=&D^c_k u
- \frac{(\Delta x_k)^2}{6}\frac{\partial^3 u}{\partial x_k^3}+ \mathcal{O}\left((\Delta x_k)^4\right),\numberthis \label{consistencyequations_general_pde_nD}\\
\frac{\partial^2 u}{\partial x_k \partial x_p}=&D^c_kD^c_p u
- \frac{(\Delta x_k)^2}{6}\frac{\partial^4 u}{\partial x_k^3 \partial x_p} - \frac{(\Delta x_p)^2}{6}\frac{\partial^4 u}{\partial x_k \partial x_p^3}+ \mathcal{O}\left((\Delta x_k)^4\right)\\
&  + \mathcal{O}\left((\Delta x_k)^2 (\Delta x_p)^2\right) + \mathcal{O}\left((\Delta x_p)^4\right)+ \mathcal{O}\left(\frac{(\Delta x_k)^6}{\Delta x_p}\right),
\end{align*}
for $k,p \in \{1, 2,\ldots , n\}$ and $k \neq p$, evaluated at the grid points $\big(x^{(1)}_{i_1}, x^{(2)}_{i_2} ,\ldots , x^{(n)}_{i_n}\big)\in \interior{G}^{(n)}$. 
Using the approximations \eqref{consistencyequations_general_pde_nD} in \eqref{basic_general_pde_nD} gives
\begin{align}\label{semi_discrete_basic_general_pde_nD}
\begin{split}
f=& \sum\limits_{i=1}^n a_i D^2_i u  + \sum\limits_{\substack{i,j=1\\ i<j}}^n b_{ij} D^c_iD^c_j u + \sum\limits_{i=1}^n c_i D^c_iu - \sum_{i=1}^n \frac{a_i(\Delta x_i)^2}{12}\frac{\partial^4 u}{\partial x_i^4}\\
&  - \sum\limits_{\substack{i,j=1\\ i<j}}^n b_{ij}\left[\frac{(\Delta x_i)^2}{6}\frac{\partial^4 u}{\partial x_i^3 \partial x_j} + \frac{(\Delta x_j)^2}{6}\frac{\partial^4 u}{\partial x_i \partial x_j^3}\right] - \sum\limits_{i=1}^n \frac{c_i (\Delta x_i)^2}{6}\frac{\partial^3 u}{\partial x_i^3} + \varepsilon,
\end{split}
\end{align}
where $\varepsilon \in \mathcal{O}\left(h^4 \right)$ if $\Delta x_i
\in \mathcal{O}\left(h\right)$ for $i=1,2, \ldots , n$ for a step size
$h>0$. If the consistency error is in $\mathcal{O}\left(h^4 \right)$, we
call the scheme high-order. In order to achieve a high-order scheme we
need to find second-order approximations of the derivatives
$\frac{\partial^3 u}{\partial x_i^3 }$, $\frac{\partial^4 u}{\partial
  x_i^4 }$ and $\frac{\partial^4 u}{\partial x_i^3 \partial x_j}$ for
$i,j \in \{1 , \ldots , n\}$ with $i \neq j$. We call the scheme
high-order compact, if we can achieve this using only points from a
compact computational stencil for  $x = \big(x^{(1)}_{i_1}, x^{(2)}_{i_2} ,\ldots , x^{(n)}_{i_n}\big)\in \interior{G}^{(n)} $. We have
\begin{align}\label{Def_compact_stencil_n_D}
\hat{U}\left(x\right) =  \bigl\{ \big(x^{(1)}_{i_1+k_1},x^{(2)}_{i_2+k_2} \ldots , x^{(n)}_{i_n+k_n}\big) \in G^{(n)}\text{ }| \text{ } k_m \in \{-1,0,1\} \text{ for } m=1,2,\ldots , n \bigr\} 
\end{align}
for $x=\big(x^{(1)}_{i_1},x^{(2)}_{i_2} \ldots , x^{(n)}_{i_n}\big) $ as the compact computational stencil and define $U_{i_1, \ldots , i_n}\approx u\big(x^{(1)}_{i_1},x^{(2)}_{i_2} \ldots , x^{(n)}_{i_n}\big)$.

\section{Auxiliary relations for higher derivatives}
\label{sec:auxiliary}
In this section we calculate auxiliary relations for the higher derivatives appearing in \eqref{semi_discrete_basic_general_pde_nD}. These relations for the higher derivatives can be calculated by differentiating \eqref{basic_general_pde_nD}. In doing so no additional error is introduced. Differentiating equation \eqref{basic_general_pde_nD} with respect to $x_k$ and then solving for $\frac{\partial^3 u}{\partial x_k^3}$ leads to
\begin{align*}
\frac{\partial^3 u}{\partial x_k^3} =& -\sum\limits_{\substack{i=1\\i \neq k}}^n  \frac{a_i}{a_k} \frac{\partial^3 u}{\partial x_i^2 \partial x_k} 
- \sum\limits_{\substack{i=1\\i \neq k}}^n \frac{1}{a_k}\frac{\partial a_i}{\partial x_k}\frac{\partial^2 u}{\partial x_i^2}  
- \frac{1}{a_k}\frac{\partial a_k}{\partial x_k}\frac{\partial^2 u}{\partial x_k^2} - \sum\limits_{\substack{i,j=1\\ i<j}}^n  \frac{b_{ij}}{a_k} \frac{\partial^3 u}{\partial x_i \partial x_j\partial x_k} \\
&- \sum\limits_{\substack{i,j=1\\ i<j}}^n\frac{1}{a_k}\frac{\partial b_{ij}}{\partial x_k} \frac{\partial^2 u}{\partial x_i \partial x_j}   -\sum\limits_{i=1}^n \frac{c_i}{a_k} \frac{\partial^2 u}{\partial x_i \partial x_k} - \sum\limits_{i=1}^n\frac{1}{a_k}\frac{\partial c_i}{\partial x_k} \frac{\partial u}{\partial x_i}  + \frac{1}{a_k}\frac{\partial f}{\partial x_k} =: A_{k} \numberthis \label{uxxx_general_pde_nD}
\end{align*}
for $k=1 , \ldots , n$. 
The relation for $A_k$ can be approximated 
with consistency order two on the compact stencil \eqref{Def_compact_stencil_n_D}
using the central difference operator, 
as all derivatives of $u$ in the above equation are only differentiated up to twice in each direction.

Differentiating \eqref{basic_general_pde_nD} twice with respect to
$x_k$, and solving the resulting equation for $\frac{\partial^4
  u}{\partial x_k^4}$, we obtain
\begin{align}
\frac{\partial^4 u}{\partial x_k^4}  = &
-\sum\limits_{\substack{i=1\\i \neq k}}^n
\left[\frac{a_i}{a_k}\frac{\partial^4 u}{\partial x_i^2 \partial
    x_k^2} +  \frac{2}{a_k}\frac{\partial a_i}{\partial
    x_k}\frac{\partial^3 u}{\partial x_i^2 \partial x_k} +
  \frac{1}{a_k} \frac{\partial^2 a_i}{\partial x_k^2} \frac{\partial^2
    u}{\partial x_i^2}\right] - \frac{2}{a_k}\frac{\partial
  a_k}{\partial x_k}\frac{\partial^3 u}{\partial x_k^3} \nonumber\\
&- \frac{1}{a_k}\frac{\partial^2 a_k}{\partial x_k^2} \frac{\partial^2 u}{\partial x_k^2}
 -\! \sum\limits_{\substack{i,j=1\\i<j\\i,j \neq k}}^n \left[
   \frac{b_{ij}}{a_k}\frac{\partial^4 u}{\partial x_i \partial
     x_j \partial x_k^2} + \frac{2}{a_k} \frac{\partial
     b_{ij}}{\partial x_k}\frac{\partial^3 u }{\partial x_i \partial
     x_j \partial x_k} + \frac{1}{a_k}\frac{\partial^2
     b_{ij}}{\partial x_k^2}\frac{\partial^2 u}{\partial x_i \partial
     x_j} \right]
\nonumber\\
& - \sum\limits_{i=1}^{k-1} \frac{b_{ik}}{a_k}\frac{\partial^4
  u}{\partial x_i \partial x_k^3} - \sum\limits_{i=1}^{k-1} \left[
  \frac{2}{a_k} \frac{\partial b_{ik}}{\partial x_k}\frac{\partial^3
    u}{\partial x_i \partial x_k^2} + \frac{1}{a_k}\frac{\partial^2
    b_{ik}}{\partial x_k^2}\frac{\partial^2 u}{\partial x_i \partial
    x_k} \right]
\nonumber\\
& - \sum\limits_{j=k+1}^n \frac{b_{kj}}{a_k}\frac{\partial^4
  u}{\partial x_j \partial x_k^3} - \sum\limits_{j=k+1}^n \left[
  \frac{2}{a_k}\frac{\partial b_{kj}}{\partial x_k}\frac{\partial^3 u
  }{\partial x_j \partial x_k^2} + \frac{1}{a_k} \frac{\partial^2
    b_{kj}}{\partial x_k^2}\frac{\partial^2 u}{\partial x_j \partial
    x_k}\right]\label{uxxxx_general_pde_nD} \\
& - \sum\limits_{i=1}^n \left[ \frac{c_i}{a_k} \frac{\partial^3
    u}{\partial x_i \partial x_k^2} + \frac{2}{a_k}\frac{\partial
    c_i}{\partial x_k}\frac{\partial^2 u}{\partial x_i \partial x_k} +
  \frac{1}{a_k}\frac{\partial^2 c_i}{\partial x_k^2}\frac{\partial
    u}{\partial x_i}\right] + \frac{1}{a_k}\frac{\partial^2
  f}{\partial x_k^2}\nonumber\\
=:& B_{k} - \sum\limits_{i=1}^{k-1} \frac{b_{ik}}{a_k} \frac{\partial^4 u}{\partial x_i \partial x_k^3} - \sum\limits_{j=k+1}^n \frac{b_{kj}}{a_k} \frac{\partial^4 u}{\partial x_j \partial x_k^3} . \nonumber
\end{align}
We can approximate $B_{k}$ with second order consistency on the compact stencil \eqref{Def_compact_stencil_n_D}, when using the central difference operator and the auxiliary relations for $A_k$ in \eqref{uxxx_general_pde_nD} for $k=1,\ldots, n$. Differentiating equation \eqref{basic_general_pde_nD} once with respect to $x_k$ and once with respect to $x_p$  leads to
\begin{align*}
&a_k\frac{\partial^4 u}{\partial x_k^3 \partial x_p} + a_p \frac{\partial^4 u}{\partial x_k \partial x_p^3}  \\
=& -\sum\limits_{\substack{i=1\\i\neq k,p }}^n \left[ a_i \frac{\partial^4 u}{\partial x_i^2 \partial x_k \partial x_p} + \frac{\partial a_i}{\partial x_k}\frac{\partial^3 u}{\partial x_i^2 \partial x_p} +  \frac{\partial a_i}{\partial x_p}\frac{\partial^3 u}{\partial x_i^2 \partial x_k} + \frac{\partial^2 a_i}{\partial x_k \partial x_p}\frac{\partial^2 u}{\partial x_i^2}\right] -\frac{\partial a_p}{\partial x_k}\frac{\partial^3 u}{\partial x_p^3 }\\
&  -  \frac{\partial a_p}{\partial x_p}\frac{\partial^3 u}{\partial x_p^2 \partial x_k} - \frac{\partial^2 a_p}{\partial x_k \partial x_p}\frac{\partial^2 u}{\partial x_p^2}- \frac{\partial a_k}{\partial x_k}\frac{\partial^3 u}{\partial x_k^2 \partial x_p} -  \frac{\partial a_k}{\partial x_p}\frac{\partial^3 u}{\partial x_k^3 } - \frac{\partial^2 a_k}{\partial x_k \partial x_p}\frac{\partial^2 u}{\partial x_k^2}\\
& - \sum\limits_{\substack{i,j=1\\i<j}}^n \left[ b_{ij} \frac{\partial^4 u}{\partial x_i \partial x_j \partial x_k \partial x_p} + \frac{\partial b_{ij}}{\partial x_k}\frac{\partial^3 u}{\partial x_i \partial x_j \partial x_p} + \frac{\partial b_{ij}}{\partial x_p}\frac{\partial^3 u}{\partial x_i \partial x_j \partial x_k} + \frac{\partial^2 b_{ij}}{\partial x_k \partial x_p}\frac{\partial^2 u}{\partial x_i \partial x_j}\right]\\
& - \sum\limits_{i=1}^n \left[c_i \frac{\partial^3 u}{\partial x_i \partial x_k \partial x_p} + \frac{\partial c_i}{\partial x_k}\frac{\partial^2 u}{\partial x_i \partial x_p}  + \frac{\partial c_i}{\partial x_p}\frac{\partial^2 u}{\partial x_i \partial x_k} + \frac{\partial^2 c_i}{\partial x_k \partial x_p}\frac{\partial u}{\partial x_i} \right] + \frac{\partial^2 f}{\partial x_k \partial x_p}=: C_{kp}, 
\end{align*}
where $C_{kp}$ can be approximated on the compact stencil
\eqref{Def_compact_stencil_n_D} using $A_k$ and $A_p$, as defined in
equation \eqref{uxxx_general_pde_nD}, and the central difference
operator for $k,p = 1 , \ldots , n$ with $k\neq p$. This can be
written as
\begin{align}\label{uxxxy_general_pde_nD}
\frac{\partial^4 u}{\partial x_k^3 \partial x_p} = \frac{C_{kp}}{a_k} - \frac{a_p}{a_k}\frac{\partial^4 u}{\partial x_k \partial x_p^3} .
\end{align}

\section{Conditions for obtaining a high-order compact scheme}
\label{sec:conditions}
In this section we derive conditions on the coefficients of the partial differential equation \eqref{basic_general_pde_nD_without_usage_of_f} 
under which it is possible to obtain a high-order compact scheme, 
i.e.\ only using points of the $n$-dimensional compact stencil \eqref{Def_compact_stencil_n_D} for discretisation 
and receiving a fourth-order scheme with $\Delta x_i \in \mathcal{O}\left( h\right)$ for $j=1, \ldots , n$ for a given step size $h>0$. 
Using equations \eqref{uxxx_general_pde_nD} and \eqref{uxxxx_general_pde_nD} and then \eqref{uxxxy_general_pde_nD} in \eqref{semi_discrete_basic_general_pde_nD} leads to
\begin{align}
f  = & \sum\limits_{i=1}^n a_iD_i^2 u  + \sum\limits_{\substack{i,j=1\\ i<j}}^n b_{ij} D^c_iD^c_j u + \sum\limits_{i=1}^n c_i D^c_iu - \sum\limits_{i=1}^n \frac{a_i (\Delta x_i)^2B_i}{12} + \varepsilon \nonumber\\
&
- \sum\limits_{\substack{i,j=1\\ i<j}}^n \frac{b_{ij}(\Delta x_i)^2C_{ij}}{12 a_i} - \sum\limits_{\substack{i,j=1\\ i<j}}^n \frac{b_{ij}}{12}\frac{\partial^4 u}{\partial x_i \partial x_j^3}\left[ (\Delta x_j)^2 - \frac{a_j (\Delta x_i)^2}{a_i} \right] - \sum\limits_{i=1}^n \frac{c_i (\Delta x_i)^2 A_i}{6} , \label{pde_to_discretise_n_D}
\end{align}
where $\varepsilon \in \mathcal{O}\left(h^4 \right)$, if $\Delta x_i
\in \mathcal{O}\left(h\right)$ for $i=1, \ldots , n$.
The leading error terms are given by
$ \frac{b_{ij}}{12}\frac{\partial^4 u}{\partial x_i \partial x_j^3}\left[ (\Delta x_j)^2 - \frac{a_j (\Delta x_i)^2}{a_i} \right] $ for $i,j\in\{1,\ldots , n\}$ with $i \neq j$. If the conditions
\begin{equation}\label{restriction_general_pde_nD_number_one}
b_{ij} =  0\quad  \text{or} \quad  \frac{(\Delta x_j)^2}{(\Delta x_i)^2}  =  \frac{a_j}{a_i}
\end{equation}
are fulfilled for all $i,j \in \{1, \ldots , n\}$ with $i \neq j$ these second order terms vanish and the resulting error term is of fourth order. 
Hence, for any partial differential equation
\eqref{basic_general_pde_nD_without_usage_of_f} which satisfies
\eqref{restriction_general_pde_nD_number_one} we obtain a high-order compact scheme.
In the case $b_{i,j} \equiv 0$ for all $i,j \in {1, \ldots, n}$, it is
possible to choose $\Delta x_i>0$ freely for each spatial direction, whereas in
other possible cases 
there are interdependencies
for at least some of the step sizes. For each pair $(i,j)$ with $b_{ij}\neq 0$ the condition $\frac{(\Delta x_j)^2}{(\Delta x_i)^2}  =  \frac{a_j}{a_i}$ 
has to hold for all 
$x = \big(x^{(1)}_{i_1}, x^{(2)}_{i_2} ,\ldots , x^{(n)}_{i_n}\big)\in \interior{G}^{(n)} $.
This means ${a_j}/{a_i}$ has to be constant as ${(\Delta x_j)^2}/{(\Delta x_i)^2}$ is constant, 
see \eqref{n_dimensional_Grid_general_general_stepsizes}. 

\section{Semi-discrete high-order compact schemes}
\label{sec:hocsemi}

In this section we present the semi-discrete high-order compact
schemes in spatial dimensions
$n=2,3$. We consider the case where the cross derivatives do not
vanish, hence we assume, for simplicity, $a_i \equiv a$ 
in combination with $\Delta x_i = h>0$
for $i=1, \ldots n$ to satisfy condition
\eqref{restriction_general_pde_nD_number_one}. Our aim in this section is to derive a
semi-discrete scheme of the form
\begin{align}\label{semi_discrete_pde_two_dimensions_general}
\sum\limits_{\hat{x}\in G^{(n)}} \left[M_x(\hat{x},\tau) \partial_{\tau}U_{i_1, \ldots , i_n}(\tau) + K_x(\hat{x},\tau) U_{i_1, \ldots , i_n}(\tau) \right] = & \tilde{g}(x,\tau) 
\end{align}
at each point $x\in \interior{G}^{(n)}$ with $\Delta x_i = h>0$ for $i=1, \ldots , n$ and time $\tau$, where the
function $\tilde{g}: \interior{G}^{(n)} \times \Omega_{\tau}
\rightarrow \mathbb{R}$ depends on the function $g$ given in
\eqref{basic_general_pde_nD_without_usage_of_f}.

\subsection{Semi-discrete two-dimensional scheme} \label{Section_2DGeneralPDESemi_discrete}
In this section we derive the high-order compact discretisation of
\eqref{basic_general_pde_nD_without_usage_of_f} in spatial dimension
$n=2$. Considering the grid point $\bigl(x_{i_1}^{(1)},x_{i_2}^{(2)}\bigr) \in \interior{G}^{(2)}  $
with $\Delta x_1 = \Delta x_2 =h>0$ and time $\tau \in \Omega_{\tau}$ we are able to obtain the
coefficients $\hat{K}_{l,m}$ of $U_{l,m}\left(\tau\right)$ for $l \in
\{i_1 - 1, i_1, i_1 + 1\}$ and $m \in \{i_2 - 1, i_2, i_2 + 1\}$ on
the compact stencil by employing the central difference operator in \eqref{pde_to_discretise_n_D}.
To streamline notation we denote by $[\cdot ]_{k}$ the first
derivative with respect to $x_k$  and by $[ \cdot ]_{kp}$ the
second derivative, once in $x_k$- and once in $x_p$-direction with
$k,p \in \{1,2\}$. Note that in the following the functions $a$,
$b_{1,2}$, $c_1$, $c_2$ and $g$ are evaluated at
$\bigl(x_{i_1}^{(1)},x_{i_2}^{(2)}\bigr)\in \interior{G}^{(2)}$ and
$\tau \in \Omega_{\tau}$. We omit these arguments for the sake of
readability. The coefficients are given by:
\begin{small}
\begin{align*}
\hat{K}_{i_1,i_2}  =  &- {\frac {b_{{12}}[a]_{{{{1}}{{2}}}}}{3a}}
- {\frac {b_{{12}}[c_{{2}}]_{{{{1}}}}}{6a}}
+ {\frac {b_{{12}}[a]_{{{{2}}}}c_{{1}}}{6{a}^{2}}}
+ {\frac {2b_{{12}}[a]_{{{{1}}}}[a]_{{{{2}}}}}{3{a}^{2}}}
- \frac{[a]_{{{{{2}}}{{{2}}}}}}{3}
- {\frac {{c^{2}_{{1}}}}{6a}}
+ {\frac {2{[a]^{2}_{{{{1}}}}}}{3a}}\\
&
- \frac{[a]_{{{{{1}}}{{{1}}}}}}{3}
- {\frac {10a}{3{h}^{2}}}
- \frac{[c_{{2}}]_{{{{2}}}}}{3}
- \frac{[c_{{1}}]_{{{{1}}}}}{3}
- {\frac {b_{{12}}[c_{{1}}]_{{{{2}}}}}{6a}}
+ {\frac {2{[a]^{2}_{{{{2}}}}}}{3a}}
- {\frac {{c^{2}_{{2}}}}{6a}}
+ {\frac {{b^{2}_{{12}}}}{3a{h}^{2}}}
+ {\frac {b_{{12}}[a]_{{{{1}}}}c_{{2}}}{6{a}^{2}}} ,\\
\hat{K}_{i_1\pm 1,i_2}  = & {\frac {c_{{2}}[a]_{{{{2}}}}}{12a}}
- {\frac {{b^{2}_{{12}}}}{6a{h}^{2}}}
+ {\frac {b_{{12}}[a]_{{{{1}}{{2}}}}}{12a}}
- {\frac {c_{{1}}[a]_{{{{1}}}}}{12a}}
\mp {\frac {hb_{{12}}[a]_{{{{2}}}}[c_{{1}}]_{{{{1}}}}}{24{a}^{2}}}
\mp {\frac {hb_{{12}}[a]_{{{{1}}}}[c_{{1}}]_{{{{2}}}}}{24{a}^{2}}}
\pm \frac{h[c_{{1}}]_{{{{{1}}}{{{1}}}}}}{24}\\
&
\pm \frac{h[c_{{1}}]_{{{{{2}}}{{{2}}}}}}{24}
+ {\frac {{c^{2}_{{1}}}}{12a}}
\pm {\frac {hc_{{1}}[c_{{1}}]_{{{{1}}}}}{24a}}
\mp {\frac {h[a]_{{{{1}}}}[c_{{1}}]_{{{{1}}}}}{12a}}
\pm {\frac {hb_{{12}}[c_{{1}}]_{{{{1}}{{2}}}}}{24a}}
- {\frac {b_{{12}}[a]_{{{{2}}}}c_{{1}}}{12{a}^{2}}}
\pm {\frac {hc_{{2}}[c_{{1}}]_{{{{2}}}}}{24a}}\\
&
\mp {\frac {h[a]_{{{{2}}}}[c_{{1}}]_{{{{2}}}}}{12a}}
+ \frac{[c_{{1}}]_{{{{1}}}}}{6}
- {\frac {{[a]^{2}_{{{{1}}}}}}{6a}}
- {\frac {{[a]^{2}_{{{{2}}}}}}{6a}}
+ \frac{[a]_{{{{{2}}}{{{2}}}}}}{12}
+ \frac{[a]_{{{{{1}}}{{{1}}}}}}{12}
\mp {\frac {c_{{2}}b_{{12}}}{6ah}}
\mp {\frac {b_{{12}}[b_{{12}}]_{{{{1}}}}}{12ah}}\\
&
+ {\frac {b_{{12}}[c_{{1}}]_{{{{2}}}}}{12a}}
+ {\frac {2a}{3{h}^{2}}}
- {\frac {b_{{12}}[a]_{{{{1}}}}[a]_{{{{2}}}}}{6{a}^{2}}}
\pm {\frac {b_{{12}}[a]_{{{{2}}}}}{6ah}}
\mp {\frac {[b_{{12}}]_{{{{2}}}}}{6h}}
\pm {\frac {{b^{2}_{{12}}}[a]_{{{{1}}}}}{12{a}^{2}h}}
\pm {\frac {c_{{1}}}{3h}} ,\\
\hat{K}_{i_1,i_2\pm 1}  = & - {\frac {c_{{2}}[a]_{{{{2}}}}}{12a}}
- {\frac {{b^{2}_{{12}}}}{6a{h}^{2}}}
+ {\frac {b_{{12}}[c_{{2}}]_{{{{1}}}}}{12a}}
+ {\frac {b_{{12}}[a]_{{{{1}}{{2}}}}}{12a}}
+ {\frac {c_{{1}}[a]_{{{{1}}}}}{12a}}
\mp {\frac {hb_{{12}}[a]_{{{{2}}}}[c_{{2}}]_{{{{1}}}}}{24{a}^{2}}}
+ \frac{[c_{{2}}]_{{{{2}}}}}{6}
\\
&
\mp {\frac {hb_{{12}}[a]_{{{{1}}}}[c_{{2}}]_{{{{2}}}}}{24{a}^{2}}}
- {\frac {{[a]^{2}_{{{{1}}}}}}{6a}}
- {\frac {{[a]^{2}_{{{{2}}}}}}{6a}}
+ {\frac {{c^{2}_{{2}}}}{12a}}
+ \frac{[a]_{{{{{2}}}{{{2}}}}}}{12}
+ \frac{[a]_{{{{{1}}}{{{1}}}}}}{12}
\mp {\frac {b_{{12}}[b_{{12}}]_{{{{2}}}}}{12ah}}
\pm \frac{h[c_{{2}}]_{{{{{2}}}{{{2}}}}}}{24}\\
&
\pm \frac{h[c_{{2}}]_{{{{{1}}}{{{1}}}}}}{24}
+ {\frac {2a}{3{h}^{2}}}
\pm {\frac {hc_{{1}}[c_{{2}}]_{{{{1}}}}}{24a}}
\mp {\frac {h[a]_{{{{1}}}}[c_{{2}}]_{{{{1}}}}}{12a}}
- {\frac {b_{{12}}[a]_{{{{1}}}}[a]_{{{{2}}}}}{6{a}^{2}}}
\pm {\frac {hb_{{12}}[c_{{2}}]_{{{{1}}{{2}}}}}{24a}}
\pm {\frac {c_{{2}}}{3h}}
\\
&
- {\frac {b_{{12}}[a]_{{{{1}}}}c_{{2}}}{12{a}^{2}}}
\mp {\frac {h[a]_{{{{2}}}}[c_{{2}}]_{{{{2}}}}}{12a}}
\pm {\frac {hc_{{2}}[c_{{2}}]_{{{{2}}}}}{24a}}
\pm {\frac {{b^{2}_{{12}}}[a]_{{{{2}}}}}{12{a}^{2}h}}
\pm {\frac {b_{{12}}[a]_{{{{1}}}}}{6ah}}
\mp {\frac {c_{{1}}b_{{12}}}{6ah}}
\mp {\frac {[b_{{12}}]_{{{{1}}}}}{6h}},\\
\hat{K}_{i_1\pm 1,i_2-1}  = &  {\frac {{b^{2}_{{12}}}}{12a{h}^{2}}}
\mp {\frac {c_{{1}}c_{{2}}}{24a}}
\pm {\frac {[a]_{{{{2}}}}c_{{1}}}{24a}}
\mp  {\frac {b_{{12}}[c_{{2}}]_{{{{2}}}}}{48a}}
\pm  {\frac {[a]_{{{{2}}}}[b_{{12}}]_{{{{2}}}}}{24a}}
\pm  {\frac {[a]_{{{{1}}}}c_{{2}}}{24a}}
\pm  {\frac {[a]_{{{{1}}}}[b_{{12}}]_{{{{1}}}}}{24a}}\\
&
\mp  {\frac {c_{{1}}[b_{{12}}]_{{{{1}}}}}{48a}}
\mp  {\frac {b_{{12}}[c_{{1}}]_{{{{1}}}}}{48a}}
\mp  {\frac {c_{{2}}[b_{{12}}]_{{{{2}}}}}{48a}}
\mp  {\frac {b_{{12}}[b_{{12}}]_{{{{1}}{{2}}}}}{48a}}
\mp  \frac{[c_{{1}}]_{{{{2}}}}}{24}
\mp  \frac{[c_{{2}}]_{{{{1}}}}}{24}
\mp  \frac{[b_{{12}}]_{{{{{1}}}{{{1}}}}}}{48}\\
&
\mp \frac{[b_{{12}}]_{{{{{2}}}{{{2}}}}}}{48}
\mp {\frac {b_{{12}}[b_{{12}}]_{{{{2}}}}}{24ah}}
\pm {\frac {c_{{2}}b_{{12}}}{12ah}}
\pm {\frac {b_{{12}}[b_{{12}}]_{{{{1}}}}}{24ah}}
\pm {\frac {b_{{12}}[a]_{{{{2}}}}[b_{{12}}]_{{{{1}}}}}{48{a}^{2}}}
\pm {\frac {b_{{12}}[a]_{{{{1}}}}c_{{1}}}{48{a}^{2}}}
+ {\frac {a}{6{h}^{2}}}\\
&
+ {\frac {{b^{2}_{{12}}}[a]_{{{{2}}}}}{24{a}^{2}h}}
\pm {\frac {b_{{12}}[a]_{{{{2}}}}c_{{2}}}{48{a}^{2}}}
+ {\frac {b_{{12}}[a]_{{{{1}}}}}{12ah}}
\mp {\frac {b_{{12}}[a]_{{{{2}}}}}{12ah}}
- {\frac {c_{{1}}b_{{12}}}{12ah}}
\pm {\frac {b_{{12}}[a]_{{{{1}}}}[b_{{12}}]_{{{{2}}}}}{48{a}^{2}}}
- {\frac {[b_{{12}}]_{{{{1}}}}}{12h}}\\
&
\pm {\frac {[b_{{12}}]_{{{{2}}}}}{12h}}
\mp {\frac {{b^{2}_{{12}}}[a]_{{{{1}}}}}{24{a}^{2}h}}
\mp {\frac {b_{{12}}}{4{h}^{2}}}
- {\frac {c_{{2}}}{12h}}
\pm {\frac {c_{{1}}}{12h}} ,\\
\hat{K}_{i_1\pm 1,i_2+1}  = &  {\frac {{b^{2}_{{12}}}}{12a{h}^{2}}}
\pm {\frac {c_{{1}}c_{{2}}}{24a}}
\mp {\frac {[a]_{{{{2}}}}c_{{1}}}{24a}}
\pm {\frac {b_{{12}}[c_{{2}}]_{{{{2}}}}}{48a}}
\mp {\frac {[a]_{{{{2}}}}[b_{{12}}]_{{{{2}}}}}{24a}}
\mp {\frac {[a]_{{{{1}}}}c_{{2}}}{24a}}
\mp {\frac {[a]_{{{{1}}}}[b_{{12}}]_{{{{1}}}}}{24a}}\\
&
\pm {\frac {c_{{1}}[b_{{12}}]_{{{{1}}}}}{48a}}
\pm {\frac {b_{{12}}[c_{{1}}]_{{{{1}}}}}{48a}}
\pm {\frac {c_{{2}}[b_{{12}}]_{{{{2}}}}}{48a}}
\pm {\frac {b_{{12}}[b_{{12}}]_{{{{1}}{{2}}}}}{48a}}
\pm \frac{[c_{{1}}]_{{{{2}}}}}{24}
\pm \frac{[c_{{2}}]_{{{{1}}}}}{24}
\pm \frac{[b_{{12}}]_{{{{{1}}}{{{1}}}}}}{48}\\
&
\pm \frac{[b_{{12}}]_{{{{{2}}}{{{2}}}}}}{48}
+ {\frac {b_{{12}}[b_{{12}}]_{{{{2}}}}}{24ah}}
\pm {\frac {c_{{2}}b_{{12}}}{12ah}}
\pm {\frac {b_{{12}}[b_{{12}}]_{{{{1}}}}}{24ah}}
\mp {\frac {b_{{12}}[a]_{{{{2}}}}[b_{{12}}]_{{{{1}}}}}{48{a}^{2}}}
\mp {\frac {b_{{12}}[a]_{{{{1}}}}c_{{1}}}{48{a}^{2}}}
+ {\frac {a}{6{h}^{2}}}\\
&
- {\frac {{b^{2}_{{12}}}[a]_{{{{2}}}}}{24{a}^{2}h}}
\mp {\frac {b_{{12}}[a]_{{{{2}}}}c_{{2}}}{48{a}^{2}}}
- {\frac {b_{{12}}[a]_{{{{1}}}}}{12ah}}
\mp {\frac {b_{{12}}[a]_{{{{2}}}}}{12ah}}
+ {\frac {c_{{1}}b_{{12}}}{12ah}}
\mp {\frac {b_{{12}}[a]_{{{{1}}}}[b_{{12}}]_{{{{2}}}}}{48{a}^{2}}}
+ {\frac {[b_{{12}}]_{{{{1}}}}}{12h}}\\
&
\pm {\frac {[b_{{12}}]_{{{{2}}}}}{12h}}
\mp {\frac {{b^{2}_{{12}}}[a]_{{{{1}}}}}{24{a}^{2}h}}
\pm {\frac {b_{{12}}}{4{h}^{2}}}
+ {\frac {c_{{2}}}{12h}}
\pm {\frac {c_{{1}}}{12h}}.
\end{align*}
\end{small}
Analogously, we obtain the coefficients $\hat{M}_{l,m}$ of
$\partial_{\tau}U_{l,m}\left(\tau\right)$ for $l \in \{i_1 - 1, i_1,
i_1 + 1\}$ and $m \in \{i_2 - 1, i_2, i_2 + 1\}$ at each point
$\bigl(x_{i_1}^{(1)},x_{i_2}^{(2)}\bigr)\in \interior{G}^{(2)}$ and
time $\tau \in \Omega_{\tau}$, 
\begin{align*}
\begin{split}
\hat{M}_{i_1+1, i_2\pm 1} = & \hat{M}_{i_1-1, i_2\mp 1}  = \pm {\frac {b_{{12}}}{48a}},
\quad \hat{M}_{i_1,i_2\pm 1}  =   \frac{1}{12}
\mp {\frac {h[a]_{{{{2}}}}}{12a}}
\mp {\frac {b_{{12}}h[a]_{{{{1}}}}}{24{a}^{2}}}
\pm {\frac {c_{{2}}h}{24a}}, 
\\
\hat{M}_{i_1\pm 1, i_2}  = & \frac{1}{12}
\mp {\frac {b_{{12}}h[a]_{{{{2}}}}}{24{a}^{2}}}
\pm {\frac {hc_{{1}}}{24a}}
\mp {\frac {h[a]_{{{{1}}}}}{12a}} , \quad \hat{M}_{i_1,i_2}  =  \frac{2}{3},
\end{split}
\end{align*}
where $\Delta x_1=\Delta x_2 = h>0$. Additionally, for $x\in \interior{G}^{(2)}$,  $\tau \in \Omega_{\tau}$,
\begin{align*}
\begin{split}
\tilde{g}(x,\tau)  = & {\frac { \left({h}^{2}{a}^{2}c_{{1}} -2{h}^{2}{a}^{2}[a]_{{{{1}}}}-b_{{12}}{h}^
{2}[a]_{{{{2}}}}a \right) [g]_{{{{1}}}}}{12
{a}^{3}}}
+ \frac{{h}^{2}[g]_{{{{{1}}}{{{1}}}}}}{12}
+ {\frac {b_{{12}}{h}^{2}[g]_{{{{1}}{{2}}}}}{12a}}
\\
&
+ {\frac { \left( 
{h}^{2}{a}^{2}c_{{2}}-b_{{12}}{h}^{2}[a]_{{{{1}}}}a -2{h}^{2}{a}^{2}[a]_{{{{2}}}} \right) [g]_{{x
_{{2}}}}}{12{a}^{3}}}
+ \frac{{h}^{2}[g]_{{{{{2}}}{{{2}}}}}}{12}
+ g
\end{split}
\end{align*}
holds, where $\Delta x_1 = \Delta x_2 = h>0$ was used. We have 
$K_x(x_{n_1}^{(1)},x_{n_2}^{(2)}, \tau)  =\hat{ K}_{n_1,n_2}$ and
$M_x(x_{n_1}^{(1)},x_{n_2}^{(2)}, \tau)  =\hat{ M}_{n_1,n_2}$
in \eqref{semi_discrete_pde_two_dimensions_general} with $n_1 \in \{
i_1 -1, i_1, i_1+1\}$ and  $n_2 \in \{ i_2 -1, i_2, i_2+1\}$ for
$x=\bigl(x_{i_1}^{(1)},x_{i_2}^{(2)}\bigr)\in \interior{G}^{(2)}$
and $\tau \in \Omega_{\tau}$. $K_x$ and $M_x$ are zero otherwise and
the approximation only uses points of the compact stencil.

\subsection{Semi-discrete three-dimensional scheme}\label{Section_3DGeneralPDESemi_discrete}
In this section we derive the high-order compact discretisation of \eqref{basic_general_pde_nD_without_usage_of_f} in spatial dimension $n=3$. 
Considering the conditions in \eqref{restriction_general_pde_nD_number_one} we observe that in the three-dimensional case 
we have three different possibilities to satisfy the conditions 
and thus obtain a high-order compact scheme. We focus on the case 
$a=a_1\equiv a_2 \equiv a_3$ and set $h=\Delta x_1 = \Delta x_2 = \Delta x_3.$
Considering an interior grid point
$\bigl(x_{i_1}^{(1)},x_{i_2}^{(2)},x_{i_3}^{(3)}\bigr)\in
\interior{G}^{(3)}$ and time $\tau \in
\Omega_{\tau}$ we are able to produce the coefficients $\hat{K}_{k,l,m}$ of
$U_{k,l,m}\left(\tau\right)$ for $k\in \{ i_1 - 1, i_1, i_1 + 1\}$,
$l\in \{ i_2 - 1, i_2, i_2 + 1\}$ and $m\in \{ i_3 - 1, i_3, i_3 +
1\}$ by employing the central difference operator in
\eqref{pde_to_discretise_n_D}.
Again, to streamline the notation we denote by $[\cdot ]_{k}$ and $[
\cdot ]_{kp}$ the first and second derivative of the coefficients with
respect to $x_k$, and with respect to $x_k$ and $x_p$, respectively.
Note again that in the following $a, b_{12}, b_{13}, b_{23}, c_1, c_2, c_3$ and $g$ are evaluated at
$\bigl(x_{i_1}^{(1)},x_{i_2}^{(2)},x_{i_3}^{(3)} \bigr)\in
\interior{G}^{(3)}$ and $\tau\in \Omega_{\tau}$, where $\Delta x_i =h>0$ for $i=1,2,3$.
We omit these arguments for the sake of readability. 
Due to the length of the coefficient expressions $\hat{K}_{k,l,m}$,
they are given in the appendix.

In a similar way we define  $\hat{M}_{k,l,m}$ as the coefficient of $\partial_{\tau} U_{k,l,m}\left(\tau\right)$ for \\
$k\in \{ i_1 - 1, i_1, i_1 + 1\}$, $l\in \{ i_2 - 1, i_2, i_2 + 1\}$ and $m\in \{ i_3 - 1, i_3, i_3 + 1\}$ by
\begin{align*}
\hat{M}_{i_1\pm 1,i_2-1,i_3}  = & \hat{M}_{i_1 \mp 1,i_2+1,i_3 } = \mp  {\frac {b_{{12}}}{48a}} \text{, }\quad \hat{M}_{i_1,i_2,i_3} =  \frac{1}{2},
\\
\hat{M}_{i_1\pm 1,i_2,i_3-1}  = & \hat{M}_{i_1\mp 1,i_2,i_3+1} = \mp {\frac {b_{{13}}}{48a}},\quad \hat{M}_{i_1,i_2\pm 1,i_3-1}  =  \hat{M}_{i_1,i_2\mp 1, i_3+1} = \mp {\frac {b_{{23}}}{48a}},
\\
\hat{M}_{i_1 \pm 1,i_2,i_3}  = &  \frac{1}{12}
 \mp {\frac {hb_{{12}}[a]_{{{{2}}}}}{24{a}^{2}}}
 \mp {\frac {hb_{{13}}[a]_{{{{3}}}}}{24{a}^{2}}}
 \pm {\frac {hc_{{1}}}{24a}}
\mp {\frac {h[a]_{{{{1}}}}}{12a}},
\\
\hat{M}_{i_1,i_2\pm 1,i_3 }  = &  \frac{1}{12}
 \mp {\frac {hb_{{12}}[a]_{{{{1}}}}}{24{a}^{2}}}
  \mp {\frac {hb_{{23}}[a]_{{{{3}}}}}{24{a}^{2}}}
 \pm {\frac {hc_{{2}}}{24a}}
 \mp {\frac {h[a]_{{{{2}}}}}{12a}},
\\
\hat{M}_{i_1,i_2,i_3\pm 1}  = & \frac{1}{12}
 \mp {\frac {hb_{{23}}[a]_{{{{2}}}}}{24{a}^{2}}}
 \mp {\frac {hb_{{13}}[a]_{{{{1}}}}}{24{a}^{2}}}
 \pm {\frac {hc_{{3}}}{24a}}
\mp {\frac {h[a]_{{{{3}}}}}{12a}},
\\
\hat{M}_{i_1 \pm 1,i_2-1,i_3-1}  = & \hat{M}_{i_1\pm 1,i_2+1,i_3-1}=
\hat{M}_{i_1 \pm 1,i_2-1,i_3+1}  = \hat{M}_{i_1\pm 1,i_2+1,i_3+1}=0.
\end{align*}
For the right hand side of
\eqref{semi_discrete_pde_two_dimensions_general} we have for
$x=\bigl(x_{i_1}^{(1)},x_{i_2}^{(2)},x_{i_3}^{(3)} \bigr)\in
\interior{G}^{(3)}$, $\tau\in \Omega_{\tau}$,
\begin{align*}
\tilde{g}(x,\tau)  = & 
{\frac { \left( c_{{1}}{h}^{2}a-2{h}^{2}[a]_{{{{1}}}}a-b_{{12}}{h}^{2}[a]_{{{{2}}}}-b_{{13}}{h}^{2}[a]_{{{{3}}}} \right) [g]_{{{{1}}}}}{12{a}^{2}}} 
 +{\frac {b_{{13}}{h}^{2}[g]_{{{{1}}{{3}}}}}{12a}}
\\
&
 + {\frac { \left( c_{{2}}{h}^{2}a-2{h}^{2}[a]_{{{{2}}}}a-b_{{12}}{h}^{2}[a]_{{{{1}}}}-b_{{23}}{h}^{2}[a]_{{{{3}}}} \right) [g]_{{{{2}}}}}{12{a}^{2}}}
+ {\frac {b_{{23}}{h}^{2}[g]_{{{{2}}{{3}}}}}{12a}}
\\
&
+ {\frac { \left( c_{{3}}{h}^{2}a-2{h}^{2}[a]_{{{{3}}}}a-b_{{13}}{h}^{2}[a]_{{{{1}}}}-b_{{23}}{h}^{2}[a]_{{{{2}}}} \right) [g]_{{{{3}}}}}{12{a}^{2}}}
 + \frac{{h}^{2}[g]_{{{{{1}}}{{{1}}}}}}{12}
\\
&
+{\frac {b_{{12}}{h}^{2}[g]_{{{{1}}{{2}}}}}{12a}}
+\frac{{h}^{2}[g]_{{{{{3}}}{{{3}}}}}}{12}
+\frac{{h}^{2}[g]_{{{{{2}}}{{{2}}}}}}{12}
+g .
\end{align*}
We define $K_x(x_{n_1}^{(1)},x_{n_2}^{(2)},x_{n_3}^{(3)},\tau) =\hat{
  K}_{n_1,n_2,n_3}$ and 
$M_x(x_{n_1}^{(1)},x_{n_2}^{(2)},x_{n_3}^{(3)},\tau)  =\hat{ M}_{n_1,n_2,n_3}$
for each point $x=\bigl(x_{i_1}^{(1)},x_{i_2}^{(2)},x_{i_3}^{(3)} \bigr)\in \interior{G}^{(3)}$ and $\tau \in \Omega_{\tau}$,
where $n_j \in \{ i_j -1, i_j, i_j+1\}$ with $j=1,2,3$. 
$K_x$ and $M_x$ are zero otherwise. 
Hence, the approximation only uses points of the compact stencil \eqref{Def_compact_stencil_n_D}.

\section{Fully discrete scheme}
\label{sec:hoctime}
The semi-discrete scheme presented in the previous
sections can be extended to a fully discrete scheme for the
parabolic problem \eqref{basic_general_pde_nD_without_usage_of_f}
by additionally discretising in time. Any time integrator can be implemented to solve the problem as in \cite{SpoCar01}. 
Here we consider 
a Crank-Nicolson type time-discretisation with constant time step $\Delta\tau$ to obtain a fully
discrete scheme. Let 
\begin{align*}
A_x(\hat{x},\tau_{k+1}) = \hat{M}_x\left(\hat{x},\tau_{k}\right) +
\frac{\Delta \tau}{2}K_x \left( \hat{x},\tau_{k+1}\right),\;
B_x(\hat{x},\tau_{k}) = \hat{M}_x\left(\hat{x},\tau_{k}\right) -
\frac{\Delta \tau}{2}K_x \left( \hat{x},\tau_{k}\right),
\end{align*}
where 
$\hat{M}_x\left(\hat{x},\tau_{k}\right) =\left(M_x\left(\hat{x},\tau_{k}\right) + M_x\left(\hat{x},\tau_{k+1}\right)\right)/2$. 
$K_x\left(\hat{x},\tau\right)$ and
$M_x\left(\hat{x},\tau\right)$ are defined through a semi-discrete
finite difference scheme with fourth-order consistency using only points of
the compact stencil \eqref{Def_compact_stencil_n_D}.
Then, a fully discrete high-order compact finite difference scheme for \eqref{basic_general_pde_nD_without_usage_of_f} with $n \in \mathbb{N}$ 
on the time grid $\tau_k = k \Delta_{\tau}$ for $k=0, \ldots , N_{\tau}$ and $\Delta x_i=h$ for all $i$ 
is given at each point $x=\bigl(x_{i_1}^{(1)},\ldots , x_{i_n}^{(n)}\bigr)\in\interior{G}^{(n)}$ by
\begin{align}\label{Def_general_n_dim_HOC_scheme_for_stability}
  \sum\limits_{\hat{x}\in \hat{U}\left( x \right)}  A_x\left(\hat{x},\tau_{k+1}\right) U_{l_1,\ldots, l_n}^{k+1} 
=& 
 \sum\limits_{\hat{x}\in \hat{U}\left( x \right)}  B_x\left(\hat{x},\tau_{k}\right) U_{l_1,\ldots ,l_n}^{k} 
+  \frac{\Delta \tau}{2}\hat{g}\left(x,\tau_{k},\tau_{k+1}\right),
\end{align}
where $\hat{g}\left(x,\tau_{k},\tau_{k+1}\right)= \tilde{g}\left(x,\tau_k\right) + \tilde{g}\left(x,\tau_{k+1}\right)$ and $\hat{x }=\bigl(x_{l_1}^{(1)},\ldots , x_{l_n}^{(n)}\bigr)\in \hat{U}\left(x\right)$. 
This scheme is second-order consistent in time and fourth-order
consistent in space. 
We have fourth-order
consistency in terms of $h$ for $\Delta \tau \in \mathcal{O}(h^2)$ while only using the compact stencil.
Note that up to this point only the spatial interior is discussed. The applied boundary conditions may still have an effect the above numerical scheme. 

\section{Stability analysis for the Cauchy problem in dimensions
  $\bf{n=2,3}$}
\label{sec:stability}
In this section we consider the stability analysis of the high-order
compact scheme for the Cauchy problem associated with
\eqref{basic_general_pde_nD_without_usage_of_f} in the case
$n=2,3$. The coefficients of the semi-discrete scheme are given in
Section~\ref{Section_2DGeneralPDESemi_discrete} for two spatial
dimensions and in Section~\ref{Section_3DGeneralPDESemi_discrete},
when three spatial dimensions occur. Those coefficients are
non-constant, as the coefficients of the parabolic partial
differential equation \eqref{basic_general_pde_nD_without_usage_of_f}
are non-constant.

We consider a von Neumann stability analysis. Other approaches which take into account boundary
conditions like normal mode analysis \cite{GuKrOl96} are beyond the
scope of the present paper. For both $n=2$ and $n=3$, we give a proof of
stability in the case of vanishing cross derivative terms and frozen
coefficients in time and space, which means that all possible values
for the coefficients are considered, but as constants, hence the
derivatives of the coefficients of the partial differential equation
appearing in the discrete schemes are set to zero. This approach has
been used as well in \cite{GuKrOl96,Strick04} and gives a necessary
stability condition, whereas slightly stronger conditions are sufficient to ensure
overall stability \cite{RiMo67}.  This approach is extensively used in
the literature and yields good criteria on the robustness of the
scheme. 
In \eqref{Def_general_n_dim_HOC_scheme_for_stability} we use
\begin{align*}
U_{j_1, \ldots , j_n}^k = g^k e^{IS_n}\quad  \text{with} \quad S_n=\sum\limits_{m=1}^nj_mz_m 
\end{align*}
for $j_m \in \{ i_m -1, i_m , i_m + 1 \}$, where $I$ is the imaginary
unit, $g^k$ is the amplitude at time level $k$ and $z_m=2\pi h
/\lambda_m$ for the wavelength $\lambda_m \in [0,2\pi[$ for
$m=1,\ldots ,n$. Then the fully discrete scheme satisfies the {\em necessary von Neumann stability condition\/} for all $z_1,z_2$, when the amplification factor $ G=g^{k+1}/g^k$ satisfies 
\begin{align} \label{general_stability_condition}
\vert G \vert^2 - 1 \leq 0 .
\end{align}

\subsection{Stability analysis for the two-dimensional case}

In this section we perform the von Neumann stability analysis for the
two-dimensional high-order compact scheme of
Section~\ref{Section_2DGeneralPDESemi_discrete}. The analysis of the
case with vanishing cross-derivative and frozen coefficients are
carried out in detail. In the case of non-vanishing mixed derivatives partial results are given for frozen coefficients.


\begin{theorem}
\label{thm:stab2d}
For $a=a_1=a_2<0$, $b_{1,2}=0$ and $\Delta x_1 = \Delta x_2 = h>0$, the fully discrete high-order
compact finite difference scheme given in
\eqref{Def_general_n_dim_HOC_scheme_for_stability} with $n=2$, with
coefficients defined in
Section~\ref{Section_2DGeneralPDESemi_discrete}, satisfies (for frozen
coefficients) the necessary stability condition \eqref{general_stability_condition}.
\end{theorem}
\begin{proof}
Let $\xi_1=\cos({z_1/2})$, $\xi_2=\cos({z_2/2})$, $\eta_1=\sin(z_1/2)$ and $\eta_2=\sin({z_2/2}).$
The stability condition \eqref{general_stability_condition} for the fully
discrete scheme \eqref{Def_general_n_dim_HOC_scheme_for_stability} using the
coefficients defined in
Section~\ref{Section_2DGeneralPDESemi_discrete} yields
$\vert G \vert^2 - 1  = {N_G}/{D_G}$ (explicit expressions for $N_G$,
$D_G$ are given below).
We discuss the numerator $N_G$ and the denominator $D_G$
separately in the following.

The numerator can be written as
$ N_G = 8ka\left(n_4 h^4 + n_2h^2 \right)$
where the polynomials
\begin{align*}
n_2 =& 8a^2 f_1\left(\xi_1,\xi_2 \right)f_2\left(\xi_1,\xi_2\right) 
\quad \text{and} \quad n_4 = f_3\left(\xi_1\right) f_4\left(\xi_1,\xi_2\right) c_1^2 + f_3\left(\xi_2\right) f_4\left(\xi_2,\xi_1\right) c_2^2 
\end{align*}
are non-negative, since
\begin{align*}
 f_1\left(x,y\right)  = & x^2+y^2+1\geq 0,&
 f_2\left(x,y\right)  = & 2 - x\Bigl(y^2 + \frac{1}{2}\Bigr)- \frac{y^2}{2} \geq 0, \\
 f_3\left(x\right) =& x^2 - 1 \leq 0,&
 f_4\left(x,y\right) =& 2x^2 y^2 - x^2 -1 \leq 0,
\end{align*}
for $x,y \in [-1,1]$. Hence, we observe that $N_G\leq 0$ holds, as
$\xi_1,\xi_2 \in \left[ -1,1 \right]$. 

Now we consider the denominator $D_G$, which can be written as
$$D_G = d_6 h^6 + \left(d_{4,2}k^2 + d_{4,1}k + d_{4,0} \right)h^4 + (d_{2,2}k^2 + d_{2,1}k) h^2 + d_0,$$
where 
\begin{align*}
d_0  = & 16 a^4 k^2  \left(2\xi_1^2 \xi_2^2 + \xi_1^2  +\xi_2^2 - 4\right)^2 \geq 0,\quad
d_{2,1}  = 16a^3 f_1\left(\xi_1,\xi_2\right) f_5\left( \xi_1,\xi_2 \right)\geq 0 ,\\
d_{2,2}  = &4a^2\left[9\left(\xi_1\eta_1 c_1 + \xi_2\eta_2 c_2 \right)^2 
+ 2 f_3\left(\xi_1\right)f_6\left(\xi_1,\xi_2\right)c_1^2 
+ 2 f_3\left(\xi_2\right)f_6\left(\xi_2,\xi_1\right)c_2^2 \right],\\
d_{4,0}  = & 4 a^2 f_1\left( \xi_1,\xi_2\right)^2 \geq 0, \quad
d_{4,1}  =  -4a n_4 \geq 0 ,\\
d_{4,2}  = & \left[f_3(\xi_1) c_1^2 - 2\eta_1 \eta_2 \xi_1 \xi_2 c_1c_2 + f_3(\xi_2)c_2^2 \right]^2  \geq 0,\quad
d_6  =  \left(\xi_1 \eta_1 c_1 + \xi_2 \eta_2 c_2\right)^2 \geq 0,
\end{align*}
because $a < 0$ and where
\begin{align*}
f_5\left( x,y \right)  = & 2x^2y^2 +x^2+y^2-4  \leq 0,\quad
f_6\left( x,y \right)  =  2x^2y^4 - 5 x^2 - y^2 +4
\end{align*}
with $x,y \in [-1,1]$. We observe that $f_6\left( x,y \right)$ changes
sign, as, for example $f_6\left( 0,0 \right)=4$ and $f_6\left( 1,0
\right)=-1$. Hence, we cannot determine the sign of $d_{2,2}$
directly. 

If $c_1=c_2=0$, we have $d_{2,2} = 0$ and hence
$D_G\geq 0$. Since $d_{2,2}$ is
symmetric, we can say without loss of generality that $c_1 \neq 0$ in
the following. Furthermore, as both $c_1$ and $c_2$ are frozen coefficients, we set $m={c_2}/{c_1}$, which leads to
\begin{align*}
d_{2,2}  = &4a^2c_1^2[9(\xi_1 \eta_1 + \xi_2 \eta_2 m )^2  
+ 2 f_3(\xi_1)f_6(\xi_1,\xi_2) 
+ 2 f_3(\xi_2)f_6(\xi_2,\xi_1)m^2 ] =: 4 a^2 c_1^2 g(m).
\end{align*}
The function $g\left(m\right) $ can be rewritten as
\begin{align*}
g\left(m\right)  = & 
\eta_1^2 f_7\left(\xi_1,\xi_2 \right) m^2 
+ 18 \xi_1 \xi_2 \eta_1 \eta_2 m 
+ \eta_2^2 f_7\left(\xi_2,\xi_1\right)
\end{align*}
with 
$f_7\left(x,y\right)  =  4x^4y^2 -2x^2-y^2 +8 \geq -2x^2-y^2 +8 \geq 5.$
In the case $\eta_1=0$ we have $g(m)=\eta_2^2
f_7\left(\xi_2,\xi_1\right) \geq 0$ and thus $d_{2,2}\geq0$ and
$D_G\geq 0$. In the case
$\eta_1\neq 0$ we have $\eta_1^2 f_7(\xi_1 ,\xi_2)>0$, hence the function
$g\left(m\right)$ has a global minimum. This minimum is located at
\begin{align*}
\hat{m}  = & \frac{-9\xi_1 \xi_2 \eta_2 }{\eta_1 f_7\left(\xi_1 ,\xi_2 \right)},
\quad \text{which leads to }
\quad g\left(\hat{m}\right)  = 
\frac{2 \eta_1^2 f_5\left(\xi_1,\xi_2\right)f_8}{f_7\left(\xi_1,\xi_2\right)},
\end{align*}
where
$f_8  = 6\xi_1^2 \xi_2^2 + \xi_1^2 + \xi_2^2  - 2\xi_1^4 \xi_2^2 \eta_2^2 - 2\xi_1^2 \eta_1^2 \xi_2^4 - 8 \leq 0.$
Since $f_5\left(\xi_1 ,\xi_2\right)\leq 0$ we have $ g(m) \geq 0
\text{ for all } m\in \mathbb{R},$ and thus we have $D_G \geq 0$ for all cases as
$a < 0$. \\
We still need to show that $D_G>0$ for all $\xi_1,\xi_2\in [-1,1]$. It holds
$ d_0>0 $ for all $ (\xi_1,\xi_2) \in [-1,1]^2\setminus \{-1,1\}^2$
as $a<0$ and $k>0$. This leads to $D_G>0$ in these cases. For the case $(\xi_1,\xi_2) \in \{-1,1\}^2$ it holds $f_1\left(\xi_1,\xi_2\right) = 3$, which leads to $d_{4,0}=36 a^2>0$ and $D_G>0$.
Therefore, we have $D_G>0$ for all $(\xi_1,\xi_2) \in [-1,1]^2$ and
condition \eqref{general_stability_condition} is satisfied.\qquad
\end{proof}

For $b_{1,2}\neq 0$ the situation becomes much more
involved. Many additional terms appear in the expression for the
amplification factor $G$ and we face an additional degree of freedom
through $b_{1,2}$. Since we have proven condition
\eqref{general_stability_condition} holds for $b_{1,2}= 0$ it seems
reasonable to assume it also holds at least for values of $b_{1,2}$
close to zero. 
In von Neumann stability analysis it is often most difficult to guarantee that stability condition
\eqref{general_stability_condition} holds for extreme values of
$\eta_1$, $\eta_2$, $\xi_1$ and $\xi_2$. We have the following partial
result which holds in the case of frozen coefficients and non-vanishing coefficient of the mixed derivative, i.e.\ $b_{1,2}\neq 0$.
\begin{lemma}
\label{lem:genstab2d}
For $a=a_1=a_2<0$, arbitrary $b_{1,2}$ and $\Delta x_1 = \Delta x_2 = h>0$, the high-order compact scheme
\eqref{Def_general_n_dim_HOC_scheme_for_stability}
with the coefficients for the two-dimensional case defined in Section~\ref{Section_2DGeneralPDESemi_discrete} 
satisfies (for frozen coefficients) the stability condition \eqref{general_stability_condition} at the corner points $\xi_1=\pm 1$ and $\xi_2= \pm 1$.
\end{lemma}
\begin{proof}
Using $\eta_1 =\sin\left({z_1}/{2} \right) = \sqrt{1 -  \xi_1^2}=0$
for $\xi_1= \pm 1$ and $\eta_2=\sin\left({z_2}/{2} \right) = \sqrt{1 -
  \xi_2^2}=0$ for $\xi_2= \pm 1$, straight-forward computation shows
that on each corner point 
$\vert G\vert^2 - 1 =0$. Hence, condition
\eqref{general_stability_condition} holds.\qquad
\end{proof}

It is worth mentioning that in a comparable situation in
\cite{DuFo12p} (where a specific partial differential equation of type
\eqref{basic_general_pde_nD_without_usage_of_f} is considered) an additional
numerical evaluation of condition \eqref{general_stability_condition} revealed
it to hold also for non-vanishing mixed derivatives with $\left(\xi_1^2,\xi_2^2\right) \neq \left( 1,1\right)$. However, the 
left hand side of \eqref{general_stability_condition} was very close
to zero, and although the inequality was always satisfied, this left little room for analytical
estimates. This leads to the conjecture that the stability condition in
that case was satisfied also for
general parameters, although it would be hard to prove
analytically. Lemma~\ref{lem:genstab2d} above suggests the present case is similar.
We remark that in our numerical experiments we observe
a stable behaviour throughout, also for general choice of parameters. 

\subsection{Stability analysis for the three-dimensional case}
In this section we analyse the stability of the high-order compact
scheme with coefficients given in
Section~\ref{Section_3DGeneralPDESemi_discrete} in three space dimensions. 
We first perform a thorough von Neumann stability analysis in the case of vanishing cross derivative terms for frozen coefficients. 
We observe no additional stability condition in this case. 
Then we give partial results in the case of non-vanishing cross-derivative terms for frozen coefficients.
\begin{theorem}
For $a_i=a<0$, $b_{i,j}=0$ and $\Delta x_i =h>0$ for $i,j\in \{1,2,3\}$, $i\neq j$, the fully
discrete high-order
compact scheme given in
\eqref{Def_general_n_dim_HOC_scheme_for_stability} with $n=3$, with
coefficients given in Section~\ref{Section_3DGeneralPDESemi_discrete}, satisfies (for frozen coefficients) the necessary stability condition \eqref{general_stability_condition}.
\end{theorem}
\begin{proof}
Let $\xi_i=\cos({z_i/2})$ and $\eta_i=\sin(z_i/2)$ for $i=1,2,3$.
The stability condition
\eqref{general_stability_condition} yields
$\vert G \vert^2 - 1 ={N_G}/{D_G}$ (explicit expressions for $N_G$,
$D_G$ are given below).

For the numerator we have $N_G = - 8ak\left(n_4h^4 + n_2 h^2\right)
\leq 0,$ since $a <0$ and the polynomials
\begin{align*}
n_2  = &4a^2 f_1\left(\xi_1,\xi_2,\xi_3\right) 
\left[f_2\left(\xi_1,\xi_2\right) 
+ f_2\left(\xi_3,\xi_1\right) 
+ f_2\left(\xi_2 ,\xi_3\right) \right] \leq 0 ,\\
n_4  = & 
   \left[ f_3\left( \xi_1 , \xi_2 \right) + f_3\left( \xi_1 , \xi_3 \right)\right]  c_1^2
 + \left[ f_3\left( \xi_2 , \xi_1 \right) + f_3\left( \xi_2 , \xi_3 \right)\right]  c_2^2
 + \left[ f_3\left( \xi_3 , \xi_1 \right) + f_3\left( \xi_3 , \xi_2 \right)\right]  c_3^2  \\
& 
- \eta_3^2 \left( \xi_1 \eta_1  c_1 + \xi_2 \eta_2  c_2 \right)^2
- \eta_2^2 \left( \xi_1 \eta_1  c_1 + \xi_3 \eta_3  c_3 \right)^2 
- \eta_1^2 \left( \xi_2 \eta_2  c_2 + \xi_3 \eta_3  c_3 \right)^2\leq 0 ,
\end{align*}
are non-negative since
\begin{align*}
f_1\left(x,y\right)  = & x^2 +y^2 + z^2  \geq 0,\qquad
f_2\left(x,y\right)  =  2x^2y^2 - x^2 -1  \leq 0,\\
f_3\left( x,y\right) = & x^2y^2\left( 1-x^2\right) + y^2 \left(x^2 - 1 \right) \leq y^2\left( 1-x^2\right) + y^2 \left(x^2 - 1 \right) = 0,
\end{align*}
for $x,y,z \in \left[-1,1\right]$. 

The denominator $D_G$ can be written as
\begin{eqnarray}
\notag D_G = d_6h^6 + \left(d_{4,2}k^2 + d_{4,1}k+ d_{4,0}\right)h^4 +  \left(d_{2,2}k^2 + d_{2,1}k \right)h^2 + d_0,
\end{eqnarray}
where
\begin{align*}
d_0  = & 16 a^4 k^2 \left[
  m_1( \xi_1 ,\xi_2 ) 
+ m_1( \xi_3 ,\xi_1 ) 
+ m_1( \xi_2 ,\xi_3 ) 
\right]^2 \geq 0 , \quad
d_{2,1}  =  
 4an_2 \geq 0 ,
\\
d_{2,2}  = 
&  4a^2\left[  
  m_6\left( \xi_1 , \eta_1 , \xi_2 \right) c_1^2
+ 2m_7\left( \xi_3 \right) \xi_1 \xi_2 \eta_1 \eta_2 c_1 c_2
+ m_6\left( \xi_2 , \eta_2 , \xi_1 \right) c_2^2
\right.
\\
&
+ m_6\left( \xi_1 , \eta_1 , \xi_3 \right) c_1^2
+ 2m_7\left( \xi_2 \right) \xi_1 \xi_3 \eta_1 \eta_3 c_1 c_3
+ m_6\left( \xi_3 , \eta_3 , \xi_1 \right) c_3^2
\\
&
+ m_6\left( \xi_2 , \eta_2 , \xi_3 \right) c_2^2
+ 2m_7\left( \xi_1 \right) \xi_2 \xi_3 \eta_2 \eta_3 c_2 c_3
+ m_6\left( \xi_3 , \eta_3 , \xi_2 \right) c_3^2
\\
&
\left.  
+ m_5\left( \eta_1 , \xi_2 , \xi_3 \right) c_1^2
+ m_5\left( \eta_2 , \xi_1 , \xi_3 \right) c_2^2
+ m_5\left( \eta_3 , \xi_1 , \xi_2 \right) c_3^2
\right]
\\
d_{4,0}  = & 4 a^2 m_2\left( \xi_1 , \xi_2 , \xi_3 \right)^2 \geq 0,\;
d_{4,1}  =  4an_4 \geq 0 ,\;
d_6  = 
\left[
  \xi_1 \eta_1 c_1
+ \xi_2 \eta_2 c_2
+ \xi_3 \eta_3 c_3
\right]^2 \geq 0 ,
\\
d_{4,2}  = & \left[
  \eta_1^2 c_1^2
+ \eta_2^2 c_2^2
+ \eta_3^2 c_3^2
+ 2 \xi_1 \eta_1 \xi_2 \eta_2 c_1 c_2
+ 2 \xi_1 \eta_1 \xi_3 \eta_3 c_1 c_3 
+ 2 \xi_2 \eta_2 \xi_3 \eta_3 c_2 c_3 
\right]^2 \geq 0,
\end{align*}
since $a<0$ and
\begin{align*}
m_1\left(x,y\right)  = & 2x^2y^2 - x^2 -1 \leq x^2 - 1 \leq 0 ,\quad
m_2\left( x,y,z \right)  =  x^2 + y^2 + z^2  \geq 0 ,\\
m_3\left( x,y\right) = & x^2y^2\left( 1-x^2\right) + y^2 \left(x^2 - 1 \right) \leq y^2\left( 1-x^2\right) + y^2 \left(x^2 - 1 \right) = 0 ,\\
m_4\left(x,y\right)   =&(1 - x^2 ) [x^2(y^2-1)  +y^2(x^2-1)] \leq 0 , \\
m_5 \left(x,y,z\right) = &-8x^4 y^2z^2 + 4x^2 y^2z^2  + 4x^2 \geq  -8x^2 y^2z^2 + 4x^2 y^2z^2 + 4x^2 \\
 =& - 4x^2 y^2z^2  + 4x^2 \geq  - 4x^2  + 4x^2 =0, \\
m_6\left(x_1,x_2,y\right)  = & 4x_2^2 x_1^2y^4 + (-8x_2^2x_1^2 +
2x_2^2 )y^2 + x_2^2 + \frac{3}{2}x_1^2x_2^2 \in [ 0,3],\\
m_7\left( x\right) = &2x^2(x^2 - (1 - x^2)) + 7 \geq 0,
\end{align*}
for $x,y,z \in [-1,1]$. We still need to show $d_{2,2}\geq 0$. Since
we cannot determine the sign of $d_{2,2}$ directly, we consider three
different cases.

Having $\xi_2^2= \xi_3^2=1$ leads to
\begin{align*}
d_{2,2}  = &  4a^2 \left[2\bigl(-2.5\xi_1^2 \eta_1^2 + 3 \eta_1^2 \bigr) c_1^2+ \left(- 8 \eta_1^4 + 8 \eta_1^2 \right) c_1^2\right]
\geq 0
\end{align*}
as $\xi_1^2 \leq 1$ and $\eta_1^2\leq 1$. \\
Secondly, we consider $ c_1 = c_2 = c_3 =0$. This leads directly to $d_{2,2}=0$. \\
From now on we have $\left(c_1 ,c_2,c_3\right) \neq \left(0,0,0\right)$. Since $d_{2,2}$ is symmetric with respect to $c_1,c_2,c_3$, we assume without loss of generality $c_1 \neq 0$. Additionally, we have $\left(\xi_2^2, \xi_3^2 \right) \neq \big(1,1\big)$. Setting $p_2:= c_2/c_1$ and $p_3:= c_3/c_1$ gives
\begin{align*}
d_{2,2}  = 
&  4a^2c_1^2 \left[  
  m_6\left( \xi_1 , \eta_1 , \xi_2 \right) 
+ 2m_7\left( \xi_3 \right) \xi_1 \xi_2 \eta_1 \eta_2  p_2
+ m_6\left( \xi_2 , \eta_2 , \xi_1 \right) p_2^2
\right.
\\
&
+ m_6\left( \xi_1 , \eta_1 , \xi_3 \right) 
+ 2m_7\left( \xi_2 \right) \xi_1 \xi_3 \eta_1 \eta_3 p_3
+ m_6\left( \xi_3 , \eta_3 , \xi_1 \right) p_3^2
\\
&
+ m_6\left( \xi_2 , \eta_2 , \xi_3 \right) p_2^2
+ 2m_7\left( \xi_1 \right) \xi_2 \xi_3 \eta_2 \eta_3 p_2 p_3
+ m_6\left( \xi_3 , \eta_3 , \xi_2 \right) p_3^2
\\
&
\left.  
+ m_5\left( \eta_1 , \xi_2 , \xi_3 \right) 
+ m_5\left( \eta_2 , \xi_1 , \xi_3 \right) p_2^2
+ m_5\left( \eta_3 , \xi_1 , \xi_2 \right) p_3^2
\right]
\\
=:& 4a^2c_1^2\left[k_{11}p_2^2 + k_{22} p_3^2 + k_{12}p_2p_3 + k_{1} p_2 + k_2 p_3 + k_0 \right]=:4a^2c_1^2 g\left(p_2,p_3\right).
\end{align*}
To calculate the extremum  of $g\left(p_2,p_3\right)$,
\begin{align*}
 \nabla g\left(\hat{p}_2,\hat{p}_3\right)  =&
\begin{pmatrix}
 2k_{11}\hat{p}_2 + k_{12} \hat{p}_3 + k_1\\
 k_{12} \hat{p}_2 + 2k_{22} \hat{p}_3 +k_2
 \end{pmatrix}
  =  
\begin{pmatrix}
 0\\
 0
\end{pmatrix}
\end{align*}
is necessary, which leads to
\begin{align*}
\hat{p}_2 =& \frac{2k_1 k_{22} - k_2k_{12}}{k_{12}^2 - 4k_{11}^2k_{22}^2},  
\quad 
\hat{p}_3 = \frac{2k_2 k_{11} - k_1k_{12}}{k_{12}^2 - 4k_{11}^2k_{22}^2},\quad
\text{where } k_{12}^2 - 4k_{11}^2k_{22}^2 = q_1q_2q_3
\end{align*}
with
\begin{align*}
q_1 = &{{ \eta_2}}^{2}{{ \eta_3}}^{2},\quad
q_2= -2\,{{ \xi_1}}^{2}{{ \xi_2}}^{
2}-2\,{{ \xi_1}}^{2}{{\xi_3}}^{2}-2\,{{ \xi_2}}^{2}{{ \xi_3}}^{2}
+{{ \xi_1}}^{2}+{{ \xi_2}}^{2}+{{\xi_3}}^{2}+3 \in [0,4],\\
q_3=&
 8\,
{{\xi_1}}^{4}{{\xi_2}}^{2}{{\xi_3}}^{2}+4\,{{\xi_1}}^{2}{{
\xi_2}}^{4}{{\xi_3}}^{2}+4\,{{\xi_1}}^{2}{{\xi_2}}^{2}{{\xi_3}}^
{4}+4\,{{\xi_2}}^{4}{{\xi_3}}^{4}-4\,{{\xi_1}}^{4}{{\xi_2}}^{2
}
\\
&
-4\,{{\xi_1}}^{4}{{\xi_3}}^{2}-22\,{{\xi_1}}^{2}{{\xi_2}}^{2}
{{\xi_3}}^{2}-6\,{{\xi_2}}^{4}{{\xi_3}}^{2}-6\,{{\xi_2}}^{2}{{
\xi_3}}^{4}+8\,{{\xi_1}}^{2}{{\xi_2}}^{2}
\\
&
+8\,{{\xi_1}}^{2}{{
\xi_3}}^{2}+20\,{{\xi_2}}^{2}{{\xi_3}}^{2}-2\,{{\xi_1}}^{2}-3
\,{{\xi_2}}^{2}-3\,{{\xi_3}}^{2}-6  \in [-9,0] .
\end{align*}
It holds $q_1q_2q_3\neq 0$ for $\left(\xi_2^2, \xi_3^2\right) \neq (1,1)$. Since this is the unique root of $\nabla g$,  as $k_{11},k_{22}\geq 0$, we have a minimum at $(p_2,p_3) = (\hat{p}_2,\hat{p}_3)$. We obtain
$g\left(\hat{p}_2,\hat{p}_3\right)= {q_4 q_5 }/{q_6 },$
where
\begin{align*}
q_4  = & 2 \eta_1^2
\left( 
2\xi_1^2 \xi_2^2 
+ 2 \xi_1^2 \xi_3^2 
+ 2 \xi_2^2 \xi_3^2 
- \xi_1^2 
- \xi_2^2 
- \xi_3^2 
-3\right) \leq  
2 \eta_1^2 \left(
\xi_1^2 
+ \xi_2^2 
+ \xi_3^2 
-3 
\right)\leq 0
\\
q_5  =  & 
  8 \xi_1^4 \xi_2^4 \xi_3^2
+ 8 \xi_1^4 \xi_2^2 \xi_3^4
+ 8 \xi_1^2 \xi_2^4 \xi_3^4
- 4 \xi_1^4 \xi_2^4
- 20 \xi_1^4 \xi_2^2 \xi_3^2
- 4 \xi_1^4 \xi_3^4 
- 20 \xi_1^2 \xi_2^4 \xi_3^2
-20 \xi_1^2 \xi_2^2 \xi_3^4
\\
&
- 4 \xi_2^4 \xi_3^4
+ 6 \xi_2^2 \xi_1^4
+ 6 \xi_1^4 \xi_3^2
+ 6 \xi_1^2 \xi_2^4
+ 57 \xi_1^2 \xi_2^2 \xi_3^2
+ 6 \xi_1^2 \xi_3^4
+ 6 \xi_2^4 \xi_3^2
+ 6 \xi_2^2 \xi_3^4
\\
&
- 20 \xi_2^2 \xi_1^2
- 20 \xi_1^2 \xi_3^2
- 20 \xi_2^2 \xi_3^2
+ 3 \xi_1^2
+ 3 \xi_2^2
+ 3 \xi_3^2
+6 \in \left[ 0 , 9 \right],
\\
q_6  = &  
8 \xi_1^4 \xi_2^2 \xi_3^2
+ 4 \xi_1^2 \xi_2^4 \xi_3^2
+ 4 \xi_1^2 \xi_2^2 \xi_3^4
+ 4 \xi_2^4 \xi_3^4
- 4 \xi_2^2 \xi_1^4
- 4 \xi_1^4 \xi_3^2
- 22 \xi_1^2 \xi_2^2 \xi_3^2
\\
&
- 6 \xi_2^4 \xi_3^2
- 6 \xi_2^2 \xi_3^4
+ 8 \xi_2^2 \xi_1^2
+ 8 \xi_1^2 \xi_3^2
+ 20 \xi_2^2 \xi_3^2
- 2 \xi_1^2
- 3 \xi_2^2
- 3 \xi_3^2
- 6 \in \left[-9,0\right],
\end{align*}
with $q_6 \neq 0$ for $\left(\xi_2^2,\xi_3^2\right) \neq
(1,1)$. Hence, in all three cases we conclude $d_{2,2} \geq 0$, and
$D_G \geq 0$ holds. 

We still need to show that $D_G>0$ for all $\xi_1,\xi_2,\xi_3\in [-1,1]$. It holds 
$ d_0>0 $ for all $(\xi_1,\xi_2,\xi_3) \in [-1,1]^3\setminus \{-1,1\}^3$
as $a<0$ and $k>0$. This leads to $D_G>0$ in these cases. For the case $(\xi_1,\xi_2,\xi_3) \in \{-1,1\}^3$ we have $m_2\left(\xi_1,\xi_2,\xi_3\right) = 3$, which leads to $d_{4,0}=36 a^2>0$ 
and $D_G>0$.
Therefore, $D_G>0$ holds for all $(\xi_1,\xi_2,\xi_3) \in [-1,1]^3$ and condition \eqref{general_stability_condition} is satisfied.
\end{proof}

For the more general case with non-vanishing cross-derivatives we have the following
result. The comments made in the previous section also apply here.
\begin{lemma}
For $a_i=a<0$, $\Delta x_i =  h>0$ for $i=1,2,3$ and arbitrary $b_{1,2}$, $b_{1,3}$ and $b_{2,3}$, the high-order compact scheme
\eqref{Def_general_n_dim_HOC_scheme_for_stability} with the
coefficients for the three-dimensional case defined in
Section~\ref{Section_3DGeneralPDESemi_discrete} satisfies (for frozen coefficients) the stability condition \eqref{general_stability_condition} 
at the corner points $\xi_1 =\pm 1$, $\xi_2 = \pm 1$ and $\xi_3 = \pm 1$.
\end{lemma}
\begin{proof}
Using $\sin\left(z_1/2 \right) = \sqrt{1 - \xi_1^2}=0$ for $\xi_1 =
\pm 1$, $\sin\left(z_2/2 \right) = \sqrt{1 - \xi_2^2}=0$ for $\xi_2 =
\pm 1$ and $\sin\left(z_3/2 \right) = \sqrt{1 - \xi_3^2}=0$ for $\xi_3
= \pm 1$, straight-forward computation yields just as in the two-dimensional spatial setting to
$\vert G\vert^2 - 1  =0$ for all corner points. Hence,
condition \eqref{general_stability_condition} is satisfied.
\end{proof}

\section{Application to Black-Scholes Basket options}
\label{Section_Application}
To illustrate the practicality of the proposed scheme we now consider
the $n$-dimensional Black-Scholes option pricing PDE (see, e.g.\
\cite{Wil98}). In the option pricing problem mixed derivatives appear
naturally from correlation of the underlying assets.
After transformations, the conditions
\eqref{restriction_general_pde_nD_number_one} are satisfied, and we give the coefficients of the resulting
scheme. Then we discuss the boundary conditions as well as the time
discretisation.

\subsection{Transformation of the $n$-dimensional Black-Scholes
  equation}
\label{section_transformation_of_the_n_dim_BLS_eq}
In the multidimensional Black Scholes model the asset prices follow a geometric Brownian motion,
\begin{eqnarray}\label{sde_for_stock}
 dS_i(t) = (\mu_i - \delta_i) S_i(t) dt +\sigma_i S_i(t) dW_i(t) ,
\end{eqnarray}
where $S_i$ is the $i$-th underlying asset which has an expected
return of $\mu_i$, a continuous dividend of $\delta_i$, and the volatility $\sigma_i$ for $i=1, \ldots , n$ and $n\in \mathbb{N}$. The Wiener processes are correlated with $\langle {\it d}W_i,{\it d}W_j \rangle=: \rho_{i,j}dt$ for $i ,j=1, \ldots , n$ with $i \neq j$. Application of It$\hat{\text{o}}$'s lemma and standard arbitrage arguments show that any option price $V(S,\sigma,t)$ solves the $n$-dimensional Black-Scholes partial differential equation,
\begin{align}\label{basicpdemutliblsBasket}
\frac{\partial V}{\partial t}+
\frac{1}{2}\sum\limits_{i=1}^{n}\sigma_i^2 S_i^2 \frac{\partial^2
  V}{\partial S_i^2} + \sum\limits_{\substack{i,j=1\\
    i<j}}^{n}\rho_{ij}\sigma_i \sigma_j S_i S_j \frac{\partial^2
  V}{\partial S_i \partial S_j} + \sum\limits_{i=1}^n \eta_i S_i \frac{\partial V}{\partial S_i} -rV  =&0,
\end{align}
where $\eta_i=r-\delta_i.$
The transformations
\begin{align}\label{transformation_BLS_Basket_pde}
 x_i =& {\gamma}\ln\left({S_i}/{K}\right)/{\sigma_i}, \quad \tau  =  T-t \quad \text{ and } \quad u=e^{r\tau}{V}/{K},
\end{align}
for $i=1, \ldots , n$, where $\gamma$ is a constant scaling parameter to assure that the resulting computational domain does not get too large, 
leads to
\begin{align}\label{usedpdemutliblsBasket}
u_{\tau} -\frac{\gamma^2}{2}\sum\limits_{i=1}^n \frac{\partial^2 u}{\partial x_i^2} - \gamma^2 \sum\limits_{\substack{i,j=1\\i<j}}^n \rho_{ij} \frac{\partial^2 u}{\partial x_i \partial x_j} + \gamma \sum\limits_{i=1}^n \varsigma_i\frac{\partial u}{\partial x_i} = & 0,
\end{align}
where $\varsigma_i={\sigma_i}/{2}-\eta_i/{\sigma_i}.$
Comparing this equation with \eqref{basic_general_pde_nD_without_usage_of_f}, we identify
\begin{equation}\label{pde_coefficients_n_dim_bls}
a_i= -\frac{\gamma^2}{2}, \quad b_{ij}= -\gamma^2 \rho_{ij},  \quad
 c_i=\gamma \varsigma_i , \quad g= 0, 
\end{equation}
for $i,j = 1, \ldots ,n$ and $i<j$. We find that the transformed
partial differential equation \eqref{usedpdemutliblsBasket} with these
coefficients satisfies the conditions given by
\eqref{restriction_general_pde_nD_number_one}, if $\Delta x_i=h$ for a
step size $h>0$ is used. 
Hence, we are able to obtain a high-order compact scheme in any
spatial dimension $n \in \mathbb{N}$.

We consider a European Power-Put Basket option, thus the final condition for \eqref{basicpdemutliblsBasket} is given by
\begin{align*}
V(S_1, \ldots ,S_n,T)=&\max\biggl(K-\sum\limits_{i=1}^n \omega_i S_i ,0 \biggr)^p ,
\end{align*}
where $p$ is an integer and the asset weights satisfy
$\sum\limits_{i=1}^n \omega_i=1$. 
Applying the transformations \eqref{transformation_BLS_Basket_pde}
leads to the initial condition
\begin{align}\label{general_Initial_Cond_Bls_Basket}
u(x_1, \ldots ,x_n,0)=&K^{p - 1}\max\biggl(1-\sum\limits_{i=1}^n \omega_i e^{\frac{\sigma_i x_i}{\gamma}} ,0 \biggr)^p.
\end{align}

\subsection{Semi-discrete two-dimensional Black-Scholes equation}\label{Section_2D_BLS_PDE_Semi_discrete}
In this section we apply our general two-dimensional semi-discrete
scheme, see Section~\ref{Section_2DGeneralPDESemi_discrete}, to the
two-dimensional Black-Scholes model. To obtain the semi-discrete
scheme \eqref{semi_discrete_pde_two_dimensions_general} we have to
apply \eqref{pde_coefficients_n_dim_bls} with $n=2$ to the
coefficients in Section~\ref{Section_2DGeneralPDESemi_discrete}, which gives
\begin{align*}
\hat{K}_{i_1,i_2}  = & {\frac {{\gamma}^{2}(5-2{\rho^{2}_{{12}}})}{3{h}^{2}}}
+\frac{\varsigma_1 ^{2}+\varsigma_2^{2}}{3},\;
\hat{K}_{i_1\pm 1,i_2}  =  {\frac {{\gamma}^{2}{\rho^{2}_{{12}}}}{3{h}^{2}}}
\pm \frac{\gamma\varsigma_1 }{3h}
\mp \frac{\gamma\varsigma_2\rho_{{12}}}{3h}
- \frac{\varsigma_1 ^{2}}{6}
-{\frac {{\gamma}^{2}}{3{h}^{2}}},
\\
\hat{K}_{i_1,i_2\pm 1} = & {\frac {{\gamma}^{2}{\rho^{2}_{{12}}}}{3{h}^{2}}}
\pm \frac{\gamma\varsigma_2} {3h}
\mp \frac{\gamma\varsigma_1\rho_{{12}} } {3h}
- \frac{ \varsigma_2 ^{2}}{6}
- \,{\frac {{\gamma}^{2}}{3{h}^{2}}},
\\
\hat{K}_{i_1\pm 1,i_2-1}  = & \pm \frac{ \varsigma_2 \varsigma_1 }{12}
- \frac{\gamma\varsigma_2 }{12h}
\pm \frac{\gamma\varsigma_1 }{12h}
- \frac{\gamma\varsigma_1\rho_{{12}} }{6h}
\pm \frac{\gamma\varsigma_2\rho_{{12}}}{6h}
- {\frac {{\gamma}^{2}}{12{h}^{2}}}
\pm {\frac {{\gamma}^{2}\rho_{{12}}}{4{h}^{2}}}
- {\frac {{\gamma}^{2}{\rho^{2}_{{12}}}}{6{h}^{2}}} ,
\\
\hat{K}_{i_1\pm 1,i_2+1}  = &  \frac{\gamma\varsigma_2 }{12h}
\mp \frac{ \varsigma_2\varsigma_1 }{12}
\pm \frac{\gamma\varsigma_1 }{12h}
+ \frac{\gamma\rho_{{12}}\varsigma_1 }{6h}
\pm \frac{\gamma\varsigma_2\rho_{{12}}}{6h}
- {\frac {{\gamma}^{2}}{12{h}^{2}}}
\mp {\frac {{\gamma}^{2}\rho_{{12}}}{4{h}^{2}}}
- {\frac {{\gamma}^{2}{\rho^{2}_{{12}}}}{6{h}^{2}}},
\end{align*}
where $\hat{K}_{l,m}$ is the coefficient of $U_{l,m}\left(\tau\right)$ for $l \in \{i_1 - 1, i_1, i_1 + 1\}$ and $m \in \{i_2 - 1, i_2, i_2 + 1\}$. 
The coefficients of $\partial_{\tau}U_{l,m}\left(\tau\right)$ are given by
\begin{align*}
M_{i_1,i_2}  =& \frac{2}{3} , & M_{i_1+1, i_2\pm 1} = & M_{i_1-1, i_2\mp 1}  = \pm \frac{\rho_{{12}}}{24} , \\
M_{i_1 \pm 1,i_2}  = & \frac{1}{12}
\mp \frac{h \varsigma_1 }{12\gamma} , & M_{i_1,i_2\pm1}  = & \frac{1}{12}
\mp \frac{h \varsigma_2 }{12\gamma} .
\end{align*}
Additionally, it holds $\tilde{g}(x,\tau) = 0$. 
This gives a semi-discrete scheme of the form \eqref{semi_discrete_pde_two_dimensions_general}, where $K_x$ and $M_x$ are time-independent. 
As in \ref{sec:hoctime} we apply Crank-Nicolson type time discretisation and obtain the fully discrete scheme for the spatial interior.

\subsection{Semi-discrete three-dimensional Black-Scholes equation}\label{Section_3D_BLS_PDE_Semi_discrete}
In this section we give the semi-discrete scheme
\eqref{semi_discrete_pde_two_dimensions_general} for the
three-dimensional Black-Scholes Basket option. Using
\eqref{pde_coefficients_n_dim_bls} with  $n=3$ in
Section~\ref{Section_2DGeneralPDESemi_discrete} and the appendix
we obtain the coefficients $\hat{K}_{k,l,m}$ of
$U_{k,l,m}\left(\tau\right)$ for $k\in \{ i_1 - 1, i_1, i_1 + 1\}$,
$l\in \{ i_2 - 1, i_2, i_2 + 1\}$ and $m\in \{ i_3 - 1, i_3, i_3 +
1\}$, which are
\begin{align*}
\hat{K}_{i_1,i_2,i_3}  = &
\frac{\varsigma_1 ^{2}}{3}
 +\frac{\varsigma_2 ^{2}}{3}
+\frac{\varsigma_3 ^{2}}{3}
-{\frac {2{\gamma}^{2}{\rho^{2}_{{12}}}}{3{h}^{2}}}
-{\frac {2{\gamma}^{2}{\rho^{2}_{{13}}}}{3{h}^{2}}}
-{\frac {2{\gamma}^{2}{\rho^{2}_{{23}}}}{3{h}^{2}}}
+{\frac {2{\gamma}^{2}}{{h}^{2}}} ,
\\
\hat{K}_{i_1 \pm 1,i_2,i_3}  = & 
 \pm \frac{\gamma\varsigma_1 }{6h}
- \frac{\varsigma_1 ^{2} }{6}
 \mp \frac{\gamma \rho_{{12}}\varsigma_2 }{3h}
+ {\frac {{\gamma}^{2}{\rho^{2}_{{12}}}}{3{h}^{2}}}
- {\frac {{\gamma}^{2}}{6{h}^{2}}} 
  \mp \frac{\gamma\rho_{{13}} \varsigma_3 }{3h}
+ {\frac {{\gamma}^{2}{\rho^{2}_{{13}}}}{3{h}^{2}}},
\\
\hat{K}_{i_1,i_2\pm 1,i_3}  = & 
 \pm \frac{\gamma\varsigma_2 }{6h}
- \frac{\varsigma_2 ^{2}}{6}
 \mp \frac{\gamma\varsigma_1 }{3h}
 + {\frac {{\gamma}^{2}{\rho^{2}_{{12}}}}{3{h}^{2}}}
 - {\frac {{\gamma}^{2}}{6{h}^{2}}} 
 \mp \frac{\gamma\rho_{{23}}\varsigma_3 }{3h}
 + {\frac {{\gamma}^{2}{\rho^{2}_{{23}}}}{3{h}^{2}}},
\\
\hat{K}_{i_1,i_2,i_3\pm 1}  = & 
\pm\frac{\gamma\varsigma_3 }{6h}
- \frac{\varsigma_3 ^{2}}{6}
 \mp\frac{\gamma\rho_{{13}} }{3h}
+ {\frac {{\gamma}^{2}{\rho^{2}_{{13}}}}{3{h}^{2}}}
- {\frac {{\gamma}^{2}}{6{h}^{2}}} 
 \mp\frac{\gamma\rho_{{23}}\varsigma_2 }{3h}
+ {\frac {{\gamma}^{2}{\rho^{2}_{{23}}}}{3{h}^{2}}},
\\
\hat{K}_{i_1\pm 1,i_2-1,i_3} = & 
-\gamma   \frac{\varsigma_2\mp \varsigma_1}{12h}
 \pm \frac{ \varsigma_1  \varsigma_2 }{12}
-  {\frac {{\gamma}^{2}}{12{h}^{2}}} 
-\gamma \rho_{{12}}  \frac{\varsigma_1 \mp \varsigma_2  }{6h}
-  {\gamma}^{2}{\frac {{\rho^{2}_{{12}}}\mp  \rho_{{12}} \pm \rho_{{13}}\rho_{{23}} }{6{h}^{2}}},\\
\hat{K}_{i_1\pm 1,i_2+1,i_3}  = & 
\gamma \frac{\varsigma_2\pm \varsigma_1}{12h}
  \mp \frac{ \varsigma_1  \varsigma_2 }{12}
-  {\frac {{\gamma}^{2}}{12{h}^{2}}} 
+\gamma \rho_{{12}}  \frac{\varsigma_1\pm \varsigma_2 }{6h}
- {\gamma}^{2} {\frac {{\rho^{2}_{{12}}}\pm \rho_{{12}}  \mp \rho_{{13}}\rho_{{23}} }{6{h}^{2}}},\\
\hat{K}_{i_1\pm 1,i_2,i_3-1}  = & 
- \gamma  \frac{\varsigma_3\mp \varsigma_1}{12h}
 \pm \frac{ \varsigma_1  \varsigma_3 }{12}
-  {\frac {{\gamma}^{2}}{12{h}^{2}}} 
- \gamma \rho_{{13}} \frac{\varsigma_1\mp \varsigma_3}{6h}
 -{\gamma}^{2}  {\frac {{\rho^{2}_{{13}}}\mp \rho_{{13}} \pm \rho_{{12}}\rho_{{23}}}{6{h}^{2}}},\\
\hat{K}_{i_1\pm 1,i_2,i_3+1}  = & 
 \gamma \frac{\varsigma_3\pm \varsigma_1 }{12h}
 \mp\frac{ \varsigma_1  \varsigma_3}{12} 
-  {\frac {{\gamma}^{2}}{12{h}^{2}}}
+ \gamma \rho_{{13}} \frac{\varsigma_1\pm \varsigma_3 }{6h}
- {\gamma}^{2} {\frac {{\rho^{2}_{{13}}} \pm \rho_{{13}}\mp   \rho_{{12}}\rho_{{23}} }{6{h}^{2}}},\\
\hat{K}_{i_1,i_2 \pm 1,i_3-1}  = & 
- \gamma \frac{\varsigma_3 \mp\varsigma_2}{12h}
\pm \frac{ \varsigma_2  \varsigma_3 }{12}
-  {\frac {{\gamma}^{2}}{12{h}^{2}}} 
- \gamma \rho_{{23}} \frac{\varsigma_2\mp \varsigma_3 }{6h}
- {\gamma}^{2} {\frac {{\rho^{2}_{{23}}}\mp \rho_{{23}} \pm \rho_{{12}}\rho_{{13}}}{6{h}^{2}}},\\
\hat{K}_{i_1,i_2\pm 1,i_3+1}  = & 
 \gamma  \frac{\varsigma_3\pm \varsigma_2}{12h}
 \mp \frac{ \varsigma_2  \varsigma_3}{12} 
-  {\frac {{\gamma}^{2}}{12{h}^{2}}} 
+ \gamma \rho_{{23}} \frac{\varsigma_2 \pm \varsigma_3}{6h}
- {\gamma}^{2} {\frac {{\rho^{2}_{{23}}} \pm \rho_{{23}} \mp \rho_{{12}}\rho_{{13}} }{6{h}^{2}}},\\
\hat{K}_{i_1 \pm 1,i_2-1,i_3-1}  = & 
 \pm \gamma \frac{\rho_{{23}} \varsigma_1+\rho_{{13}} \varsigma_2+\rho_{{12}} \varsigma_3 }{24h}
-  {\gamma}^{2}{\frac {\rho_{{23}}\mp \rho_{{12}} \mp \rho_{{13}} }{24{h}^{2}}}
- {\gamma}^{2} {\frac {\rho_{{12}}\rho_{{13}}\mp \rho_{{12}}\rho_{{23}} \mp \rho_{{13}}\rho_{{23}}}{12{h}^{2}}} ,
\\
\hat{K}_{i_1\pm 1,i_2+1,i_3-1}  = & 
 \mp \gamma \frac{\rho_{{23}} \varsigma_1 + \rho_{{13}} \varsigma_2+\rho_{{12}} \varsigma_3}{24h}
 +{\gamma}^{2}  {\frac {\rho_{{23}}\mp  \rho_{{12}}\pm
     \rho_{{13}}}{24{h}^{2}}}
+ {\gamma}^{2} {\frac {\rho_{{12}}\rho_{{13}}\pm \rho_{{12}}\rho_{{23}}\mp \rho_{{13}}\rho_{{23}}}{12{h}^{2}}} ,
\\
\hat{K}_{i_1\pm 1,i_2-1,i_3+1}  = & 
\mp \gamma  \frac{\rho_{{23}} \varsigma_1+\rho_{{13}} \varsigma_2 +\rho_{{12}} \varsigma_3  }{24h}
+{\gamma}^{2}  {\frac {\rho_{{23}}\pm  \rho_{{12}}\mp
    \rho_{{13}}}{24{h}^{2}}} 
+{\gamma}^{2}  {\frac {\rho_{{13}}\rho_{{23}}\mp \rho_{{12}}\rho_{{23}}\pm \rho_{{12}}\rho_{{13}}}{12{h}^{2}}},
\\
\hat{K}_{i_1\pm 1,i_2+1,i_3+1}  = & 
\pm \gamma \frac{\rho_{{23}} \varsigma_1+   \rho_{{13}} \varsigma_2+ \rho_{{12}} \varsigma_3  }{24h}
- {\gamma}^{2} {\frac {\rho_{{23}}\pm \rho_{{12}} \pm
    \rho_{{13}}}{24{h}^{2}}}
- {\gamma}^{2} {\frac {\rho_{{12}}\rho_{{23}}\pm \rho_{{12}}\rho_{{13}}\pm \rho_{{13}}\rho_{{23}}}{12{h}^{2}}}.
\end{align*}
Similarly, we get the coefficients $\hat{M}_{k,l,m}$ of
$\partial_{\tau} U_{k,l,m}\left(\tau\right)$, given by
\begin{align*}
\hat{M}_{i\pm 1,j,m-1}  = & \hat{M}_{i\mp 1,j,m+1} = \mp \frac{\rho_{{13}}}{24}, & \hat{M}_{i,j \pm 1,m-1}  = & \hat{M}_{i,j\mp 1,m+1} =\mp \frac{\rho_{{23}}}{24}  ,
\\
\hat{M}_{i\pm 1,j-1,m}  = & \hat{M}_{i \mp 1 ,j+1,m }  = \mp \frac{\rho_{{12}}}{24}, &\hat{M}_{i\pm 1 , j , m}  = & \frac{1}{12} \mp\frac{h \varsigma_1}{12\gamma} , 
\\ 
\hat{M}_{i,j\pm 1,m}  = & \frac{1}{12} \mp\frac{h \varsigma_2 }{12\gamma} , &\hat{M}_{i,j,m \pm 1}  = & \frac{1}{12}  \mp\frac{h \varsigma_3} {12\gamma}, \quad \hat{M}_{i,j,m}  =   \frac{1}{2} ,
\\
\hat{M}_{i \pm 1,j-1,m+1}  = &\hat{M}_{i\pm 1,j+1,m+1}=0 & \hat{M}_{i \pm 1,j-1,m-1}  = & \hat{M}_{i\pm 1,j+1,m-1}=0 .
\end{align*}
Additionally, we have $\tilde{g}(x,\tau)  =  0$. We obtain a semi-discrete scheme of the form \eqref{semi_discrete_pde_two_dimensions_general}, where $K_x$ and $M_x$ are time-independent.
As in \ref{sec:hoctime} we apply Crank-Nicolson type time discretisation and obtain the fully discrete scheme for the spatial interior.

\subsection{Treatment of the boundary conditions}\label{boundary_conditions_multi_dim_bls}
After deriving a high-order compact scheme for the spatial interior we
now discuss the boundary conditions.
 
\subsubsection{Lower boundaries}\label{Lower_boundaries_n_dim_BLS_Basket}
The first boundary we discuss is $S_i=0$ for some $i\in I \subset \{1,
\ldots n\}$ at time $t\in [0,T[$. Once the value of the asset is zero,
it stays constant over time, see \eqref{sde_for_stock}. Hence, using
$S_i=0$ for $i\in I$ in \eqref{basicpdemutliblsBasket} and applying
the transformation \eqref{transformation_BLS_Basket_pde} leads to
\begin{align*}
-\frac{\gamma^2}{2}\sum\limits_{\substack{i=1\\i \notin I}}^n \frac{\partial^2 u}{\partial x_i^2} - \gamma^2 \sum\limits_{\substack{i,j=1\\ i,j \notin I\\i<j}}^n \rho_{ij} \frac{\partial^2 u}{\partial x_i \partial x_j} + \gamma\sum\limits_{\substack{i=1\\i \notin I}}^n \varsigma_i\frac{\partial u}{\partial x_i} = & f,
\end{align*}
with $f=-u_\tau.$
Hence, at these
boundaries we are able to obtain high-order compact schemes in the same manner as shown for the spatial interior with then $n - \vert I \vert$ spatial dimensions, 
as the coefficients of the partial differential equations of these boundaries satisfy condition \eqref{restriction_general_pde_nD_number_one}. 
The case $I=\left\{1, \ldots, n\right\}$, i.e.\ $\vert I \vert =n$, leads to the Dirichlet boundary condition
$u(x_{\min}^{(1)},\ldots , x_{\min}^{(n)},\tau)=u(x_{\min}^{(1)},\ldots , x_{\min}^{(n)},0)$
at time $\tau\in ]0,\tau_{\max}]$, since in that case $u_\tau=0$.

\subsubsection{Upper boundaries} \label{Sec_Upper_bound_BLS_Basket}
Upper boundaries are boundaries with $S_i=S_i^{\max}$ for some $i \in J \subset \left\{1, \ldots, n\right\}$ at time $t\in [0,T[$. For a sufficiently large $S_i^{\max}$ for $i \in J$, we set
\begin{align*}
\left. \frac{\partial V\left(S_1, \ldots,S_n, t \right) }{\partial S_i}\right|_{S_i = S_i^{\max}} \equiv& 0,
\end{align*}
with $S_k \in \left[S_k^{\min} , S_k^{\max}\right]$ for $k = \{1, \ldots , n\} \setminus \{i\}$ for a European Power Put Basket option. Employing this in \eqref{basicpdemutliblsBasket} and
using the transformations \eqref{transformation_BLS_Basket_pde}, yields 
\begin{align}\label{usedpdemutliblsBasketboundaryx_i_max}
-\frac{\gamma^2}{2}\sum\limits_{\substack{i=1\\i \notin J}}^n \frac{\partial^2 u}{\partial x_i^2} - \gamma^2 \sum\limits_{\substack{i,j=1\\ i,j \notin J\\i<j}}^n \rho_{ij} \frac{\partial^2 u}{\partial x_i \partial x_j} + \gamma \sum\limits_{\substack{i=1\\i \notin J}}^n \varsigma_i\frac{\partial u}{\partial x_i} =&f,
\end{align}
with $f=-u_\tau$.
Hence the upper boundaries show the same behaviour as the lower boundaries for a European Power Put Basket.
Analogously, we have the Dirichlet boundary condition
$u(x_{\max}^{(1)},\ldots , x_{\max}^{(n)}, \tau)  =  u(x_{\max}^{(1)},\ldots , x_{\max}^{(n)}, 0)$
for $\tau \in ]0, \tau_{\max}]$ if $J=\{1, \ldots , n\}$.
\subsection{Combination of upper and lower boundaries} 
A combination of upper and lower boundaries thus behaves in the same manner and the resulting partial differential equations with $n-\vert I \vert -\vert J \vert$ spatial dimensions satisfy condition \eqref{restriction_general_pde_nD_number_one} as well. 
For the corner points of $\Omega$ we have $\vert I\vert + \vert J \vert =n$ and thus again $u=u_0$.

\section{Numerical experiments for Black-Scholes Basket options}
\label{sec:numexp}
In this section we discuss the numerical experiments for the
Black-Scholes Basket Power Puts in spatial dimensions $n=2,3$. The
equation systems which have to be solved over time have been derived
in Section \ref{Section_Application}. According to \cite{KrThWi70}, we
cannot expect fourth-order convergence if the initial condition is not
sufficiently smooth. Hence, we have to smooth the initial
condition for Power Puts with $p=1,2$. In \cite{KrThWi70} suitable
smoothing operators are identified in Fourier space. Since the order
of convergence of our high-order compact scheme is four, we use the
smoothing operator $\Phi_4$, given by its Fourier transform
$$
\hat{\Phi}_4 (\omega)= \biggl(\frac{\sin({\omega}/{2}) }{{\omega}/{2}}\biggr)^4 \left[ 1 + \frac{2}{3}\sin^2({\omega}/{2} )\right].
$$
This leads to the smoothed initial condition 
$$
\tilde{u}_0\left(x_1,x_2\right) = \frac{1}{h^2} \int\limits_{-3h}^{3h} \int\limits_{-3h}^{3h}\Phi_4 \left(\frac{x}{h}\right)\Phi_4 \left(\frac{y}{h}\right) u_0\left(x_1-x,x_2-y\right) \text{d}x\, \text{d}y,
$$
in the case $n=2$
for any step size $h>0$, where $u_0$ is the original initial condition and $\Phi_4(x)$ denotes the Fourier inverse of $\hat{\Phi}_4(\omega)$, see \cite{KrThWi70}. If $u_0$ is smooth enough in the integrated region around $\left(x_1, \ldots, x_n\right)$, we have $\tilde{u}_0 \left(x_1, \ldots, x_n\right) = u_0\left(x_1, \ldots, x_n\right)$. That means that it is possible to identify the points where smoothing is necessary. 

\begin{figure}[h]\centering
  \includegraphics[width=11cm,height=3.5cm]{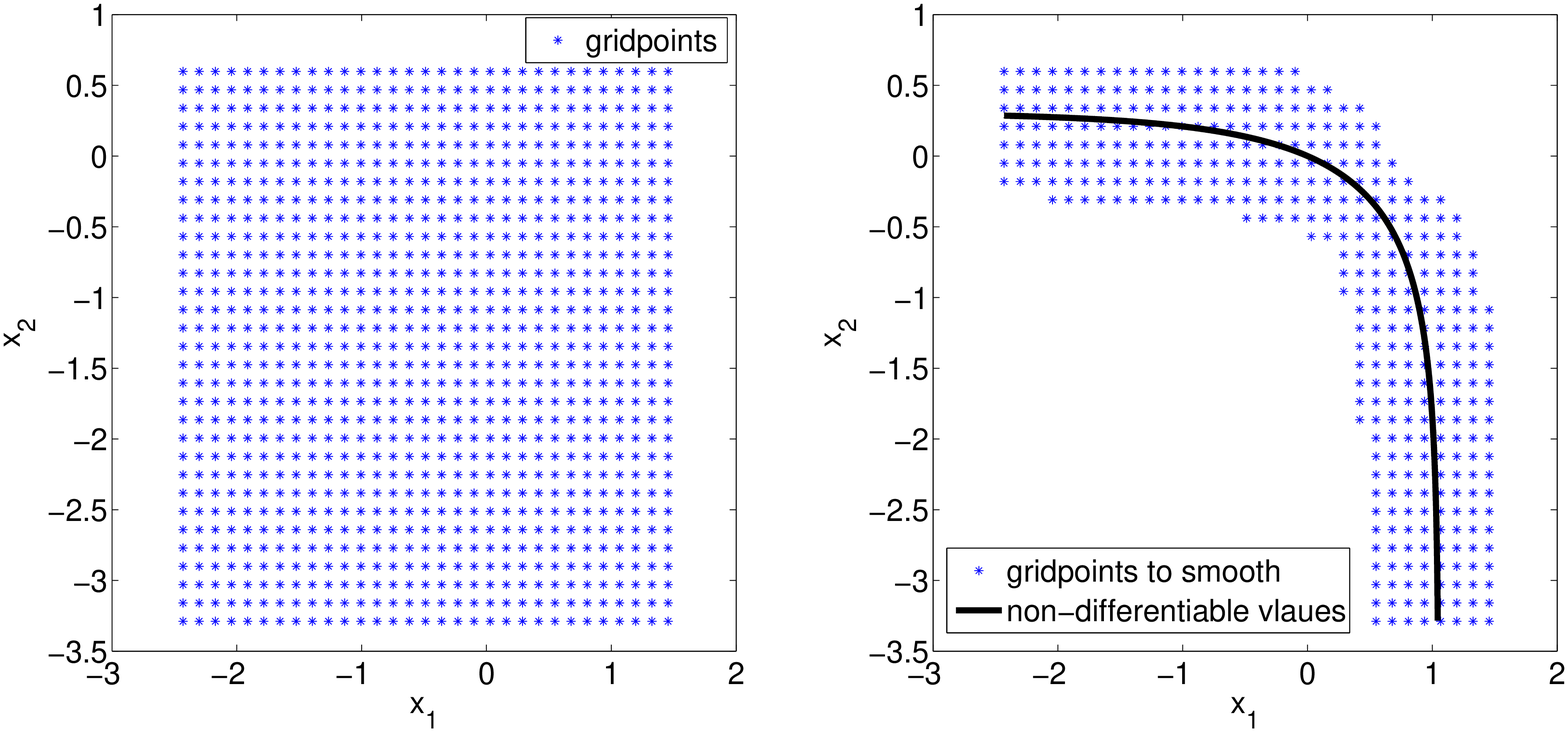}
  \caption{Example of grid points selected for the smoothing procedure in two
    space dimensions. We employ the smoothing operators of Kreiss et al.\
    \cite{KrThWi70} to ensure high-order convergence of the approximations of the smoothed problem to the true solution of \eqref{usedpdemutliblsBasket}.}
  \label{fig:smooting_points}
\end{figure}
Figure \ref{fig:smooting_points} shows an example of a two-dimensional
grid on the left side and on the right side a graph of the non-differentiable
points of the initial condition given in
\eqref{general_Initial_Cond_Bls_Basket} together with the identified
grid points,
where smoothing is necessary. The points are chosen in
such a way that we ensure that the non-differentiable points have
no influence on $\tilde{u}_0\left(x_1,x_2\right) $ for those points,
which are not shown in Figure \ref{fig:smooting_points} on the right
hand side. This approach reduces the
necessary calculations significantly. As $h
\rightarrow 0$, the smooth initial condition $\tilde{u}_0$ converges
towards the original initial condition $u_0$ given in
\eqref{general_Initial_Cond_Bls_Basket}. The results in \cite{KrThWi70}
guarantee high-order convergence of the approximation of the smoothed problem to the true solution of \eqref{usedpdemutliblsBasket}.

We use the relative $l^2$-error $\Vert U_{\text{ref}} -
U\Vert_{l^2}/\Vert U_{\text{ref}}\Vert_{l^2},$ as well as the $l^{\infty}$-error $\Vert
U_{\text{ref}} - U\Vert_{l^{\infty}}$ to
examine the numerical convergence rate, where $U_{\text{ref}}$
denotes a reference solution on a fine grid and $U$ is the
approximation. When identifying the convergence order of the schemes,
we determine it as the slope of the linear least square fit of the
individual error points in the loglog-plots of error versus number of
grid points per spatial direction.

\subsection{Numerical example with two underlying assets}\label{sec:num_example_2D_const_coeff}
In this section we report the numerical results for a two-dimensional
Black-Scholes Basket Power Put. We compare the high-order compact
scheme (`HOC') with the standard scheme (`2nd order'), which is obtained by using the
central difference operator directly in \eqref{usedpdemutliblsBasket}
for $n=2$ with no further action and thus leads to a classical second-order
scheme. We consider plain European Puts ($p=1$) and use the smoothing procedure outlined above for the initial condition
\eqref{general_Initial_Cond_Bls_Basket}. The parameter values
$$ \sigma_1 = 0.25, \quad \sigma_2 = 0.35,\quad \gamma = 0.25, \quad r = \ln(1.05), \quad \omega_1 = 0.35 = 1-\omega_2, \quad K=10, $$
and $\delta_1=\delta_2 = 0$ are used, unless stated otherwise. 
The parabolic mesh ratio is fixed to $\Delta \tau/h^2=0.4$,
although we point out that neither the von Neumann stability analysis
nor our numerical experiments revealed any practical restrictions on
its choice. 

\begin{figure}[h]\centering
      \includegraphics[width=6cm,height=3.5cm]{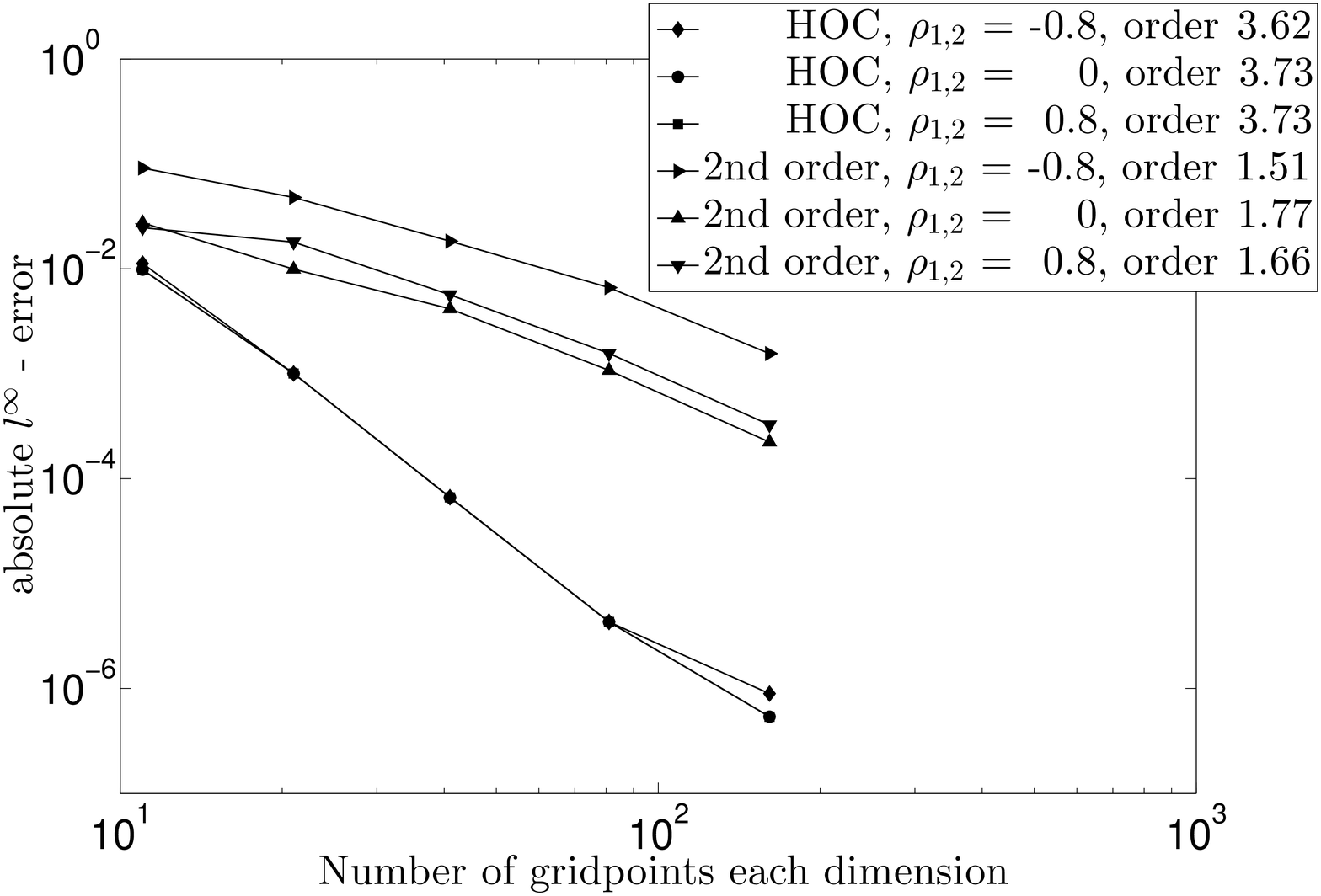}%
      \includegraphics[width=6cm,height=3.5cm]{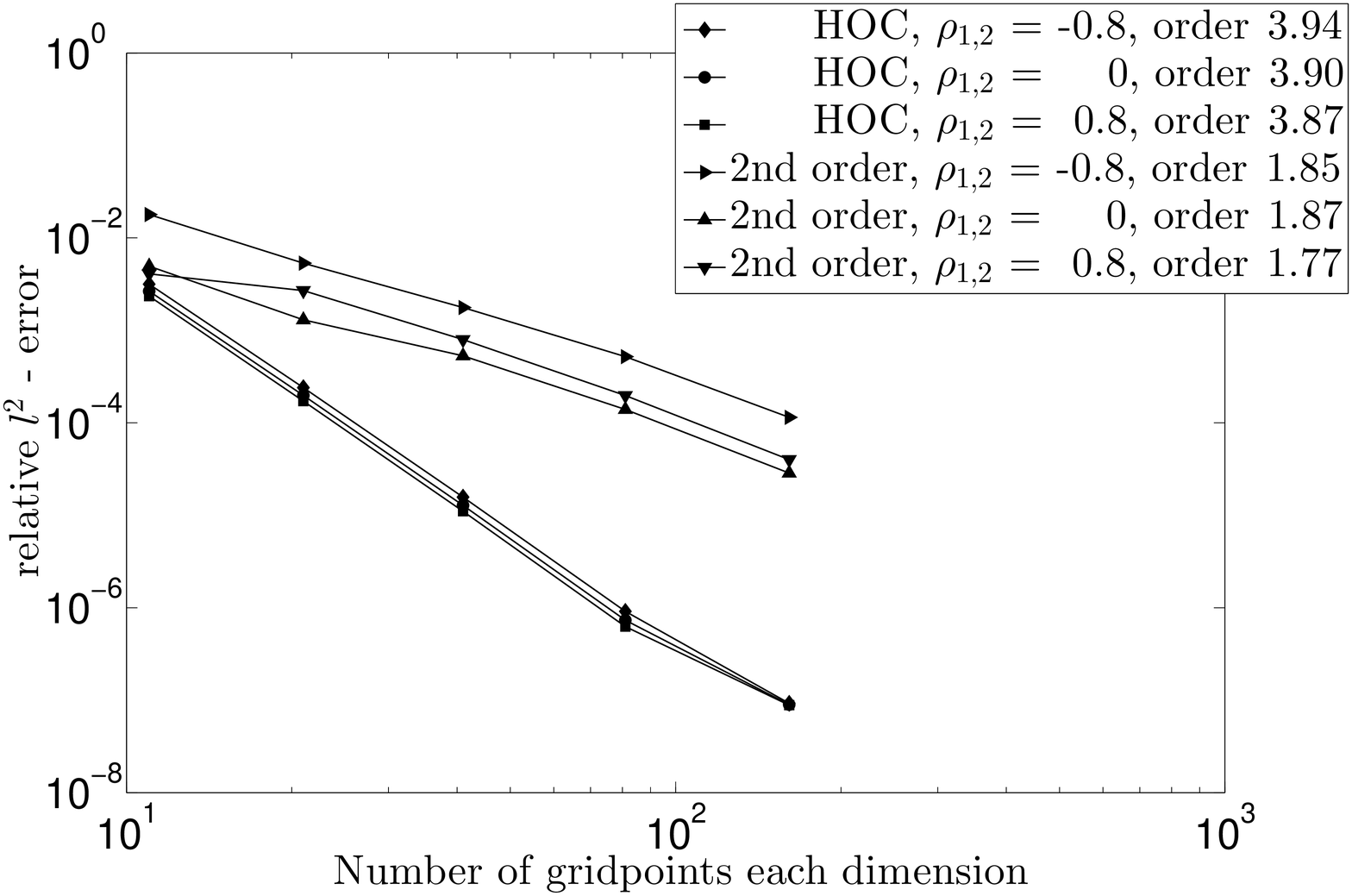}
      \caption{$l^{\infty}$- (left) and relative $l^2$-error (right)
        for two-dimensional Black-Scholes Basket Put and smoothed initial condition.}
      \label{fig:loglog_plot_L_2_error_Power_Basket_power_1_K_is_10_n_is_2}
\end{figure} 

Figure~\ref{fig:loglog_plot_L_2_error_Power_Basket_power_1_K_is_10_n_is_2}
shows convergence plots for the $l^{\infty}$-error (left)
and for the relative $l^2$-error (right) for a European Put,
respectively. The initial condition is smoothed using the procedure
outlined above. For both types of errors we observe that the numerical
convergence rates agree very well with the theoretical orders of the
schemes. The high-order compact scheme yields numerical convergence
orders close to four and strongly outperforms the standard
second-order scheme. The choice of the correlation parameter 
$\rho_{12} =-0.8$, $\rho_{12} = 0$ and $\rho_{12} = 0.8$ 
has very little influence.

\subsection{Numerical example with three assets}
In this section we report on numerical experiments with three
underlying assets.
We choose the parameters 
\begin{equation*}
\delta_i = 0.01,\quad  \sigma_i = 0.3,\quad
\omega_i = 1/3,  \quad  r =\ln ( 1.05),\quad  \gamma = 0.3,\quad T =
0.25, \quad K=10.
\end{equation*}
Due to the computational intensity of the three-dimensional problem the number of grid points
per spatial dimension is smaller compared to the results in two
dimensions reported above. To ensure that at the same time there is a
sufficiently large number of grid points in time, we fix the parabolic
mesh ratio to $\Delta \tau / h^2 = 0.1$ (not for stability reasons).
We perform two types of experiments: without any correlation between
the assets (labeled by `nc' in the plots), and
with correlation (labeled by `c' in the plots) using the parameter values $\rho_{1,2} = -0.4$, $\rho_{1,3} = -0.1$, $\rho_{2,3} = -0.2.$

We compare the standard approximation to our high-order compact
scheme for European Power Put
options with $p=3,4$. For the European Power Puts with $p=1,2$ one
would smooth the initial condition, similar as above, to ensure
high-order convergence.
\begin{figure}\centering
      \includegraphics[width=6cm,height=3.5cm]{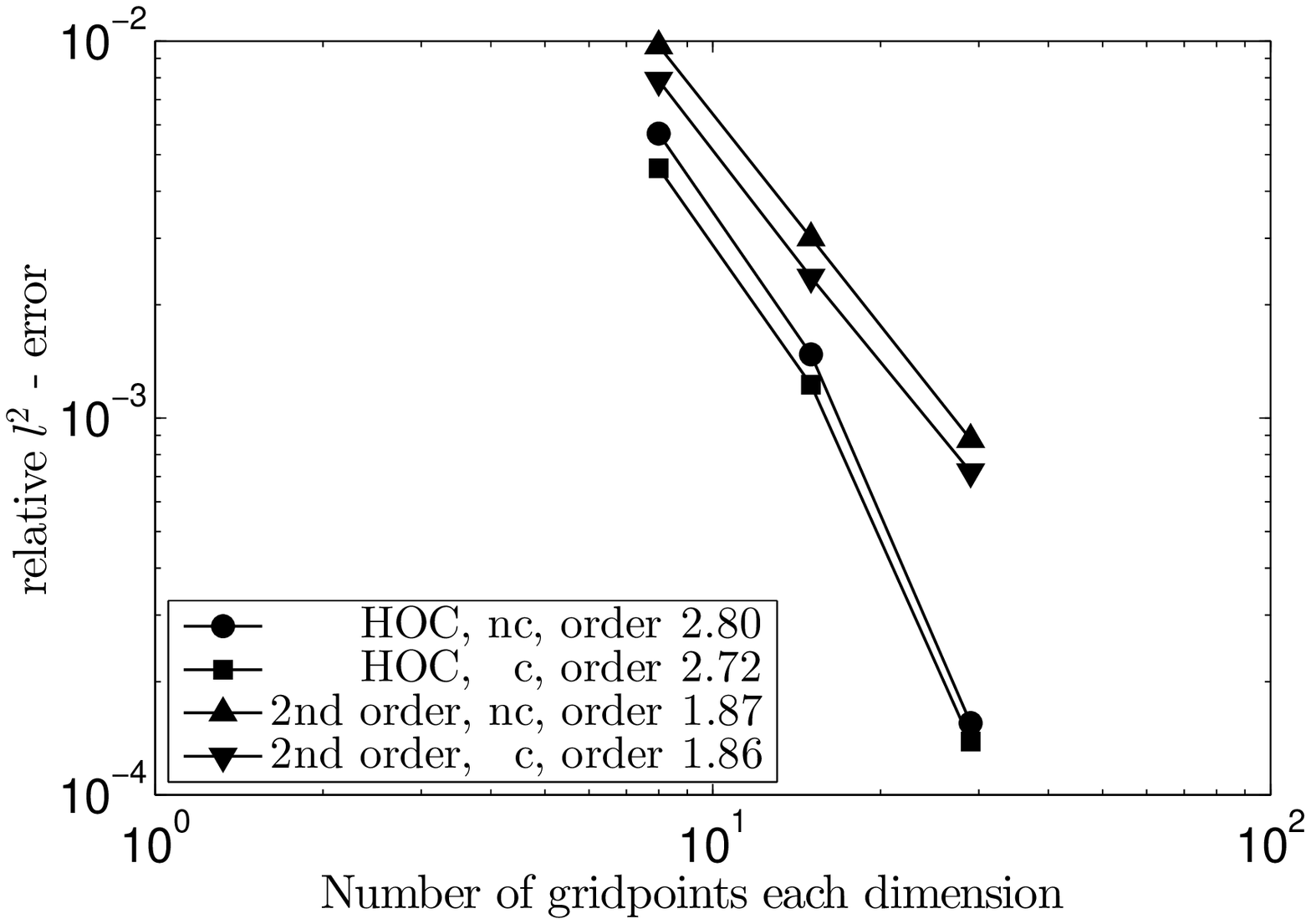}%
      \includegraphics[width=6cm,height=3.5cm]{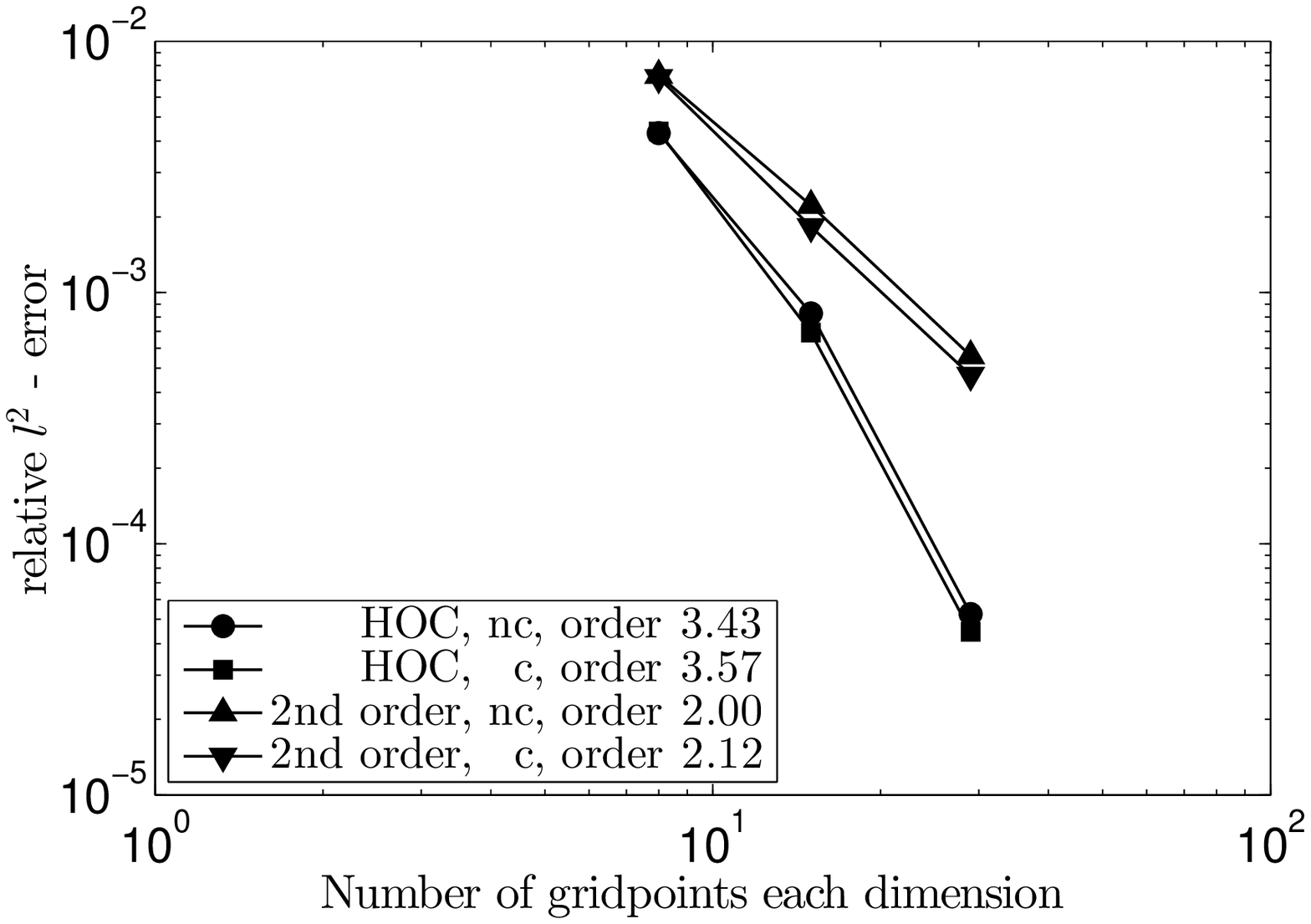}
      \caption{Relative $l^2$-error for three-dimensional
        Black-Scholes Basket Power Put, with $p=3$ (left) and $p=4$ (right)}
      \label{fig:loglog_plot_L_2_error_Power_Basket_power_34_K_is_10_n_is_3}
\end{figure}
Figure ~\ref{fig:loglog_plot_L_2_error_Power_Basket_power_34_K_is_10_n_is_3} shows
the convergence of the relative $l^2$-error for a European Power Put with $p=3$ and
$p=4$. We use the original initial conditions, no smoothing is
applied here. The
numerical convergence rates of the high-order compact scheme are
slightly reduced to about three and three and a half,
respectively. Additional smoothing, which we omitted here due to limit
the computational load, would result in even better results. Still, in
the high-order compact scheme outperforms the standard second-order scheme significantly
in all cases.
\subsection{Numerical example with space-dependent coefficients}
In this section we will apply numerical examples for
\eqref{basicpdemutliblsBasket}, where the continuous dividends are
dependent on the underlying asset price.
For both asset prices $S_i$ with $i=1,2$ we consider the following example, where the continuous dividends are zero for small asset prices and 
then smoothly increase around an asset price $S_i^\star>0$ towards a given parameter $\delta_i^\star\geq 0$,
$$ \delta_i = \delta_i(S_i) = \frac{\delta_i^\star\left[\tanh\left(\zeta_i(S_i-S_i^\star)\right) - \tanh\left(-\zeta_i S_i^\star\right) \right]}{2}.$$
Financially, the interpretation could be as follows: if the asset is a
dividend-paying stock, low stock prices may mean that the company may
not be in the financial position to pay dividends. A low value of $\zeta_i>0$ leads to slow transition from $0$ to $\delta_i^\star$. 
We can apply the transformations given in
\eqref{transformation_BLS_Basket_pde} and hence use the coefficients
\begin{equation}\label{pde_coefficients_2_dim_bls_space_depend_coefficients}
a_i= -\frac{\gamma^2}{2}, \quad b_{ij}= -\gamma^2 \rho_{ij},  \quad
 c_i=\gamma \left(\frac{\sigma_i}{2} - \frac{r - \delta_i(Ke^{  \frac{x_i\sigma_i}{\gamma}})}{\sigma_i}\right), \quad g= 0, 
\end{equation}
for $i=1,2$ to obtain the coefficients of the numerical scheme, see Section~\ref{Section_2DGeneralPDESemi_discrete}.
The boundary conditions of Section~\ref{boundary_conditions_multi_dim_bls} are employed and the parameter values of Section~\ref{sec:num_example_2D_const_coeff} 
as well as
$$ \delta_1^\star = 0.02,\quad  \delta_2^\star = 0.01,\quad   \zeta_1 = 0.35,\quad   \zeta_2 = 0.5,\quad  S_i^\star = {0.9 K}/{\omega_i},$$
for $i=1,2$ are used in the numerical experiments.
\begin{figure}[h]\centering
      \includegraphics[width=6cm,height=3.5cm]{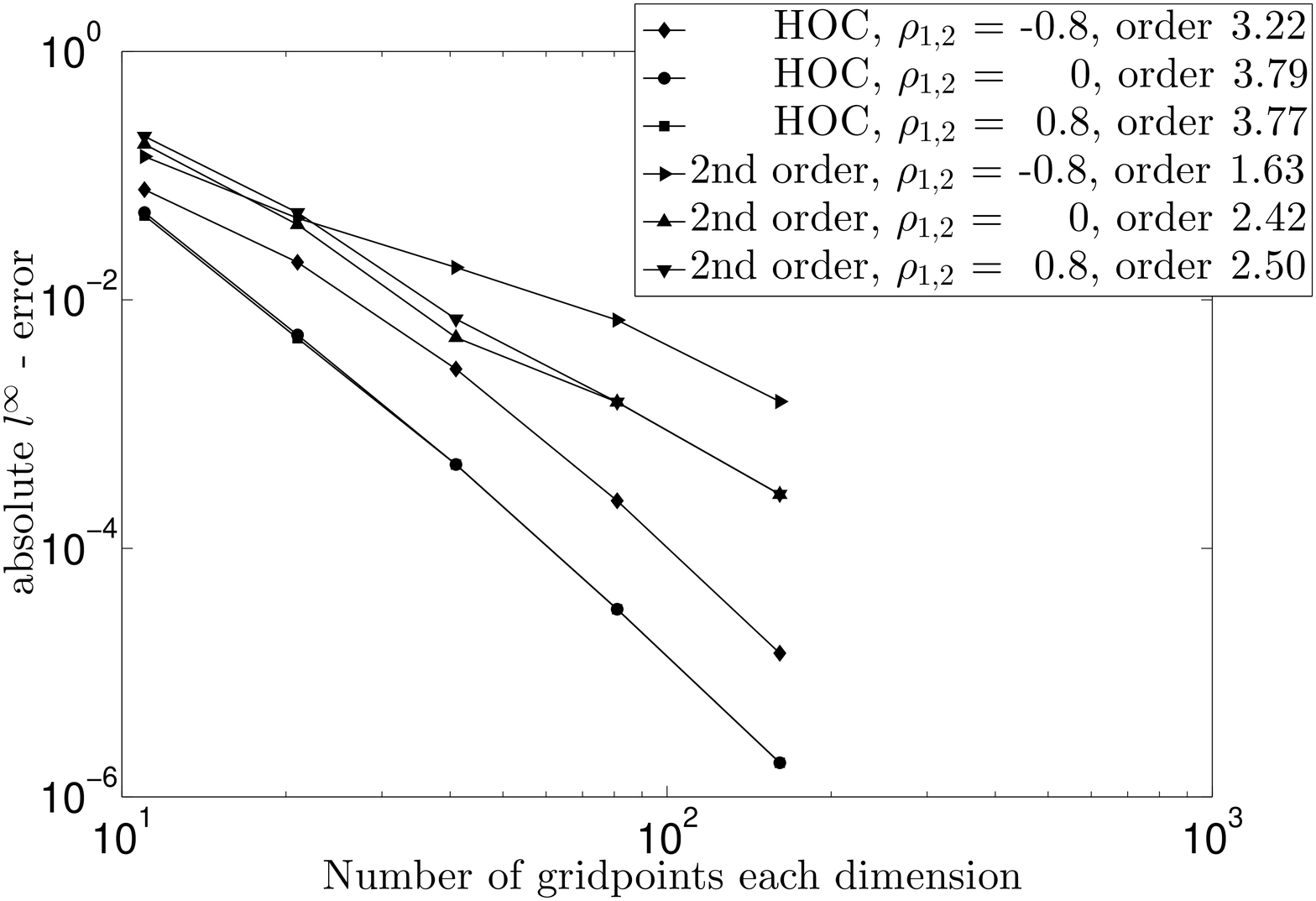}%
      \includegraphics[width=6cm,height=3.5cm]{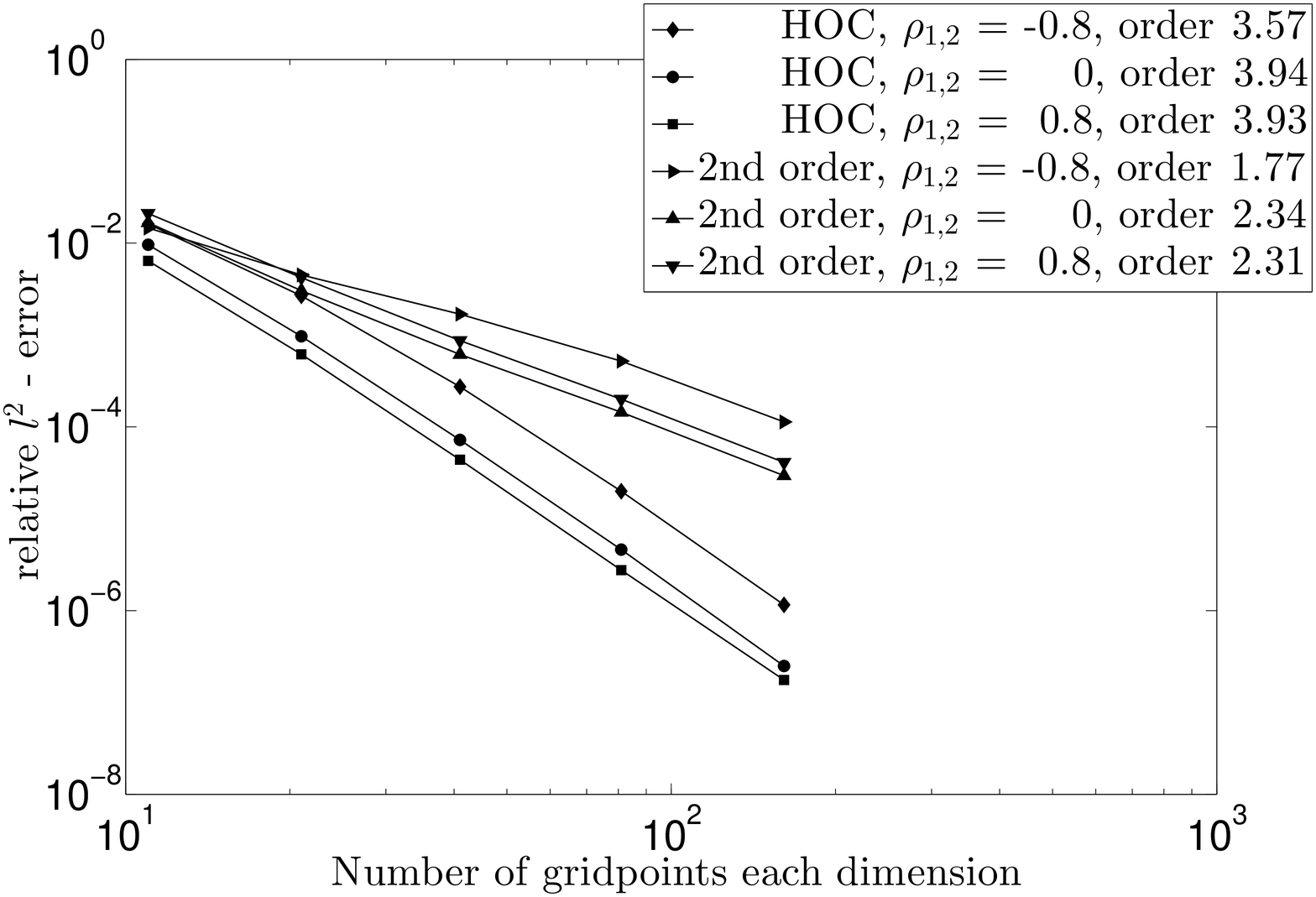}
      \caption{$l^{\infty}$- (left) and relative $l^2$-error (right)
        for two-dimensional Black-Scholes Basket Put with space-dependent dividend and smoothed initial condition.}
      \label{fig:loglog_plot_L_2_error_Power_Basket_power_1_space_dep_cont_dividend_K_is_10_n_is_2}
\end{figure} 
Figure~\ref{fig:loglog_plot_L_2_error_Power_Basket_power_1_space_dep_cont_dividend_K_is_10_n_is_2}
shows numerical convergence plots for a European Put with space-dependent continuous dividend.
Again, smoothing of the initial condition is
employed. For the $l^{\infty}$-error as well as the $l^{2}$-error the high-order compact scheme has convergence rates close to
four for $\rho_{1,2}=0,$ and $\rho_{1,2}=0.8$. The convergence rate for the case $\rho_{1,2}= -0.8$ is $3.22$ in the $l^{\infty}$-error, which is mainly due to the two approximations with eleven and 21 grid-points per spatial direction, and $3.57$ in the $l^{2}$-error. 
The convergence orders of the standard scheme are for $\rho_{1,2} = 0,
0.8$ are slightly above two for both types of errors. For $\rho_{1,2}= -0.8$ the convergence orders are
noticeable lower as well. In all cases of correlation the high-order compact scheme outperforms the standard second-order scheme significantly.

\section{Conclusion}\label{sec:Conclusion_n_dim_HOC}

We presented a new high-order compact scheme for a 
class of 
parabolic partial differential equations with time and space dependent
coefficients, including mixed second-order derivative terms in $n$
spatial dimensions. The resulting schemes are fourth-order accurate in
space and second-order accurate in time. In a thorough von Neumann
stability analysis, where we focussed on the case of vanishing mixed
derivative terms, we showed that a necessary stability condition
holds for frozen coefficients without further conditions in two and three space dimensions. For
non-vanishing mixed derivative terms we were able to give partial
results. The results suggest unconditional stability of the scheme.
As an application example we considered the pricing of
European Power Put options in the
multidimensional Black-Scholes model. The typical
initial conditions of this problem lack sufficient regularity,
therefore a suitable smoothing procedure was employed to ensure
high-order convergence.
In all numerical experiments performed a comparative standard second-order
scheme is significantly outperformed.

Although we derived the scheme in arbitrary space dimension, it was not
our aim in this paper to attack the so-called curse of
dimensionality. The issue of exponentially increasing number of
unknowns with growing spatial dimension on full grids is of course
alleviated to some degree by a high-order scheme. To obtain a similar accuracy as a
second-order scheme which uses $\mathcal{O}(N^d)$ unknowns on a full
grid, our high-order compact approach will `only' require
$\mathcal{O}(N^{d/2})$ unknowns. To really attack very
high-dimensional problems one would need to combine our approach with
hierarchical approaches, e.g.\ using sparse grids (typically requiring
$\mathcal{O}(N\ln(N)^{d-1})$ unknowns), which is beyond the scope of
the present paper. 

\section*{Acknowledgment}
The authors are grateful to the anonymous reviewers for their
constructive comments.
The second author acknowledges support by the European Union
in the FP7-PEOPLE-2012-ITN Program under Grant Agreement Number 304617
(FP7 Marie Curie Action, Project Multi-ITN {\em STRIKE -- Novel
  Methods in Computational Finance}).


%% file: Appendix_3_dim_general_coefficients.tex
\section{Coefficients for semi-discrete scheme in three
  dimensions}

Considering an interior grid point
$\bigl(x_{i_1}^{(1)},x_{i_2}^{(2)},x_{i_3}^{(3)}\bigr)\in
\interior{G}_h^{(3)}$ and time $\tau \in
\Omega_{\tau}$, the coefficients $\hat{K}_{k,l,m}$ of
$U_{k,l,m}\left(\tau\right)$ for $k\in \{ i_1 - 1, i_1, i_1 + 1\}$,
$l\in \{ i_2 - 1, i_2, i_2 + 1\}$ and $m\in \{ i_3 - 1, i_3, i_3 +
1\}$ of the three-dimensional semi-discrete scheme in Section 5.2 are given by:
\begin{small}
\begin{align*}
\hat{K}_{i_1,i_2,i_3}  = & 
{\frac {b_{{23}}[a]_{{{{2}}}}c_{{3}}}{6{a}^{2}}}
+ {\frac {b_{{13}}[a]_{{{{1}}}}c_{{3}}}{6{a}^{2}}}
- \frac{[c_{{3}}]_{{{{3}}}}}{3}
- {\frac {{c^{2}_{{1}}}}{6a}}
- {\frac {{c^{2}_{{3}}}}{6a}}
- \frac{[a]_{{{ 11}}}}{2}
- \frac{[a]_{{{ 22}}}}{2}
- \frac{[a]_{{{ 33}}}}{2}
+ {\frac {b_{{13}}[a]_{{{{3}}}}c_{{1}}}{6{a}^{2}}}
\\[4pt]
&
+ {\frac {b_{{12}}[a]_{{{{2}}}}c_{{1}}}{6{a}^{2}}}
- {\frac {4a}{{h}^{2}}}
+ {\frac {b_{{13}}[a]_{{{{3}}}}[a]_{{{{1}}}}}{{a}^{2}}}
+ {\frac {b_{{23}}[a]_{{{{3}}}}[a]_{{{{2}}}}}{{a}^{2}}}
+ {\frac {b_{{23}}[a]_{{{{3}}}}c_{{2}}}{6{a}^{2}}}
+ {\frac {b_{{12}}[a]_{{{{1}}}}[a]_{{{{2}}}}}{{a}^{2}}}
\\[4pt]
&
+ {\frac {b_{{12}}[a]_{{{{1}}}}c_{{2}}}{6{a}^{2}}}
- {\frac {b_{{13}}[c_{{3}}]_{{{{1}}}}}{6a}}
- {\frac {c_{{1}}[a]_{{{{1}}}}}{6a}}
+ {\frac {{b^{2}_{{23}}}}{3a{h}^{2}}}
- {\frac {b_{{12}}[a]_{{{{1}}{{2}}}}}{2a}}
- {\frac {c_{{2}}[a]_{{{{2}}}}}{6a}}
+ {\frac {{b^{2}_{{13}}}}{3a{h}^{2}}}
+ {\frac {{b^{2}_{{12}}}}{3a{h}^{2}}}
\\[4pt]
&
- {\frac {c_{{3}}[a]_{{{{3}}}}}{6a}}
- {\frac {b_{{13}}[a]_{{{{1}}{{2}}}}}{2a}}
- {\frac {b_{{23}}[c_{{2}}]_{{{{3}}}}}{6a}}
- {\frac {b_{{12}}[c_{{2}}]_{{{{1}}}}}{6a}}
- {\frac {b_{{23}}[a]_{{{{2}}{{3}}}}}{2a}}
- {\frac {b_{{13}}[c_{{1}}]_{{{{3}}}}}{6a}}
- {\frac {b_{{23}}[c_{{3}}]_{{{{2}}}}}{6a}}
\\[4pt]
&
- {\frac {b_{{12}}[c_{{1}}]_{{{{2}}}}}{6a}}
- {\frac {{c^{2}_{{2}}}}{6a}}
+ {\frac {{[a]^{2}_{{{{1}}}}}}{a}}
+ {\frac {{[a]^{2}_{{{{3}}}}}}{a}}
+ {\frac {{[a]^{2}_{{{{2}}}}}}{a}}
- \frac{[c_{{2}}]_{{{{2}}}}}{3}
- \frac{[c_{{1}}]_{{{{1}}}}}{3} ,\\
\hat{K}_{i_1\pm 1,i_2-1,i_3}  = & 
{\frac {b_{{13}}[a]_{{{{3}}}}b_{{12}}}{24{a}^{2}h}}
\mp {\frac {b_{{23}}[a]_{{{{3}}}}b_{{12}}}{24{a}^{2}h}}
\mp \frac {[b_{{12}}]_{{{11}}}}{48}
\mp \frac {[b_{{12}}]_{{{22}}}}{48}
\mp \frac {[b_{{12}}]_{{{33}}}}{48}
+ {\frac {b_{{12}}[a]_{{{{1}}}}}{12ah}}
- {\frac {b_{{12}}c_{{1}}}{12ah}}
\pm {\frac {b_{{12}}c_{{2}}}{12ah}}
\\[4pt]
&
\pm {\frac {b_{{12}}[a]_{{{{1}}}}c_{{1}}}{48{a}^{2}}}
\pm {\frac {b_{{12}}[a]_{{{{1}}}}[b_{{12}}]_{{{{2}}}}}{48{a}^{2}}}
\mp {\frac {{b^{2}_{{12}}}[a]_{{{{1}}}}}{24{a}^{2}h}}
\mp {\frac {b_{{12}}[a]_{{{{2}}}}}{12ah}}
\pm  {\frac {b_{{23}}[a]_{{{{2}}}}[b_{{12}}]_{{{{3}}}}}{48{a}^{2}}}
\pm {\frac {b_{{13}}[a]_{{{{1}}}}[b_{{12}}]_{{{{3}}}}}{48{a}^{2}}}
\\[4pt]
&
\pm {\frac {b_{{12}}[a]_{{{{2}}}}[b_{{12}}]_{{{{1}}}}}{48{a}^{2}}}
\pm {\frac {b_{{12}}[a]_{{{{2}}}}c_{{2}}}{48{a}^{2}}}
\pm {\frac {b_{{23}}[a]_{{{{3}}}}c_{{1}}}{48{a}^{2}}}
\pm {\frac {b_{{23}}[a]_{{{{3}}}}[b_{{12}}]_{{{{2}}}}}{48{a}^{2}}}
- {\frac {b_{{12}}[b_{{12}}]_{{{{2}}}}}{24ah}}
\pm {\frac {b_{{13}}[a]_{{{{3}}}}[b_{{12}}]_{{{{1}}}}}{48{a}^{2}}}
\\[4pt]
&
\pm {\frac {b_{{13}}[a]_{{{{3}}}}c_{{2}}}{48{a}^{2}}}
\pm {\frac {b_{{12}}[b_{{12}}]_{{{{1}}}}}{24ah}}
\pm {\frac {b_{{23}}[b_{{12}}]_{{{{3}}}}}{24ah}}
- {\frac {b_{{13}}[b_{{12}}]_{{{{3}}}}}{24ah}}
\pm {\frac {b_{{13}}b_{{23}}}{12a{h}^{2}}}
+ {\frac {[a]_{{{{2}}}}{b^{2}_{{12}}}}{24{a}^{2}h}}
- {\frac {c_{{2}}}{12h}}
\mp {\frac {b_{{12}}}{6{h}^{2}}}
\\[4pt]
&
\pm {\frac {c_{{1}}}{12h}}
+ {\frac {a}{6{h}^{2}}}
- {\frac {[b_{{12}}]_{{{{1}}}}}{12h}}
\pm {\frac {[b_{{12}}]_{{{{2}}}}}{12h}}
\mp {\frac {b_{{23}}[b_{{12}}]_{{{{2}}{{3}}}}}{48a}}
\mp {\frac {b_{{13}}[c_{{2}}]_{{{{3}}}}}{48a}}
\mp {\frac {b_{{12}}[b_{{12}}]_{{{{1}}{{2}}}}}{48a}}
\mp {\frac {b_{{23}}[c_{{1}}]_{{{{3}}}}}{48a}}
\\[4pt]
&
\mp {\frac {b_{{12}}[c_{{1}}]_{{{{1}}}}}{48a}}
\pm  {\frac {[a]_{{{{1}}}}c_{{2}}}{24a}}
\mp {\frac {c_{{2}}[b_{{12}}]_{{{{2}}}}}{48a}}
+ {\frac {{b^{2}_{{12}}}}{12a{h}^{2}}}
\mp {\frac {b_{{13}}[b_{{12}}]_{{{{1}}{{2}}}}}{48a}}
\mp {\frac {c_{{1}}[b_{{12}}]_{{{{1}}}}}{48a}}
\pm {\frac {[a]_{{{{3}}}}[b_{{12}}]_{{{{3}}}}}{24a}}
\\[4pt]
&
\mp {\frac {c_{{3}}[b_{{12}}]_{{{{3}}}}}{48a}}
\pm {\frac {[a]_{{{{2}}}}c_{{1}}}{24a}}
\mp {\frac {b_{{12}}[c_{{2}}]_{{{{2}}}}}{48a}}
\pm {\frac {[a]_{{{{1}}}}[b_{{12}}]_{{{{1}}}}}{24a}}
\pm {\frac {[a]_{{{{2}}}}[b_{{12}}]_{{{{2}}}}}{24a}}
\mp {\frac {c_{{1}}c_{{2}}}{24a}}
\mp \frac{[c_{{1}}]_{{{{2}}}}}{24}
\mp \frac{[c_{{2}}]_{{{{1}}}}}{24},\\
\label{three_dim_general_K_coefficients_three}
\notag \hat{K}_{i\pm 1,j,m}  = & 
\pm {\frac {b_{{23}}[a]_{{{{3}}}}b_{{12}}}{12{a}^{2}h}}
\pm {\frac {b_{{23}}[a]_{{{{2}}}}b_{{13}}}{12{a}^{2}h}}
\mp {\frac {hb_{{12}}[a]_{{{{1}}}}[c_{{1}}]_{{{{2}}}}}{24{a}^{2}}}
\mp {\frac {hb_{{23}}[a]_{{{{2}}}}[c_{{1}}]_{{{{3}}}}}{24{a}^{2}}}
\mp {\frac {hb_{{13}}[a]_{{{{1}}}}[c_{{1}}]_{{{{3}}}}}{24{a}^{2}}}
\\
\notag &
+ {\frac {{c^{2}_{{1}}}}{12a}}
\mp {\frac {hb_{{13}}[a]_{{{{3}}}}[c_{{1}}]_{{{{1}}}}}{24{a}^{2}}}
\mp {\frac {hb_{{23}}[a]_{{{{3}}}}[c_{{1}}]_{{{{2}}}}}{24{a}^{2}}}
\pm \frac{h[c_{{1}}]_{{{ 33}}}}{24}
\pm \frac{h[c_{{1}}]_{{{ 22}}}}{24}
\pm \frac{h[c_{{1}}]_{{{ 11}}}}{24}
\\
\notag &
+ \frac{[a]_{{{ 11}}}}{12}
+ \frac{[a]_{{{ 22}}}}{12}
+ \frac{[a]_{{{ 33}}}}{12}
- {\frac {b_{{13}}[a]_{{{{3}}}}c_{{1}}}{12{a}^{2}}}
\mp {\frac {b_{{13}}[b_{{13}}]_{{{{1}}}}}{12ah}}
\pm {\frac {hc_{{1}}[c_{{1}}]_{{{{1}}}}}{24a}}
\pm {\frac {hc_{{2}}[c_{{1}}]_{{{{2}}}}}{24a}}
\\
\notag &
\mp {\frac {hb_{{12}}[a]_{{{{2}}}}[c_{{1}}]_{{{{1}}}}}{24{a}^{2}}}
\mp {\frac {b_{{23}}[b_{{13}}]_{{{{2}}}}}{12ah}}
\pm {\frac {{b^{2}_{{13}}}[a]_{{{{1}}}}}{12{a}^{2}h}}
\mp {\frac {c_{{3}}b_{{13}}}{6ah}}
\pm {\frac {hb_{{23}}[c_{{1}}]_{{{{2}}{{3}}}}}{24a}}
\pm {\frac {hb_{{13}}[c_{{1}}]_{{{{1}}{{2}}}}}{24a}}
\\
&
\pm {\frac {hb_{{12}}[c_{{1}}]_{{{{1}}{{2}}}}}{24a}}
\pm {\frac {b_{{13}}[a]_{{{{3}}}}}{6ah}}
- {\frac {b_{{12}}[a]_{{{{2}}}}c_{{1}}}{12{a}^{2}}}
\pm {\frac {hc_{{3}}[c_{{1}}]_{{{{3}}}}}{24a}}
\mp {\frac {h[a]_{{{{1}}}}[c_{{1}}]_{{{{1}}}}}{12a}}
\mp {\frac {h[a]_{{{{3}}}}[c_{{1}}]_{{{{3}}}}}{12a}}
\\
\notag &
\mp {\frac {b_{{12}}c_{{2}}}{6ah}}
\pm {\frac {{b^{2}_{{12}}}[a]_{{{{1}}}}}{12{a}^{2}h}}
\pm {\frac {b_{{12}}[a]_{{{{2}}}}}{6ah}}
\mp {\frac {b_{{12}}[b_{{12}}]_{{{{1}}}}}{12ah}}
\mp {\frac {b_{{23}}[b_{{12}}]_{{{{3}}}}}{12ah}}
\pm {\frac {c_{{1}}}{6h}}
+ {\frac {a}{3{h}^{2}}}
\mp {\frac {[b_{{12}}]_{{{{2}}}}}{6h}}
\\
\notag &
- {\frac {b_{{13}}[a]_{{{{3}}}}[a]_{{{{1}}}}}{6{a}^{2}}}
- {\frac {b_{{23}}[a]_{{{{3}}}}[a]_{{{{2}}}}}{6{a}^{2}}}
\mp {\frac {[b_{{13}}]_{{{{3}}}}}{6h}}
- {\frac {b_{{12}}[a]_{{{{1}}}}[a]_{{{{2}}}}}{6{a}^{2}}}
- {\frac {c_{{1}}[a]_{{{{1}}}}}{12a}}
+ {\frac {b_{{12}}[a]_{{{{1}}{{2}}}}}{12a}}
\\
\notag &
+ {\frac {c_{{2}}[a]_{{{{2}}}}}{12a}}
- {\frac {{b^{2}_{{13}}}}{6a{h}^{2}}}
- {\frac {{b^{2}_{{12}}}}{6a{h}^{2}}}
+ {\frac {c_{{3}}[a]_{{{{3}}}}}{12a}}
+ {\frac {b_{{13}}[a]_{{{{1}}{{2}}}}}{12a}}
+ {\frac {b_{{23}}[a]_{{{{2}}{{3}}}}}{12a}}
+ {\frac {b_{{13}}[c_{{1}}]_{{{{3}}}}}{12a}}
\\
\notag &
+ {\frac {b_{{12}}[c_{{1}}]_{{{{2}}}}}{12a}}
- {\frac {{[a]^{2}_{{{{1}}}}}}{6a}}
- {\frac {{[a]^{2}_{{{{3}}}}}}{6a}}
- {\frac {{[a]^{2}_{{{{2}}}}}}{6a}}
+ \frac{[c_{{1}}]_{{{{1}}}}}{6}
\mp {\frac {h[a]_{{{{2}}}}[c_{{1}}]_{{{{2}}}}}{12a}} ,\\
\hat{K}_{i,j\pm 1,m}  = & 
\pm {\frac {b_{{13}}[a]_{{{{3}}}}b_{{12}}}{12{a}^{2}h}}
\mp {\frac {hb_{{13}}[a]_{{{{1}}}}[c_{{2}}]_{{{{3}}}}}{24{a}^{2}}}
\mp {\frac {hb_{{12}}[a]_{{{{1}}}}[c_{{2}}]_{{{{2}}}}}{24{a}^{2}}}
\pm {\frac {b_{{13}}[a]_{{{{1}}}}b_{{23}}}{12{a}^{2}h}}
\mp {\frac {hb_{{13}}[a]_{{{{3}}}}[c_{{2}}]_{{{{1}}}}}{24{a}^{2}}}
\\
&
\pm {\frac {c_{{2}}}{6h}}
\mp {\frac {hb_{{23}}[a]_{{{{3}}}}[c_{{2}}]_{{{{2}}}}}{24{a}^{2}}}
\mp {\frac {hb_{{12}}[a]_{{{{2}}}}[c_{{2}}]_{{{{1}}}}}{24{a}^{2}}}
\mp {\frac {hb_{{23}}[a]_{{{{2}}}}[c_{{2}}]_{{{{3}}}}}{24{a}^{2}}}
+ \frac{[a]_{{{ 11}}}}{12}
+ \frac{[a]_{{{22 }}}}{12}
\\
&
+ \frac{[a]_{{{33}}}}{12}
\pm {\frac {b_{{12}}[a]_{{{{1}}}}}{6ah}}
\mp {\frac {b_{{12}}c_{{1}}}{6ah}}
\mp  {\frac {b_{{12}}[b_{{12}}]_{{{{2}}}}}{12ah}}
\mp {\frac {b_{{13}}[b_{{12}}]_{{{{3}}}}}{12ah}}
\pm {\frac {[a]_{{{{2}}}}{b^{2}_{{12}}}}{12{a}^{2}h}}
+ {\frac {a}{3{h}^{2}}}
\\
&
\mp {\frac {[b_{{12}}]_{{{{1}}}}}{6h}}
\mp {\frac {b_{{23}}c_{{3}}}{6ah}}
\pm {\frac {hb_{{12}}[c_{{2}}]_{{{{1}}{{2}}}}}{24a}}
\pm {\frac {hb_{{23}}[c_{{2}}]_{{{{2}}{{3}}}}}{24a}}
\pm {\frac {hb_{{13}}[c_{{2}}]_{{{{1}}{{2}}}}}{24a}}
- {\frac {b_{{13}}[a]_{{{{3}}}}[a]_{{{{1}}}}}{6{a}^{2}}}
\\
&
\pm {\frac {hc_{{2}}[c_{{2}}]_{{{{2}}}}}{24a}}
\pm {\frac {b_{{23}}[a]_{{{{3}}}}}{6ah}}
- {\frac {b_{{23}}[a]_{{{{3}}}}c_{{2}}}{12{a}^{2}}}
\mp {\frac {b_{{13}}[b_{{23}}]_{{{{1}}}}}{12ah}}
\mp {\frac {b_{{23}}[b_{{23}}]_{{{{2}}}}}{12ah}}
\mp {\frac {h[a]_{{{{2}}}}[c_{{2}}]_{{{{2}}}}}{12a}}
\\
&
\pm {\frac {hc_{{3}}[c_{{2}}]_{{{{3}}}}}{24a}}
\mp {\frac {h[a]_{{{{1}}}}[c_{{2}}]_{{{{1}}}}}{12a}}
\mp {\frac {h[a]_{{{{3}}}}[c_{{2}}]_{{{{3}}}}}{12a}}
\pm {\frac {hc_{{1}}[c_{{2}}]_{{{{1}}}}}{24a}}
\pm {\frac {[a]_{{{{2}}}}{b^{2}_{{23}}}}{12{a}^{2}h}}
- {\frac {b_{{12}}[a]_{{{{1}}}}[a]_{{{{2}}}}}{6{a}^{2}}}
\\
&
- {\frac {b_{{12}}[a]_{{{{1}}}}c_{{2}}}{12{a}^{2}}}
+ {\frac {c_{{1}}[a]_{{{{1}}}}}{12a}}
- {\frac {{b^{2}_{{23}}}}{6a{h}^{2}}}
+ {\frac {b_{{12}}[a]_{{{{1}}{{2}}}}}{12a}}
- {\frac {c_{{2}}[a]_{{{{2}}}}}{12a}}
- {\frac {{b^{2}_{{12}}}}{6a{h}^{2}}}
+ {\frac {c_{{3}}[a]_{{{{3}}}}}{12a}}
\\
&
+ {\frac {b_{{13}}[a]_{{{{1}}{{2}}}}}{12a}}
+ {\frac {b_{{23}}[c_{{2}}]_{{{{3}}}}}{12a}}
+ {\frac {b_{{12}}[c_{{2}}]_{{{{1}}}}}{12a}}
+ {\frac {b_{{23}}[a]_{{{{2}}{{3}}}}}{12a}}
\pm \frac{h[c_{{2}}]_{{{22}}}}{24}
\pm \frac{h[c_{{2}}]_{{{33}}}}{24}
\\
&
\pm \frac{h[c_{{2}}]_{{{11}}}}{24}
+ {\frac {{c^{2}_{{2}}}}{12a}}
- {\frac {{[a]^{2}_{{{{1}}}}}}{6a}}
- {\frac {{[a]^{2}_{{{{3}}}}}}{6a}}
\mp {\frac {[b_{{23}}]_{{{{3}}}}}{6h}}
- {\frac {{[a]^{2}_{{{{2}}}}}}{6a}}
- {\frac {b_{{23}}[a]_{{{{3}}}}[a]_{{{{2}}}}}{6{a}^{2}}}
+ \frac{[c_{{2}}]_{{{{2}}}}}{6} ,\\
\hat{K}_{i \pm 1,j+1,m}  = & 
- {\frac {b_{{13}}[a]_{{{{3}}}}b_{{12}}}{24{a}^{2}h}}
\mp {\frac {b_{{23}}[a]_{{{{3}}}}b_{{12}}}{24{a}^{2}h}}
\pm \frac{[b_{{12}}]_{{{11}}}}{48}
\pm \frac{[b_{{12}}]_{{{22}}}}{48}
\pm \frac{[b_{{12}}]_{{{33}}}}{48}
- {\frac {b_{{12}}[a]_{{{{1}}}}}{12ah}}
+ {\frac {b_{{12}}c_{{1}}}{12ah}}
\\
&
\pm {\frac {b_{{12}}c_{{2}}}{12ah}}
\mp {\frac {b_{{12}}[a]_{{{{1}}}}c_{{1}}}{48{a}^{2}}}
\mp {\frac {b_{{12}}[a]_{{{{1}}}}[b_{{12}}]_{{{{2}}}}}{48{a}^{2}}}
\mp {\frac {{b^{2}_{{12}}}[a]_{{{{1}}}}}{24{a}^{2}h}}
\mp {\frac {b_{{12}}[a]_{{{{2}}}}}{12ah}}
\mp {\frac {b_{{23}}[a]_{{{{2}}}}[b_{{12}}]_{{{{3}}}}}{48{a}^{2}}}
\\
&
\mp {\frac {b_{{13}}[a]_{{{{1}}}}[b_{{12}}]_{{{{3}}}}}{48{a}^{2}}}
\mp {\frac {b_{{12}}[a]_{{{{2}}}}[b_{{12}}]_{{{{1}}}}}{48{a}^{2}}}
\mp {\frac {b_{{12}}[a]_{{{{2}}}}c_{{2}}}{48{a}^{2}}}
\mp {\frac {b_{{23}}[a]_{{{{3}}}}c_{{1}}}{48{a}^{2}}}
\mp {\frac {b_{{23}}[a]_{{{{3}}}}[b_{{12}}]_{{{{2}}}}}{48{a}^{2}}}
\\
&
+ {\frac {b_{{12}}[b_{{12}}]_{{{{2}}}}}{24ah}}
\mp {\frac {b_{{13}}[a]_{{{{3}}}}[b_{{12}}]_{{{{1}}}}}{48{a}^{2}}}
\mp {\frac {b_{{13}}[a]_{{{{3}}}}c_{{2}}}{48{a}^{2}}}
\pm {\frac {b_{{12}}[b_{{12}}]_{{{{1}}}}}{24ah}}
\pm {\frac {b_{{23}}[b_{{12}}]_{{{{3}}}}}{24ah}}
+ {\frac {b_{{13}}[b_{{12}}]_{{{{3}}}}}{24ah}}
\\
&
\mp {\frac {b_{{13}}b_{{23}}}{12a{h}^{2}}}
- {\frac {[a]_{{{{2}}}}{b^{2}_{{12}}}}{24{a}^{2}h}}
+ {\frac {c_{{2}}}{12h}}
\pm {\frac {b_{{12}}}{6{h}^{2}}}
\pm {\frac {c_{{1}}}{12h}}
+ {\frac {a}{6{h}^{2}}}
+ {\frac {[b_{{12}}]_{{{{1}}}}}{12h}}
\pm {\frac {[b_{{12}}]_{{{{2}}}}}{12h}}
\pm {\frac {b_{{23}}[b_{{12}}]_{{{{2}}{{3}}}}}{48a}}
\\
&
\pm {\frac {b_{{13}}[c_{{2}}]_{{{{3}}}}}{48a}}
\pm {\frac {b_{{12}}[b_{{12}}]_{{{{1}}{{2}}}}}{48a}}
\pm {\frac {b_{{23}}[c_{{1}}]_{{{{3}}}}}{48a}}
\pm {\frac {b_{{12}}[c_{{1}}]_{{{{1}}}}}{48a}}
\mp {\frac {[a]_{{{{1}}}}c_{{2}}}{24a}}
\pm {\frac {c_{{2}}[b_{{12}}]_{{{{2}}}}}{48a}}
+ {\frac {{b^{2}_{{12}}}}{12a{h}^{2}}}
\\
&
\pm {\frac {b_{{13}}[b_{{12}}]_{{{{1}}{{2}}}}}{48a}}
\pm {\frac {c_{{1}}[b_{{12}}]_{{{{1}}}}}{48a}}
\mp {\frac {[a]_{{{{3}}}}[b_{{12}}]_{{{{3}}}}}{24a}}
\pm {\frac {c_{{3}}[b_{{12}}]_{{{{3}}}}}{48a}}
\mp {\frac {[a]_{{{{2}}}}c_{{1}}}{24a}}
\pm {\frac {b_{{12}}[c_{{2}}]_{{{{2}}}}}{48a}}
\\
&
\mp {\frac {[a]_{{{{1}}}}[b_{{12}}]_{{{{1}}}}}{24a}}
\mp {\frac {[a]_{{{{2}}}}[b_{{12}}]_{{{{2}}}}}{24a}}
\pm {\frac {c_{{1}}c_{{2}}}{24a}}
\pm \frac{[c_{{1}}]_{{{{2}}}}}{24}
\pm \frac{[c_{{2}}]_{{{{1}}}}}{24}  ,\\
\hat{K}_{i \pm 1,j-1,m-1}  = & 
\pm {\frac {[b_{{13}}]_{{{{2}}}}}{48h}}
\mp {\frac {b_{{13}}}{24{h}^{2}}}
\pm {\frac {[b_{{12}}]_{{{{3}}}}}{48h}}
\pm {\frac {[b_{{23}}]_{{{{1}}}}}{48h}}
\pm {\frac {b_{{12}}[b_{{13}}]_{{{{1}}}}}{96ah}}
\pm {\frac {b_{{13}}[b_{{12}}]_{{{{1}}}}}{96ah}}
\pm {\frac {b_{{23}}[b_{{13}}]_{{{{3}}}}}{96ah}}
\\
&
\mp {\frac {b_{{23}}b_{{12}}}{24a{h}^{2}}}
\mp {\frac {b_{{13}}[a]_{{{{3}}}}b_{{23}}}{48{a}^{2}h}}
\mp {\frac {b_{{12}}[a]_{{{{2}}}}b_{{23}}}{48{a}^{2}h}}
+ {\frac {b_{{23}}}{24{h}^{2}}}
\pm {\frac {b_{{13}}[b_{{23}}]_{{{{3}}}}}{96ah}}
\pm {\frac {b_{{23}}c_{{1}}}{48ah}}
+ {\frac {b_{{12}}b_{{13}}}{24a{h}^{2}}}
\\
&
\mp {\frac {b_{{12}}}{24{h}^{2}}}
\mp {\frac {[a]_{{{{2}}}}b_{{13}}}{48ah}}
\mp {\frac {[a]_{{{{3}}}}b_{{12}}}{48ah}}
\mp {\frac {b_{{13}}[a]_{{{{1}}}}b_{{12}}}{48{a}^{2}h}}
\mp {\frac {[a]_{{{{1}}}}b_{{23}}}{48ah}}
\pm {\frac {b_{{12}}[b_{{23}}]_{{{{2}}}}}{96ah}}
\\
&
\pm {\frac {b_{{23}}[b_{{12}}]_{{{{2}}}}}{96ah}}
\pm {\frac {c_{{3}}b_{{12}}}{48ah}}
\pm {\frac {b_{{13}}c_{{2}}}{48ah}}
\mp {\frac {b_{{13}}b_{{23}}}{24a{h}^{2}}}  ,\\
\hat{K}_{i \pm 1,j+1,m+1}  = & 
\pm  {\frac {[b_{{13}}]_{{{{2}}}}}{48h}}
\pm {\frac {b_{{13}}}{24{h}^{2}}}
\pm {\frac {[b_{{12}}]_{{{{3}}}}}{48h}}
\pm {\frac {[b_{{23}}]_{{{{1}}}}}{48h}}
\pm {\frac {b_{{12}}[b_{{13}}]_{{{{1}}}}}{96ah}}
\pm {\frac {b_{{13}}[b_{{12}}]_{{{{1}}}}}{96ah}}
\pm {\frac {b_{{23}}[b_{{13}}]_{{{{3}}}}}{96ah}}
\\
&
\pm {\frac {b_{{23}}b_{{12}}}{24a{h}^{2}}}
\mp {\frac {b_{{13}}[a]_{{{{3}}}}b_{{23}}}{48{a}^{2}h}}
\mp {\frac {b_{{12}}[a]_{{{{2}}}}b_{{23}}}{48{a}^{2}h}}
+ {\frac {b_{{23}}}{24{h}^{2}}}
\pm {\frac {b_{{13}}[b_{{23}}]_{{{{3}}}}}{96ah}}
\pm {\frac {b_{{23}}c_{{1}}}{48ah}}
+ {\frac {b_{{12}}b_{{13}}}{24a{h}^{2}}}
\\
&
\pm {\frac {b_{{12}}}{24{h}^{2}}}
\mp {\frac {[a]_{{{{2}}}}b_{{13}}}{48ah}}
\mp {\frac {[a]_{{{{3}}}}b_{{12}}}{48ah}}
\mp {\frac {b_{{13}}[a]_{{{{1}}}}b_{{12}}}{48{a}^{2}h}}
\mp {\frac {[a]_{{{{1}}}}b_{{23}}}{48ah}}
\pm {\frac {b_{{12}}[b_{{23}}]_{{{{2}}}}}{96ah}}
\\
&
\pm {\frac {b_{{23}}[b_{{12}}]_{{{{2}}}}}{96ah}}
\pm {\frac {c_{{3}}b_{{12}}}{48ah}}
\pm {\frac {b_{{13}}c_{{2}}}{48ah}}
\pm {\frac {b_{{13}}b_{{23}}}{24a{h}^{2}}}  ,\\
\hat{K}_{i\pm 1,j,m-1}  = & 
\mp \frac{[c_{{3}}]_{{{{1}}}}}{24}
\mp {\frac {b_{{13}}}{6{h}^{2}}}
- {\frac {c_{{3}}}{12h}}
- {\frac {[b_{{13}}]_{{{{1}}}}}{12h}}
\mp {\frac {b_{{23}}[a]_{{{{2}}}}b_{{13}}}{24{a}^{2}h}}
\mp \frac{[b_{{13}}]_{{{11}}}}{48}
\mp \frac{[b_{{13}}]_{{{22}}}}{48}
\mp \frac{[b_{{13}}]_{{{33}}}}{48}
\\
&
\pm {\frac {[b_{{13}}]_{{{{3}}}}}{12h}}
\pm {\frac {[a]_{{{{3}}}}[b_{{13}}]_{{{{3}}}}}{24a}}
\mp {\frac {b_{{12}}[c_{{3}}]_{{{{2}}}}}{48a}}
\mp {\frac {b_{{13}}[c_{{3}}]_{{{{3}}}}}{48a}}
\pm {\frac {[a]_{{{{1}}}}[b_{{13}}]_{{{{1}}}}}{24a}}
\mp {\frac {c_{{2}}[b_{{13}}]_{{{{2}}}}}{48a}}
+ {\frac {a}{6{h}^{2}}}
\\
&\pm {\frac {[a]_{{{{3}}}}c_{{1}}}{24a}}
\mp {\frac {b_{{13}}[c_{{1}}]_{{{{1}}}}}{48a}}
\mp {\frac {c_{{1}}[b_{{13}}]_{{{{1}}}}}{48a}}
\mp {\frac {b_{{23}}[b_{{13}}]_{{{{2}}{{3}}}}}{48a}}
\pm {\frac {[a]_{{{{1}}}}c_{{3}}}{24a}}
\mp {\frac {c_{{1}}c_{{3}}}{24a}}
\pm {\frac {[a]_{{{{2}}}}[b_{{13}}]_{{{{2}}}}}{24a}}
\\
&
\mp {\frac {c_{{3}}[b_{{13}}]_{{{{3}}}}}{48a}}
\mp {\frac {b_{{12}}[b_{{13}}]_{{{{1}}{{2}}}}}{48a}}
\pm {\frac {b_{{13}}[b_{{13}}]_{{{{1}}}}}{24ah}}
\pm {\frac {b_{{23}}[b_{{13}}]_{{{{2}}}}}{24ah}}
\mp {\frac {{b^{2}_{{13}}}[a]_{{{{1}}}}}{24{a}^{2}h}}
\pm {\frac {c_{{3}}b_{{13}}}{12ah}}
\mp {\frac {b_{{13}}[a]_{{{{3}}}}}{12ah}}
\\
&
+ {\frac {b_{{12}}[a]_{{{{2}}}}b_{{13}}}{24{a}^{2}h}}
\pm {\frac {c_{{1}}}{12h}}
+ {\frac {{b^{2}_{{13}}}}{12a{h}^{2}}}
\mp {\frac {b_{{13}}[b_{{13}}]_{{{{1}}{{2}}}}}{48a}}
\mp {\frac {b_{{23}}[c_{{1}}]_{{{{2}}}}}{48a}}
\pm {\frac {b_{{23}}b_{{12}}}{12a{h}^{2}}}
+ {\frac {[a]_{{{{3}}}}{b^{2}_{{13}}}}{24{a}^{2}h}}
\\
&
\pm {\frac {b_{{13}}[a]_{{{{1}}}}c_{{1}}}{48{a}^{2}}}
\pm {\frac {b_{{12}}[a]_{{{{1}}}}[b_{{13}}]_{{{{2}}}}}{48{a}^{2}}}
\pm {\frac {b_{{23}}[a]_{{{{2}}}}c_{{1}}}{48{a}^{2}}}
\pm {\frac {b_{{13}}[a]_{{{{1}}}}[b_{{13}}]_{{{{3}}}}}{48{a}^{2}}}
\pm {\frac {b_{{23}}[a]_{{{{3}}}}[b_{{13}}]_{{{{2}}}}}{48{a}^{2}}}
\\
&
\pm {\frac {b_{{13}}[a]_{{{{3}}}}[b_{{13}}]_{{{{1}}}}}{48{a}^{2}}}
\pm {\frac {b_{{23}}[a]_{{{{2}}}}[b_{{13}}]_{{{{3}}}}}{48{a}^{2}}}
\mp \frac{[c_{{1}}]_{{{{3}}}}}{24}
\pm {\frac {b_{{12}}[a]_{{{{2}}}}[b_{{13}}]_{{{{1}}}}}{48{a}^{2}}}
+ {\frac {b_{{13}}[a]_{{{{1}}}}}{12ah}}
- {\frac {c_{{1}}b_{{13}}}{12ah}}  
\\
&
\pm {\frac {b_{{12}}[a]_{{{{2}}}}c_{{3}}}{48{a}^{2}}}
- {\frac {b_{{13}}[b_{{13}}]_{{{{3}}}}}{24ah}}
\pm {\frac {b_{{13}}[a]_{{{{3}}}}c_{{3}}}{48{a}^{2}}}
- {\frac {b_{{12}}[b_{{13}}]_{{{{2}}}}}{24ah}},\\
\hat{K}_{i,j \pm 1,m-1}  = & 
\mp \frac{[c_{{3}}]_{{{{2}}}}}{24}
\mp {\frac {b_{{23}}}{6{h}^{2}}}
- {\frac {[b_{{23}}]_{{{{2}}}}}{12h}}
- {\frac {c_{{3}}}{12h}}
\mp {\frac {b_{{13}}[a]_{{{{1}}}}b_{{23}}}{24{a}^{2}h}}
\mp \frac{[b_{{23}}]_{{{11}}}}{48}
\mp \frac{[b_{{23}}]_{{{22}}}}{48}
\mp \frac{[b_{{23}}]_{{{33}}}}{48}
\\
&
\mp {\frac {b_{{12}}[c_{{3}}]_{{{{1}}}}}{48a}}
\pm {\frac {[a]_{{{{3}}}}[b_{{23}}]_{{{{3}}}}}{24a}}
\mp {\frac {b_{{23}}[c_{{2}}]_{{{{2}}}}}{48a}}
\mp {\frac {b_{{23}}[c_{{3}}]_{{{{3}}}}}{48a}}
\pm {\frac {[a]_{{{{2}}}}c_{{3}}}{24a}}
\pm {\frac {[a]_{{{{1}}}}[b_{{23}}]_{{{{1}}}}}{24a}}
\mp {\frac {c_{{2}}c_{{3}}}{24a}}
\\
&
\mp {\frac {c_{{3}}[b_{{23}}]_{{{{3}}}}}{48a}}
\mp {\frac {b_{{12}}[b_{{23}}]_{{{{1}}{{2}}}}}{48a}}
\mp {\frac {b_{{13}}[c_{{2}}]_{{{{1}}}}}{48a}}
\mp {\frac {b_{{13}}[b_{{23}}]_{{{{1}}{{2}}}}}{48a}}
\pm {\frac {[a]_{{{{2}}}}[b_{{23}}]_{{{{2}}}}}{24a}}
\pm {\frac {[a]_{{{{3}}}}c_{{2}}}{24a}}
\\
&
+ {\frac {b_{{12}}[a]_{{{{1}}}}b_{{23}}}{24{a}^{2}h}}
\pm {\frac {c_{{2}}}{12h}}
+ {\frac {a}{6{h}^{2}}}
\pm {\frac {b_{{23}}c_{{3}}}{12ah}}
\mp {\frac {b_{{23}}[a]_{{{{3}}}}}{12ah}}
\pm {\frac {b_{{13}}[b_{{23}}]_{{{{1}}}}}{24ah}}
\pm {\frac {b_{{23}}[b_{{23}}]_{{{{2}}}}}{24ah}}
+ {\frac {{b^{2}_{{23}}}}{12a{h}^{2}}}
\\
& 
\mp {\frac {[a]_{{{{2}}}}{b^{2}_{{23}}}}{24{a}^{2}h}}
\mp {\frac {b_{{23}}[b_{{23}}]_{{{{2}}{{3}}}}}{48a}}
\mp {\frac {c_{{2}}[b_{{23}}]_{{{{2}}}}}{48a}}
\mp {\frac {c_{{1}}[b_{{23}}]_{{{{1}}}}}{48a}}
\pm {\frac {b_{{12}}b_{{13}}}{12a{h}^{2}}}
- {\frac {c_{{2}}b_{{23}}}{12ah}}
\pm {\frac {b_{{23}}[a]_{{{{3}}}}c_{{3}}}{48{a}^{2}}}
\\
& 
\pm {\frac {b_{{23}}[a]_{{{{3}}}}[b_{{23}}]_{{{{2}}}}}{48{a}^{2}}}
+ {\frac {b_{{23}}[a]_{{{{2}}}}}{12ah}}
\pm {\frac {b_{{12}}[a]_{{{{1}}}}c_{{3}}}{48{a}^{2}}}
\pm {\frac {b_{{13}}[a]_{{{{1}}}}[b_{{23}}]_{{{{3}}}}}{48{a}^{2}}}
\pm {\frac {b_{{23}}[a]_{{{{2}}}}[b_{{23}}]_{{{{3}}}}}{48{a}^{2}}}
\\
&
\pm {\frac {b_{{23}}[a]_{{{{2}}}}c_{{2}}}{48{a}^{2}}}
\pm {\frac {b_{{12}}[a]_{{{{1}}}}[b_{{23}}]_{{{{2}}}}}{48{a}^{2}}}
\pm {\frac {b_{{13}}[a]_{{{{1}}}}c_{{2}}}{48{a}^{2}}}
\pm {\frac {b_{{13}}[a]_{{{{3}}}}[b_{{23}}]_{{{{1}}}}}{48{a}^{2}}}
\pm {\frac {b_{{12}}[a]_{{{{2}}}}[b_{{23}}]_{{{{1}}}}}{48{a}^{2}}}
\\
&
- {\frac {b_{{23}}[b_{{23}}]_{{{{3}}}}}{24ah}}
- {\frac {b_{{12}}[b_{{23}}]_{{{{1}}}}}{24ah}}
+ {\frac {{b^{2}_{{23}}}[a]_{{{{3}}}}}{24{a}^{2}h}}
\pm {\frac {[b_{{23}}]_{{{{3}}}}}{12h}}
\mp \frac{[c_{{2}}]_{{{{3}}}}}{24}  ,\\
\hat{K}_{i\pm 1,j+1,m-1}  = & 
\mp {\frac {[b_{{13}}]_{{{{2}}}}}{48h}}
\mp {\frac {b_{{13}}}{24{h}^{2}}}
\mp {\frac {[b_{{12}}]_{{{{3}}}}}{48h}}
\mp {\frac {[b_{{23}}]_{{{{1}}}}}{48h}}
\mp {\frac {b_{{12}}[b_{{13}}]_{{{{1}}}}}{96ah}}
\mp {\frac {b_{{13}}[b_{{12}}]_{{{{1}}}}}{96ah}}
\mp {\frac {b_{{23}}[b_{{13}}]_{{{{3}}}}}{96ah}}
\\
&
\mp {\frac {b_{{23}}b_{{12}}}{24a{h}^{2}}}
\pm {\frac {b_{{13}}[a]_{{{{3}}}}b_{{23}}}{48{a}^{2}h}}
\pm {\frac {b_{{12}}[a]_{{{{2}}}}b_{{23}}}{48{a}^{2}h}}
- {\frac {b_{{23}}}{24{h}^{2}}}
\mp {\frac {b_{{13}}[b_{{23}}]_{{{{3}}}}}{96ah}}
\mp {\frac {b_{{23}}c_{{1}}}{48ah}}
- {\frac {b_{{12}}b_{{13}}}{24a{h}^{2}}}
\\
&
\pm {\frac {b_{{12}}}{24{h}^{2}}}
\pm {\frac {[a]_{{{{2}}}}b_{{13}}}{48ah}}
\pm {\frac {[a]_{{{{3}}}}b_{{12}}}{48ah}}
\pm {\frac {b_{{13}}[a]_{{{{1}}}}b_{{12}}}{48{a}^{2}h}}
\pm {\frac {[a]_{{{{1}}}}b_{{23}}}{48ah}}
\mp {\frac {b_{{12}}[b_{{23}}]_{{{{2}}}}}{96ah}}
\\
&
\mp {\frac {b_{{23}}[b_{{12}}]_{{{{2}}}}}{96ah}}
\mp {\frac {c_{{3}}b_{{12}}}{48ah}}
\mp {\frac {b_{{13}}c_{{2}}}{48ah}}
\pm {\frac {b_{{13}}b_{{23}}}{24a{h}^{2}}}  ,\\
\hat{K}_{i\pm 1,j-1,m+1}  = & 
\mp {\frac {[b_{{13}}]_{{{{2}}}}}{48h}}
\pm {\frac {b_{{13}}}{24{h}^{2}}}
\mp {\frac {[b_{{12}}]_{{{{3}}}}}{48h}}
\mp {\frac {[b_{{23}}]_{{{{1}}}}}{48h}}
\mp {\frac {b_{{12}}[b_{{13}}]_{{{{1}}}}}{96ah}}
\mp {\frac {b_{{13}}[b_{{12}}]_{{{{1}}}}}{96ah}}
\mp {\frac {b_{{23}}[b_{{13}}]_{{{{3}}}}}{96ah}}
\\
&
\pm {\frac {b_{{23}}b_{{12}}}{24a{h}^{2}}}
\pm {\frac {b_{{13}}[a]_{{{{3}}}}b_{{23}}}{48{a}^{2}h}}
\pm {\frac {b_{{12}}[a]_{{{{2}}}}b_{{23}}}{48{a}^{2}h}}
- {\frac {b_{{23}}}{24{h}^{2}}}
\mp {\frac {b_{{13}}[b_{{23}}]_{{{{3}}}}}{96ah}}
\mp {\frac {b_{{23}}c_{{1}}}{48ah}}
- {\frac {b_{{12}}b_{{13}}}{24a{h}^{2}}}
\\
&
\mp {\frac {b_{{12}}}{24{h}^{2}}}
\pm {\frac {[a]_{{{{2}}}}b_{{13}}}{48ah}}
\pm {\frac {[a]_{{{{3}}}}b_{{12}}}{48ah}}
\pm {\frac {b_{{13}}[a]_{{{{1}}}}b_{{12}}}{48{a}^{2}h}}
\pm {\frac {[a]_{{{{1}}}}b_{{23}}}{48ah}}
\mp {\frac {b_{{12}}[b_{{23}}]_{{{{2}}}}}{96ah}}
\\
&
\mp {\frac {b_{{23}}[b_{{12}}]_{{{{2}}}}}{96ah}}
\mp {\frac {c_{{3}}b_{{12}}}{48ah}}
\mp {\frac {b_{{13}}c_{{2}}}{48ah}}
\mp {\frac {b_{{13}}b_{{23}}}{24a{h}^{2}}}  ,\\
\hat{K}_{i \pm 1,j,m+1}  = & 
\pm   \frac{[c_{{3}}]_{{{{1}}}}}{24}
\pm {\frac {b_{{13}}}{6{h}^{2}}}
+ {\frac {c_{{3}}}{12h}}
+ {\frac {[b_{{13}}]_{{{{1}}}}}{12h}}
-\mp{\frac {b_{{23}}[a]_{{{{2}}}}b_{{13}}}{24{a}^{2}h}}
\pm \frac{[b_{{13}}]_{{{11}}}}{48}
\pm \frac{[b_{{13}}]_{{{22}}}}{48}
\pm \frac{[b_{{13}}]_{{{33}}}}{48}
\\
&
\pm {\frac {[b_{{13}}]_{{{{3}}}}}{12h}}
\mp {\frac {[a]_{{{{3}}}}[b_{{13}}]_{{{{3}}}}}{24a}}
\pm {\frac {b_{{12}}[c_{{3}}]_{{{{2}}}}}{48a}}
\pm {\frac {b_{{13}}[c_{{3}}]_{{{{3}}}}}{48a}}
\mp {\frac {[a]_{{{{1}}}}[b_{{13}}]_{{{{1}}}}}{24a}}
\pm {\frac {c_{{2}}[b_{{13}}]_{{{{2}}}}}{48a}}
+ {\frac {a}{6{h}^{2}}}
\\
&
\mp {\frac {[a]_{{{{3}}}}c_{{1}}}{24a}}
\pm {\frac {b_{{13}}[c_{{1}}]_{{{{1}}}}}{48a}}
\pm {\frac {c_{{1}}[b_{{13}}]_{{{{1}}}}}{48a}}
\pm {\frac {b_{{23}}[b_{{13}}]_{{{{2}}{{3}}}}}{48a}}
\mp {\frac {[a]_{{{{1}}}}c_{{3}}}{24a}}
\pm {\frac {c_{{1}}c_{{3}}}{24a}}
\mp {\frac {[a]_{{{{2}}}}[b_{{13}}]_{{{{2}}}}}{24a}}
\\
&
\pm {\frac {c_{{3}}[b_{{13}}]_{{{{3}}}}}{48a}}
\pm {\frac {b_{{12}}[b_{{13}}]_{{{{1}}{{2}}}}}{48a}}
\pm {\frac {b_{{13}}[b_{{13}}]_{{{{1}}}}}{24ah}}
\pm {\frac {b_{{23}}[b_{{13}}]_{{{{2}}}}}{24ah}}
\mp {\frac {{b^{2}_{{13}}}[a]_{{{{1}}}}}{24{a}^{2}h}}
\pm {\frac {c_{{3}}b_{{13}}}{12ah}}
\mp {\frac {b_{{13}}[a]_{{{{3}}}}}{12ah}}
\\
&
- {\frac {b_{{12}}[a]_{{{{2}}}}b_{{13}}}{24{a}^{2}h}}
\pm {\frac {c_{{1}}}{12h}}
+ {\frac {{b^{2}_{{13}}}}{12a{h}^{2}}}
\pm {\frac {b_{{13}}[b_{{13}}]_{{{{1}}{{2}}}}}{48a}}
\pm {\frac {b_{{23}}[c_{{1}}]_{{{{2}}}}}{48a}}
\mp {\frac {b_{{23}}b_{{12}}}{12a{h}^{2}}}
- {\frac {[a]_{{{{3}}}}{b^{2}_{{13}}}}{24{a}^{2}h}}
\\
&
\mp {\frac {b_{{13}}[a]_{{{{1}}}}c_{{1}}}{48{a}^{2}}}
\mp {\frac {b_{{12}}[a]_{{{{1}}}}[b_{{13}}]_{{{{2}}}}}{48{a}^{2}}}
\mp {\frac {b_{{23}}[a]_{{{{2}}}}c_{{1}}}{48{a}^{2}}}
\mp {\frac {b_{{13}}[a]_{{{{1}}}}[b_{{13}}]_{{{{3}}}}}{48{a}^{2}}}
\mp {\frac {b_{{23}}[a]_{{{{3}}}}[b_{{13}}]_{{{{2}}}}}{48{a}^{2}}}
\\
&
\pm \frac{[c_{{1}}]_{{{{3}}}}}{24}
\mp {\frac {b_{{13}}[a]_{{{{3}}}}[b_{{13}}]_{{{{1}}}}}{48{a}^{2}}}
\mp {\frac {b_{{23}}[a]_{{{{2}}}}[b_{{13}}]_{{{{3}}}}}{48{a}^{2}}}
\mp {\frac {b_{{12}}[a]_{{{{2}}}}[b_{{13}}]_{{{{1}}}}}{48{a}^{2}}}
- {\frac {b_{{13}}[a]_{{{{1}}}}}{12ah}}
\\
&
\mp {\frac {b_{{12}}[a]_{{{{2}}}}c_{{3}}}{48{a}^{2}}}
+ {\frac {b_{{13}}[b_{{13}}]_{{{{3}}}}}{24ah}}
\mp {\frac {b_{{13}}[a]_{{{{3}}}}c_{{3}}}{48{a}^{2}}}
+ {\frac {b_{{12}}[b_{{13}}]_{{{{2}}}}}{24ah}}
+ {\frac {c_{{1}}b_{{13}}}{12ah}}  ,\\
\hat{K}_{i,j,m\pm 1}  = & 
- {\frac {b_{{23}}[a]_{{{{2}}}}c_{{3}}}{12{a}^{2}}}
- {\frac {b_{{13}}[a]_{{{{1}}}}c_{{3}}}{12{a}^{2}}}
+ \frac{[c_{{3}}]_{{{{3}}}}}{6}
\mp {\frac {[b_{{23}}]_{{{{2}}}}}{6h}}
\pm {\frac {c_{{3}}}{6h}}
\mp {\frac {[b_{{13}}]_{{{{1}}}}}{6h}}
\pm \frac{h[c_{{3}}]_{{{11}}}}{24}
\pm \frac{h[c_{{3}}]_{{{33}}}}{24}
\\
&
\pm \frac{h[c_{{3}}]_{{{22}}}}{24}
+ {\frac {{c^{2}_{{3}}}}{12a}}
+ \frac{[a]_{{{11}}}}{12}
+ \frac{[a]_{{{22}}}}{12}
+ \frac{[a]_{{{33}}}}{12}
\pm {\frac {b_{{12}}[a]_{{{{1}}}}b_{{23}}}{12{a}^{2}h}}
\pm {\frac {b_{{12}}[a]_{{{{2}}}}b_{{13}}}{12{a}^{2}h}}
\\
&
\mp {\frac {hb_{{13}}[a]_{{{{1}}}}[c_{{3}}]_{{{{3}}}}}{24{a}^{2}}}
\mp {\frac {hb_{{12}}[a]_{{{{1}}}}[c_{{3}}]_{{{{2}}}}}{24{a}^{2}}}
\mp {\frac {hb_{{23}}[a]_{{{{2}}}}[c_{{3}}]_{{{{3}}}}}{24{a}^{2}}}
\mp {\frac {hb_{{23}}[a]_{{{{3}}}}[c_{{3}}]_{{{{2}}}}}{24{a}^{2}}}
- {\frac {b_{{13}}[a]_{{{{3}}}}[a]_{{{{1}}}}}{6{a}^{2}}}
\\
&
\mp {\frac {hb_{{13}}[a]_{{{{3}}}}[c_{{3}}]_{{{{1}}}}}{24{a}^{2}}}
\mp {\frac {hb_{{12}}[a]_{{{{2}}}}[c_{{3}}]_{{{{1}}}}}{24{a}^{2}}}
+ {\frac {a}{3{h}^{2}}}
- {\frac {b_{{23}}[a]_{{{{3}}}}[a]_{{{{2}}}}}{6{a}^{2}}}
- {\frac {b_{{12}}[a]_{{{{1}}}}[a]_{{{{2}}}}}{6{a}^{2}}}
+ {\frac {b_{{13}}[c_{{3}}]_{{{{1}}}}}{12a}}
\\
&
+ {\frac {c_{{1}}[a]_{{{{1}}}}}{12a}}
- {\frac {{b^{2}_{{23}}}}{6a{h}^{2}}}
+ {\frac {b_{{12}}[a]_{{{{1}}{{2}}}}}{12a}}
+ {\frac {c_{{2}}[a]_{{{{2}}}}}{12a}}
- {\frac {{b^{2}_{{13}}}}{6a{h}^{2}}}
- {\frac {c_{{3}}[a]_{{{{3}}}}}{12a}}
+ {\frac {b_{{13}}[a]_{{{{1}}{{2}}}}}{12a}}
\pm {\frac {hc_{{2}}[c_{{3}}]_{{{{2}}}}}{24a}}
\\
&
+ {\frac {b_{{23}}[a]_{{{{2}}{{3}}}}}{12a}}
+ {\frac {b_{{23}}[c_{{3}}]_{{{{2}}}}}{12a}}
\mp {\frac {c_{{2}}b_{{23}}}{6ah}}
\pm {\frac {b_{{23}}[a]_{{{{2}}}}}{6ah}}
\mp {\frac {b_{{23}}[b_{{23}}]_{{{{3}}}}}{12ah}}
\mp {\frac {b_{{12}}[b_{{23}}]_{{{{1}}}}}{12ah}}
\pm {\frac {{b^{2}_{{23}}}[a]_{{{{3}}}}}{12{a}^{2}h}}
\\
&
\pm {\frac {[a]_{{{{3}}}}{b^{2}_{{13}}}}{12{a}^{2}h}}
\pm {\frac {b_{{13}}[a]_{{{{1}}}}}{6ah}}
\mp {\frac {b_{{13}}[b_{{13}}]_{{{{3}}}}}{12ah}}
\mp {\frac {b_{{12}}[b_{{13}}]_{{{{2}}}}}{12ah}}
\mp {\frac {c_{{1}}b_{{13}}}{6ah}}
\pm {\frac {hb_{{13}}[c_{{3}}]_{{{{1}}{{2}}}}}{24a}}
\pm {\frac {hc_{{1}}[c_{{3}}]_{{{{1}}}}}{24a}}
\\
&
\pm {\frac {hb_{{12}}[c_{{3}}]_{{{{1}}{{2}}}}}{24a}}
\pm {\frac {hb_{{23}}[c_{{3}}]_{{{{2}}{{3}}}}}{24a}}
\mp {\frac {h[a]_{{{{1}}}}[c_{{3}}]_{{{{1}}}}}{12a}}
\mp {\frac {h[a]_{{{{2}}}}[c_{{3}}]_{{{{2}}}}}{12a}}
\mp {\frac {h[a]_{{{{3}}}}[c_{{3}}]_{{{{3}}}}}{12a}}
\\
&
\pm {\frac {hc_{{3}}[c_{{3}}]_{{{{3}}}}}{24a}}
- {\frac {{[a]^{2}_{{{{1}}}}}}{6a}}
- {\frac {{[a]^{2}_{{{{3}}}}}}{6a}}
- {\frac {{[a]^{2}_{{{{2}}}}}}{6a}}  ,\\
\hat{K}_{i,j\pm 1,m+1}  = & 
\pm   \frac{[c_{{3}}]_{{{{2}}}}}{24}
\pm {\frac {b_{{23}}}{6{h}^{2}}}
+ {\frac {[b_{{23}}]_{{{{2}}}}}{12h}}
+ {\frac {c_{{3}}}{12h}}
\mp {\frac {b_{{13}}[a]_{{{{1}}}}b_{{23}}}{24{a}^{2}h}}
\pm \frac{[b_{{23}}]_{{{11}}}}{48}
\pm \frac{[b_{{23}}]_{{{22}}}}{48}
\pm \frac{[b_{{23}}]_{{{33}}}}{48}
\\
&
\pm {\frac {b_{{12}}[c_{{3}}]_{{{{1}}}}}{48a}}
\mp {\frac {[a]_{{{{3}}}}[b_{{23}}]_{{{{3}}}}}{24a}}
\pm {\frac {b_{{23}}[c_{{2}}]_{{{{2}}}}}{48a}}
\pm {\frac {b_{{23}}[c_{{3}}]_{{{{3}}}}}{48a}}
\mp {\frac {[a]_{{{{2}}}}c_{{3}}}{24a}}
\mp {\frac {[a]_{{{{1}}}}[b_{{23}}]_{{{{1}}}}}{24a}}
\pm {\frac {c_{{2}}c_{{3}}}{24a}}
\\
&
\pm {\frac {c_{{3}}[b_{{23}}]_{{{{3}}}}}{48a}}
\pm {\frac {b_{{12}}[b_{{23}}]_{{{{1}}{{2}}}}}{48a}}
\pm {\frac {b_{{13}}[c_{{2}}]_{{{{1}}}}}{48a}}
\pm {\frac {b_{{13}}[b_{{23}}]_{{{{1}}{{2}}}}}{48a}}
\mp {\frac {[a]_{{{{2}}}}[b_{{23}}]_{{{{2}}}}}{24a}}
\mp {\frac {[a]_{{{{3}}}}c_{{2}}}{24a}}
\\
&
- {\frac {b_{{12}}[a]_{{{{1}}}}b_{{23}}}{24{a}^{2}h}}
\pm {\frac {c_{{2}}}{12h}}
+ {\frac {a}{6{h}^{2}}}
\pm {\frac {b_{{23}}c_{{3}}}{12ah}}
\mp {\frac {b_{{23}}[a]_{{{{3}}}}}{12ah}}
\pm {\frac {b_{{13}}[b_{{23}}]_{{{{1}}}}}{24ah}}
\pm {\frac {b_{{23}}[b_{{23}}]_{{{{2}}}}}{24ah}}
+ {\frac {c_{{2}}b_{{23}}}{12ah}}
\\
&
\mp {\frac {[a]_{{{{2}}}}{b^{2}_{{23}}}}{24{a}^{2}h}}
+ {\frac {{b^{2}_{{23}}}}{12a{h}^{2}}}
\pm {\frac {b_{{23}}[b_{{23}}]_{{{{2}}{{3}}}}}{48a}}
\pm {\frac {c_{{2}}[b_{{23}}]_{{{{2}}}}}{48a}}
\pm {\frac {c_{{1}}[b_{{23}}]_{{{{1}}}}}{48a}}
\mp {\frac {b_{{12}}b_{{13}}}{12a{h}^{2}}}
- {\frac {b_{{23}}[a]_{{{{2}}}}}{12ah}}
\\
&
\mp {\frac {b_{{23}}[a]_{{{{3}}}}[b_{{23}}]_{{{{2}}}}}{48{a}^{2}}}
\mp {\frac {b_{{23}}[a]_{{{{3}}}}c_{{3}}}{48{a}^{2}}}
\mp {\frac {b_{{12}}[a]_{{{{1}}}}c_{{3}}}{48{a}^{2}}}
\mp {\frac {b_{{13}}[a]_{{{{1}}}}[b_{{23}}]_{{{{3}}}}}{48{a}^{2}}}
\mp {\frac {b_{{23}}[a]_{{{{2}}}}[b_{{23}}]_{{{{3}}}}}{48{a}^{2}}}
\\
&
\mp {\frac {b_{{23}}[a]_{{{{2}}}}c_{{2}}}{48{a}^{2}}}
\mp {\frac {b_{{12}}[a]_{{{{1}}}}[b_{{23}}]_{{{{2}}}}}{48{a}^{2}}}
\mp {\frac {b_{{13}}[a]_{{{{1}}}}c_{{2}}}{48{a}^{2}}}
\mp {\frac {b_{{13}}[a]_{{{{3}}}}[b_{{23}}]_{{{{1}}}}}{48{a}^{2}}}
\mp {\frac {b_{{12}}[a]_{{{{2}}}}[b_{{23}}]_{{{{1}}}}}{48{a}^{2}}}
\\
&
+ {\frac {b_{{23}}[b_{{23}}]_{{{{3}}}}}{24ah}}
+ {\frac {b_{{12}}[b_{{23}}]_{{{{1}}}}}{24ah}}
- {\frac {{b^{2}_{{23}}}[a]_{{{{3}}}}}{24{a}^{2}h}}
\pm {\frac {[b_{{23}}]_{{{{3}}}}}{12h}}
\pm \frac{[c_{{2}}]_{{{{3}}}}}{24}.
\end{align*}
\end{small}
Note that in the above $a, b_{12}, b_{13}, b_{23}, c_1, c_2, c_3$ and $g$ are evaluated at
$\bigl(x_{i_1}^{(1)},x_{i_2}^{(2)},x_{i_3}^{(3)} \bigr)\in
\interior{G}_h^{(3)}$ and $\tau\in \Omega_{\tau}$.
To streamline the notation we used $[\cdot ]_{k}$ and $[
\cdot ]_{kp}$ to denote the first and second derivative of the coefficients with
respect to $x_k$, and with respect to $x_k$ and $x_p$, respectively.

%% file: Project_2_Masterfile.bbl
\begin{thebibliography}{10}

\bibitem{MR2285863}
{\sc G.~Berikelashvili, M.M. Gupta, and M.~Mirianashvili}, {\em Convergence of
  fourth order compact difference schemes for three-dimensional
  convection-diffusion equations}, SIAM J. Numer. Anal., 45 (2007),
  pp.~443--455.

\bibitem{DuFo12}
{\sc B.~D{\"u}ring and M.~Fourni{\'e}}, {\em High-order compact finite
  difference scheme for option pricing in stochastic volatility models}, J.
  Comput. Appl. Math., 236 (2012), pp.~4462--4473.

\bibitem{DuFo12p}
\leavevmode\vrule height 2pt depth -1.6pt width 23pt, {\em On the stability of
  a compact finite difference scheme for option pricing}, in Progress in
  Industrial Mathematics at ECMI 2010, M.~G{\"u}nther and et~al., eds., Berlin,
  Heidelberg, 2012, Springer, pp.~215--221.

\bibitem{DuFoHe14}
{\sc B.~D{\"u}ring, M.~Fourni{\'e}, and C.~Heuer}, {\em High-order compact
  finite difference schemes for option pricing in stochastic volatility models
  on non-uniform grids}, J. Comput. Appl. Math., 271 (2014), pp.~247--266.

\bibitem{DuFoJu03}
{\sc B.~D{\"u}ring, M.~Fourni{\'e}, and A.~J\"ungel}, {\em High-order compact
  finite difference schemes for a nonlinear {B}lack-{S}choles equation},
  Intern.~J.~Theor.~Appl.~Finance, 6 (2003), pp.~767--789.

\bibitem{DuFoJu04}
\leavevmode\vrule height 2pt depth -1.6pt width 23pt, {\em Convergence of a
  high-order compact finite difference scheme for a nonlinear {B}lack-{S}choles
  equation}, Math.~Mod.~Num.~Anal., 38 (2004), pp.~359--369.

\bibitem{FoKa06}
{\sc M.~{Fourni\'e} and S.~{Karaa}}, {\em {Iterative methods and high-order
  difference schemes for 2D elliptic problems with mixed derivative.}}, {J.
  Appl. Math. Comput.}, 22 (2006), pp.~349--363.

\bibitem{FouRig11}
{\sc M.~Fourni{\'e} and A.~Rigal}, {\em High order compact schemes in
  projection methods for incompressible viscous flows}, Commun. Comput. Phys.,
  9 (2011), pp.~994--1019.

\bibitem{GuMaSt84}
{\sc M.M. {Gupta}, R.P. {Manohar}, and J.W. {Stephenson}}, {\em {A single cell
  high order scheme for the convection-diffusion equation with variable
  coefficients.}}, {Int. J. Numer. Methods Fluids}, 4 (1984), pp.~641--651.

\bibitem{GuMaSt85}
\leavevmode\vrule height 2pt depth -1.6pt width 23pt, {\em {High-order
  difference schemes for two-dimensional elliptic equations.}}, {Numer. Methods
  Partial Differ. Equations}, 1 (1985), pp.~71--80.

\bibitem{GuKrOl96}
{\sc B.~Gustafsson, H.-O. Kreiss, and J.~Oliger}, {\em Time Dependent Problems
  and Difference Methods}, John Wiley \& Sons, New York, 2013.

\bibitem{KarZha02}
{\sc S.~{Karaa} and J.~{Zhang}}, {\em {Convergence and performance of iterative
  methods for solving variable coefficient convection-diffusion equation with a
  fourth-order compact difference scheme.}}, {Comput. Math. Appl.}, 44 (2002),
  pp.~457--479.

\bibitem{KrThWi70}
{\sc H.O. {Kreiss}, V.~{Thomee}, and O.~{Widlund}}, {\em {Smoothing of initial
  data and rates of convergence for parabolic difference equations.}},
  {Commun.~Pure~Appl.~Math.}, 23 (1970), pp.~241--259.

\bibitem{Lele92}
{\sc S.K. {Lele}}, {\em {Compact finite difference schemes with spectral-like
  resolution.}}, {J. Comput. Phys.}, 103 (1992), pp.~16--42.

\bibitem{LiTa01}
{\sc M.~{Li} and T.~{Tang}}, {\em {A compact fourth-order finite difference
  scheme for unsteady viscous incompressible flows.}}, {J. Sci. Comput.}, 16
  (2001), pp.~29--45.

\bibitem{LiTaFo95}
{\sc M.~{Li}, T.~{Tang}, and B.~{Fornberg}}, {\em {A compact fourth-order
  finite difference scheme for the steady incompressible Navier-Stokes
  equations.}}, {Int. J. Numer. Methods Fluids}, 20 (1995), pp.~1137--1151.

\bibitem{RiMo67}
{\sc R.D. Richtmayer and K.W. Morton}, {\em Difference Methods for Initial
  Value Problems}, Interscience, New York, 1967.

\bibitem{SpoCar95}
{\sc W.F. {Spotz} and G.F. {Carey}}, {\em {High-order compact scheme for the
  steady stream-function vorticity equations.}}, {Int. J. Numer. Methods Eng.},
  38 (1995), pp.~3497--3512.

\bibitem{SpoCar96}
\leavevmode\vrule height 2pt depth -1.6pt width 23pt, {\em {A high-order
  compact formulation for the 3D Poisson equation.}}, {Numer. Methods Partial
  Differ. Equations}, 12 (1996), pp.~235--243.

\bibitem{SpoCar01}
\leavevmode\vrule height 2pt depth -1.6pt width 23pt, {\em {Extension of
  high-order compact schemes to time-dependent problems.}}, {Numer. Methods
  Partial Differ. Equations}, 17 (2001), pp.~657--672.

\bibitem{Strick04}
{\sc J.C. Strickwerda}, {\em Finite Difference Schemes and Partial Differential
  Equations}, SIAM, Philadelphia, 2004.

\bibitem{TaGoBh08}
{\sc D.Y. Tangman, A.~Gopaul, and M.~Bhuruth}, {\em Numerical pricing of
  options using high-order compact finite difference schemes},
  J.~Comp.~Appl.~Math., 218 (2008), pp.~270--280.

\bibitem{Wil98}
{\sc P.~Wilmott}, {\em Derivatives. The theory and practice of financial
  engineering}, John Wiley \& Sons Ltd., Chichester, UK, 1998.

\end{thebibliography}
